\documentclass{amsart}
\usepackage{amscd,amsmath,amssymb,dsfont,enumerate,graphicx,psfrag}

\newtheorem{thm}{Theorem}
\newtheorem{cor}[thm]{Corollary}
\newtheorem{lem}[thm]{Lemma}

\newtheorem{prop}[thm]{Proposition}
\theoremstyle{definition}
\newtheorem{defn}{Definition}
\theoremstyle{remark}
\newtheorem*{rem}{Remark}
\newtheorem*{example}{Example}
\numberwithin{equation}{section}

\newcommand{\set}[1]{\left\{#1\right\}}
\newcommand{\Too}{\longrightarrow}

\newcommand{\M}{S}
\newcommand{\Or}{\mathfrak{O}}
\newcommand{\K}{\mathcal{B}}
\newcommand{\fs}{\footnotesize}

\newcommand{\sds}{separated dihedral surgery\ }
\newcommand{\covcyl}{\widetilde{\Sigma\!\times\![0,1]}}
\newcommand{\Jh}{\mathfrak{X}}

\newcommand{\genusoneknot}[6]{\left(\rule{0pt}{18pt}\begin{pmatrix} #1 & #2\\ #3 & #4 \end{pmatrix}\raisebox{-5.5pt}{\huge{,}}\normalsize \begin{pmatrix} #5\\ #6 \end{pmatrix}\right)}
\newcommand{\pZ}{\mathds{Z}_n}
\newcommand{\Dn}{D_{2n}}
\newcommand{\Cn}{\mathcal{C}_n}
\newcommand{\seq}{\Longleftrightarrow}
\def\ass{\mathrel{\mathop:}=}
\def\co{\colon\thinspace}
\allowdisplaybreaks
\begin{document}

\title[Dihedral covering links]{Surgery presentations of \\ coloured knots and of their covering links}
\author{Andrew Kricker}%
\address{Division of Mathematical Sciences, School of Mathematical and Physical Sciences, Nanyang Technological University, Singapore, 637616}%
\email{ajkricker@ntu.edu.sg}%
\urladdr{http://www.ntu.edu.sg/home/AJKricker/}

\author{Daniel Moskovich}%
\address{Osaka City University Advanced Mathematical Institute,
3-3-138 Sugimoto-cho, Sumiyoshi-ku
Osaka 558-8585 JAPAN}%
\email{dmoskovich@gmail.com}%
\urladdr{http://www.sumamathematica.com/}

\thanks{The authors would like to thank Tomotada Ohtsuki and Dror Bar-Natan for their support.}%
\date{13th of May, 2008}%
\begin{abstract}
We consider knots equipped with a representation of their knot
groups onto a dihedral group $\Dn$ (where $n$ is odd). To each such
knot there corresponds a closed $3$--manifold, the (irregular)
dihedral branched covering space, with the branching set over the
knot forming a link in it. We report a variety of results relating
to the problem of passing from the initial data of a $\Dn$-coloured
knot to a surgery presentation of the corresponding branched
covering space and covering link. In particular, we describe
effective algorithms for constructing such presentations. A
by-product of these investigations is a proof of the conjecture that
two $\Dn$-coloured knots are related by a sequence of surgeries
along $\pm1$--framed unknots in the kernel of the representation if
and only if they have the same coloured untying invariant 
(a $\mathds{Z}_n$-valued algebraic invariant of $\Dn$-coloured
knots).
\end{abstract}
\maketitle
\section{Introduction}
The starting point for this work was the authors' desire to explore
the quantum topology of covering spaces as a means of acquiring a
deeper understanding of how quantum invariants actually encode
topological information. Recent results in the case of cyclic
covering spaces (see \textit{e.g.} \cite{GK03b,GK04}) suggest the
existence of such a theory.\par

Having understood the cyclic case, the natural next step is to
consider the branched dihedral covering spaces. These spaces have
long played an important role in knot theory, dating back to
Reidemeister's use of the linking matrix of a knot's dihedral
covering link to distinguish knots with the same Alexander
polynomial (\cite{Rei29}, see also \textit{e.g.} \cite{Per74}). More
recently they have also been used in investigations of knot
concordance (\textit{e.g.} \cite{Gil93}). In addition, branched
dihedral covers are useful in $3$--manifold topology: for example,
it turns out that every $3$--manifold is a $3$--fold branched
dihedral covering space over some knot (see \textit{e.g.}
\cite[Theorem 11.11]{BZ03}).\par

Quantum invariants for $3$--manifolds are typically constructed
using surgery presentations. To investigate the quantum topology of
covering spaces, then, it seems we need a combinatorial theory of
surgery presentations of covering spaces.\par

The cyclic case is well-known. Recall that there is a famous trick
for obtaining surgery presentations of $n$--fold cyclic covers for
any natural number $n$ (see \textit{e.g.} \cite[Chapter 6D]{Rol90}).
We wish to generalize this trick to dihedral covers, so we'll begin
by reviewing how it goes.\par

One first performs crossing changes to untie the knot by introducing
$\pm1$--framed unknots along which surgery is carried out. The
unknots are chosen to have linking number zero with the knot. After
this step, we have a surgery presentation of the given knot as a
$\pm1$--framed link $L$ lying in the complement of an unknot $U$,
where each component of $L$ has linking number zero modulo $n$ with
$U$. For the purpose of generalization, this last condition can be
restated: every component of $L$ lies in the kernel of the mod $n$
linking homomorphism $\mathrm{Link}_n\co
H_1\left(\overline{S^3-N(U)}\right)\twoheadrightarrow \mathds{Z}_n$.
Because this condition is satisfied the construction of a cyclic
cover can now be completed by lifting $L$ to the $n$--fold cyclic
cover of $S^3$ branched over $U$, which is of course again
$S^3$.\par

We would like analogous procedures for classes of covering spaces
corresponding to other groups, in particular to the dihedral groups.
The key feature which permitted construction in the cyclic case was
the existence of a knot (the unknot) which every other knot could be
transformed into via surgeries in the kernel of the mod $n$ linking
homomorphism, and whose branched cyclic cover could be constructed
explicitly.\par

To discuss how this generalizes it's worth introducing a few
definitions.

\begin{defn}[$G$--coloured knots]
For a finite group $G$ and a closed orientable $3$--manifold $M$,
define a \emph{$G$--coloured knot in $M$} to be a pair $(K,\rho)$ of
an oriented knot $K\subset M$ and a surjective representation
$\rho\co \pi_1\left(\overline{M-N(K)}\right)\twoheadrightarrow G$.
Unless otherwise specified it will be assumed that $M=S^3$.
\end{defn}

\begin{defn}[Surgery in $\ker\rho$]
Let $(K,\rho)$ be a $G$--coloured knot in a $3$--manifold $M$. If
$L\subset M-K$ is an integer--framed link each of whose components
is specified by a curve lying in $\ker\rho$ then we can perform
surgery along $L$ to obtain a new $G$--coloured knot
$(K^\prime,\rho^\prime)$ in a $3$--manifold $M^\prime$, as follows:
\begin{itemize}
\item  Remove tubular neighbourhoods
$N(L_i)$ of the components $L_i$ of $L$, and reattach them to
$\overline{M-\bigcup
N(L_i)}$ so as to match the meridional discs to the framing curves.
\item To specify the induced representation $\rho^\prime$, we must state the value it takes for an arbitrary curve $\gamma$ in $M^\prime-K^\prime$.
Homotope $\gamma$ into $M-\bigcup N(L_i)$, then evaluate it in the
restriction of $\rho$. This value is well-determined because the
components of $L$ lie in the kernel of $\rho$.
\end{itemize}
Such surgery is called \emph{surgery in $\ker\rho$}.
\end{defn}

\begin{defn}[Complete set of base-knots]
A \emph{complete set of base-knots}\footnote{The term base-knot
imitates base-point.} for a group $G$ is a set $\Psi$ of
$G$--coloured knots $(K_i,\rho_i)$ in $3$--manifolds $M_i$, such
that any $G$--coloured knot $(K,\rho)$ in $S^3$ can be obtained from
some $(K_i,\rho_i)\in \Psi$ by surgery in $\ker\rho_i$.
\end{defn}

To generalize the procedure from the cyclic case to some other group
$G$, we must find a complete set of base-knots whose desired
covering spaces (and covering links) we know how to construct
explicitly, and into whose covering spaces we know how to lift
surgery presentations for any $G$--coloured knot.\par

This paper deals with the case when $G$ is the dihedral group
$D_{2n}$ with $n$ any odd integer--- the group of permutations of
the vertices of a regular polygon with $n$ sides. Its presentation
is
\[
\Dn\ass\ \left\{t,s\left|\rule{0pt}{9.5pt}\ t^{2}=s^{n}=1,\
tst=s^{-1}\right.\right\}.
\]
\noindent As permutations on the set of vertices of the regular
polygon, these generators correspond to
\[ t = \left(
\begin{array}{rrcccr}
1 &  2 & 3 & \ldots & n-1 & n \\
1 & n  & n-1 & \ldots & 3 & 2
\end{array}
\right)
\]
\noindent and
\[
s = \left(
\begin{array}{rrrccr}
1 &  2 & 3 & \ldots & n-1 & n \\
2 & 3  & 4 & \ldots & n & 1
\end{array}
\right).
\]

\noindent Elements in $\Dn$ of the form $s^a$ are called
\emph{rotations}, and elements of the form $ts^a$ are called
\emph{reflections}. The cyclic group of rotations $\Cn\ass \langle
s\rangle$ is a normal subgroup in $\Dn$.

We'll present a $\Dn$--colouring $\rho$ of a knot $K\subset S^3$ by
labeling every arc of a knot diagram for $K$ by the image under
$\rho$ of the corresponding Wirtinger generator\footnote{\thinspace
Note that this is a different convention from the one normally used
for a Fox $n$--colouring of a knot, in which an arc which we would
label by $ts^a$ is labeled simply by $a$ (see \cite{Fox62}).}. More
generally, we can present a $\Dn$-colouring of a knot $K$ in a
closed $3$--manifold $M$ by a diagram of a link $L\cup K_1$ in $S^3$
where:
\begin{itemize}
\item $L$ is integer--framed, and surgery along $L$ turns
$(S^3,K_1)$ into $(M,K)$.
\item Every arc of the diagram is labeled by an element of $\Dn$.
\item Wirtinger relations are satisfied.
\item When the framing curve of any component of $L$ is expressed as
a product of Wirtinger generators of
$\pi_1\left(\overline{S^3-(L\cup K_1)}\right)$, the product of the
corresponding labels is $1\in\Dn$.
\end{itemize}

Our goal in this paper is to give a combinatorial procedure for
constructing surgery presentations of the irregular dihedral
branched covering space corresponding to some $\Dn$-coloured knot,
together with the covering link it contains. In Section
\ref{S:definitions} we'll recall exactly what these phrases refer
to.\par

Roughly speaking, we'll describe two approaches to this problem,
corresponding to two different complete sets of base-knots for
$\Dn$. The sets of base-knots will be introduced shortly. The
construction of their corresponding dihedral covering spaces,
covering links, and how to lift surgery presentations in the
complement of the base-knot, will be discussed in detail in
Sections \ref{S:untyingapproach} and \ref{S:Seifertapproach}.\par

\subsection*{The untying approach}

This first approach begins with exactly the same procedure for
untying knots as is used when constructing surgery presentations of
cyclic covers. It may be viewed as an adaptation of that approach to
the case of $\Dn$.

\begin{thm}\label{T:UnKnotBaseKnot}
Consider the following diagram, which depicts a $\Dn$-coloured
unknot in the $3$--manifold that results from surgery on the
disjoint $kn$--framed unknot (recall that this is the
$(kn,1)$--lens-space).
$$
\psfrag{s}[c]{$s$}\psfrag{t}[c]{$t$}\psfrag{k}[c]{$U$}\psfrag{f}[r]{framing\thinspace$=kn$}
\includegraphics[width=3in]{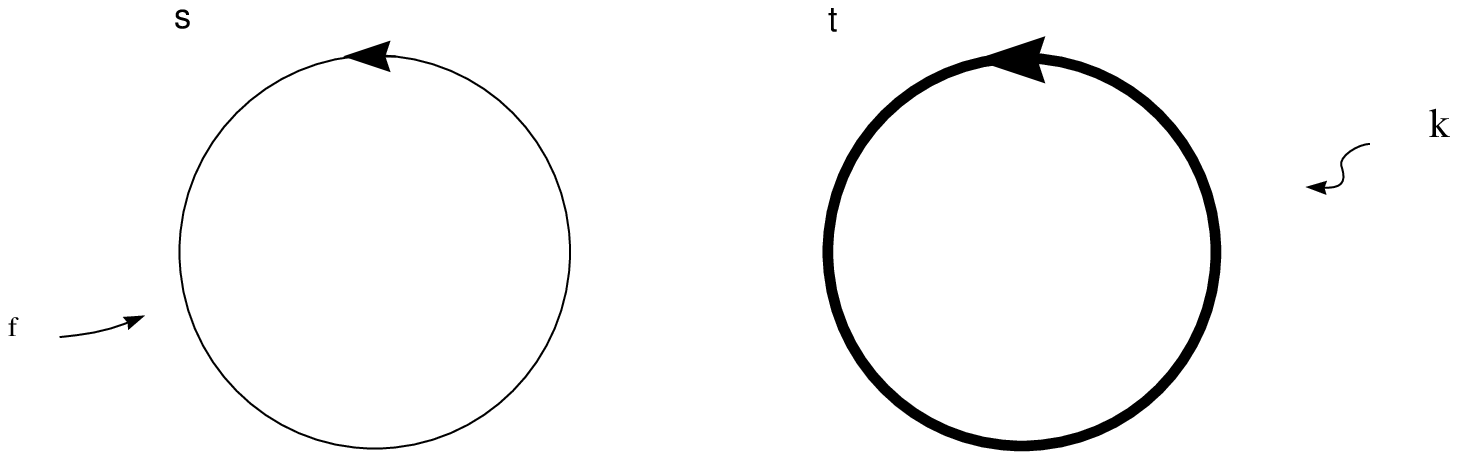}
$$
The set of these $\Dn$-coloured knots for $k=0,1,\ldots,n-1$ is a
complete set of base-knots for $\Dn$.
\end{thm}

It follows from the previous theorem that every $\Dn$-coloured knot
$(K,\rho)$ has a presentation of the following form, for some $0\leq
k<n$:
$$
\psfrag{s}[c]{$s$}\psfrag{t}[c]{$t$}\psfrag{k}[c]{$U$}\psfrag{f}[r]{framing\thinspace$=kn$}\psfrag{L}[c]{\large$T$}
\includegraphics[width=3.2in]{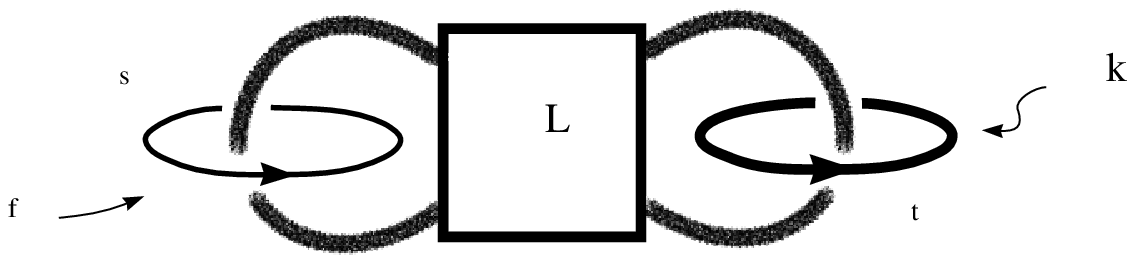}
$$

\noindent Here, the thick lines
``\raisebox{1.5pt}{\includegraphics[width=30pt]{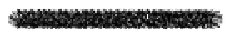}}''
denote parallel strands, $T$ is some tangle, and every component of
the framed link which results lies in $\ker\rho$ and has linking
number zero with $U$.\par

As an example, see the surgery presentation for a $D_{14}$-coloured
$5_2$ knot given in Figure \ref{F:firstfirst}.

\begin{figure}
\begin{minipage}{1.5in}
\psfrag{a}[c]{\small$t$}\psfrag{b}[c]{\small$ts^5$}\psfrag{c}[c]{\small$ts^3$}\psfrag{d}[c]{\small$ts$}\psfrag{e}[c]{\small$ts^2$}
\includegraphics[width=1.5in]{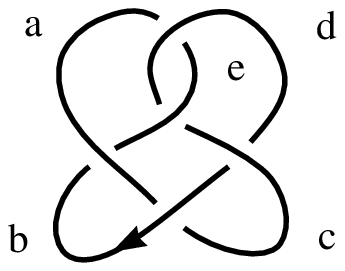}
\end{minipage}
\begin{minipage}{0.5in}
\centering
\includegraphics[width=25pt]{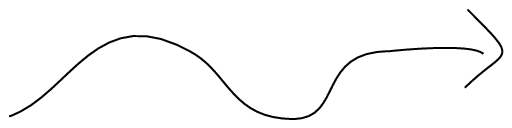}
\end{minipage}\qquad
\begin{minipage}{2.2in}
\psfrag{S}[c]{$s$}\psfrag{t}[c]{$t$}\psfrag{O}[c]{$U$}\psfrag{f}[c]{\tiny
fr\thinspace$=-7$}
\includegraphics[width=2.2in]{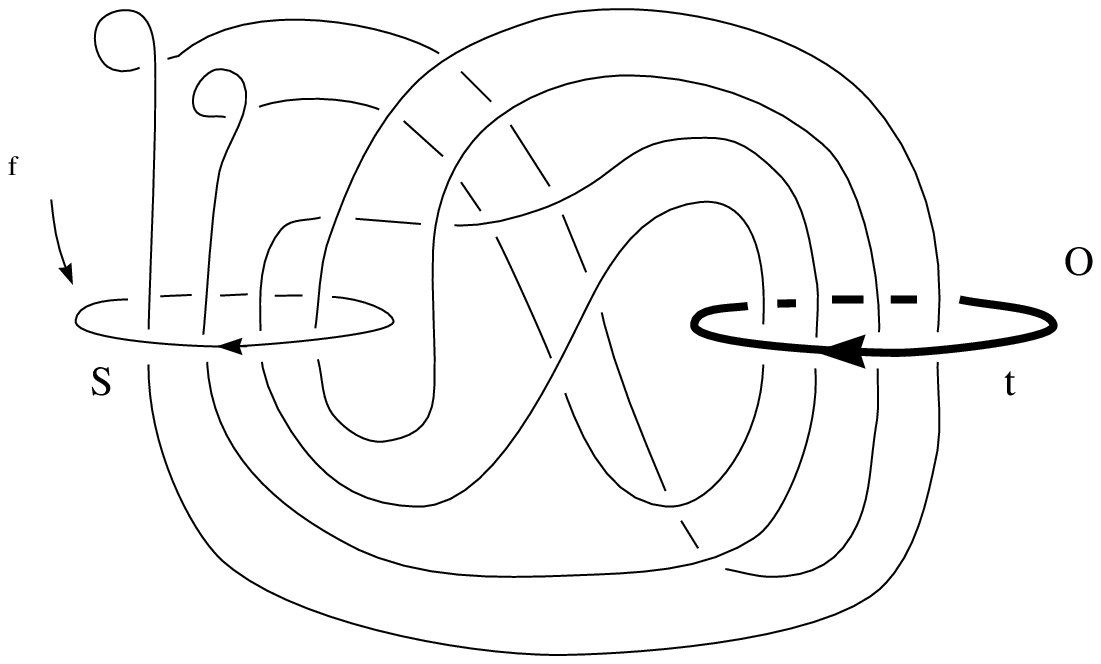}
\end{minipage}
\caption{\label{F:firstfirst}A surgery presentation for a
$D_{14}$-coloured $5_2$ knot.}
\end{figure}

\begin{thm}\label{T:untyingmethod}
A surgery presentation for the irregular dihedral branched covering
space $M$ determined by the $\Dn$-coloured knot $(K,\rho)$, and for
the covering link $\tilde{K}$ of $K$ sitting inside $M$ is as shown
in Figure \ref{F:thelink}. In that figure, a small zero near an
introduced surgery component means it has zero framing and
$\tilde{U}_1\cup\cdots\cup \tilde{U}_{\frac{n+1}{2}}$ is the
covering link of $U$, becoming $\tilde{K}$ after the surgeries are
performed.
\end{thm}

\begin{figure}
\psfrag{1}[l]{\fs$\tilde{U}_1$}\psfrag{2}[c]{\fs$\tilde{U}_2$}\psfrag{3}[c]{\fs$\tilde{U}_3$}\psfrag{4}[l]{\fs$\tilde{U}_{\frac{n+1}{2}}$}
\psfrag{M}[c]{\rotatebox{90}{$T$}}
\psfrag{a}[c]{\fs$0$}\psfrag{b}[c]{\fs$0$}\psfrag{c}[c]{\fs$0$}\psfrag{f}[c]{framing\thinspace$=k$}
\psfrag{L}[c]{\rotatebox{270}{\large$T$}}\psfrag{C}[c]{\Huge$\cdots$}
\includegraphics[width=4.2in]{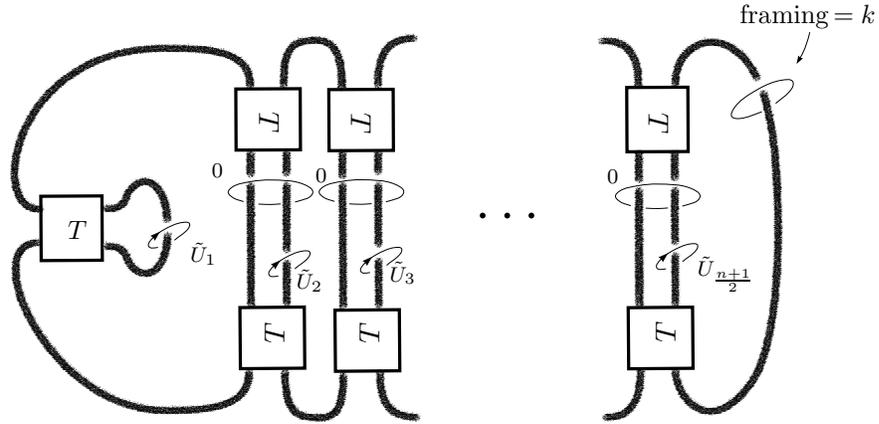}
\caption{\label{F:thelink}The surgery presentation for the covering
space in the untying approach.}
\end{figure}

\subsection*{Band projection approach}

This approach is based on a choice of band projection for a Seifert
surface $F$ of the $\Dn$-coloured knot. To a basis for $H_1(F)$
there corresponds a Seifert matrix and a \emph{colouring vector} (to
be defined in Section \ref{SSS:ColouringVector}). The colouring
vector determines the $\Dn$-colouring of the knot. The heart of this
approach will be realizing algebraic operations on the Seifert
matrix and colouring vector by sliding bands and performing
$\pm1$--framed surgeries on unknots in $\ker\rho$. While this method
seems to be less efficient in practice, it is a stronger theoretical
result because it arises from an equivalence relation on
$\Dn$-coloured knots in $S^3$ whose corresponding equivalence
classes can be detected with a certain algebraic invariant: the
coloured untying invariant.

\begin{defn}
We say that two $\Dn$-coloured knots $(K_1,\rho_1)$ and
$(K_2,\rho_2)$ in $S^3$ are $\rho$--equivalent if one can be
obtained from the other by a sequence of surgeries on $\pm1$--framed
unknots in $\ker\rho$.
\end{defn}

We alert the reader that we are restricting to surgeries along
$\pm1$--framed unknots, so that this is an equivalence relation on
$\Dn$-coloured knots in $S^3$.

This equivalence can be defined as an equivalence relation on
coloured knot diagrams without reference to surgery in the following
way:

$$
\begin{minipage}{60pt}
\psfrag{l}[c]{\Huge$\cdots$}\psfrag{a}[c]{$g_1$}\psfrag{b}[c]{$g_2$}\psfrag{c}[c]{$g_r$}
\includegraphics[height=75pt]{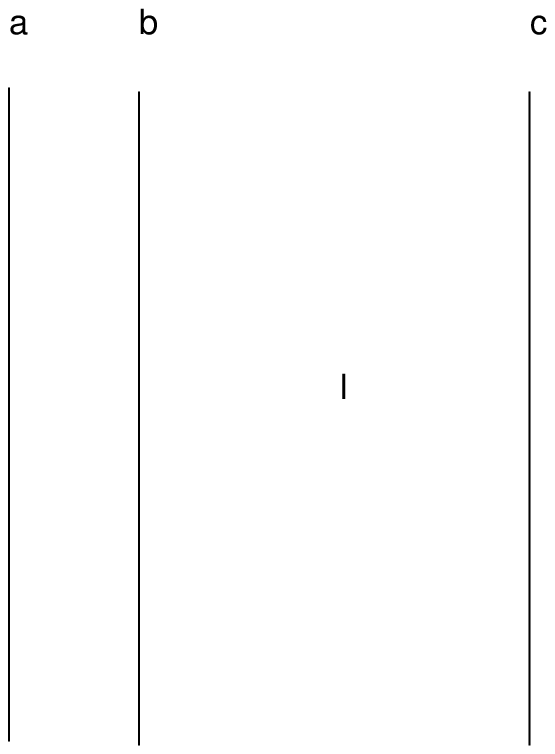}
\end{minipage}\quad\Longleftrightarrow\quad
\begin{minipage}{100pt}
\psfrag{l}[c]{\Huge$\cdots$}\psfrag{r}[c]{\Huge$\cdots$}\psfrag{a}[c]{$g_1$}\psfrag{b}[c]{$g_2$}\psfrag{c}[c]{$g_r$}
\psfrag{p}[c]{$2\pi$ twist}
\includegraphics[height=80pt]{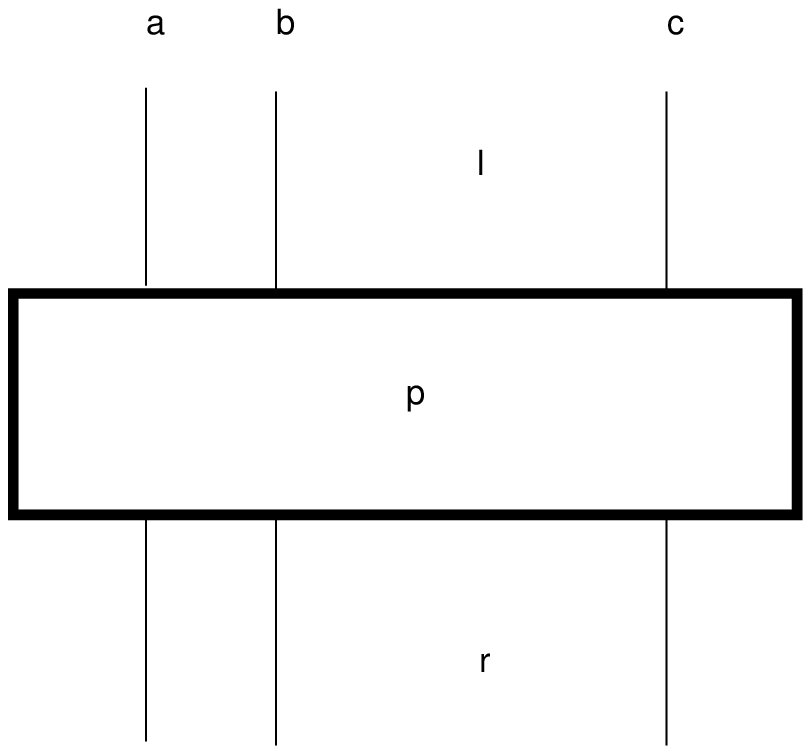}
\end{minipage}
\qquad\ \mbox{with $\prod_{i=1}^r g_i=1\in\Dn$.}
$$

\begin{thm}\label{T:PretzBaseKnot}
Any $\Dn$-coloured knot $(K,\rho)$ is $\rho$--equivalent to one of
the $\Dn$-coloured knots of Figure \ref{F:PretzBaseKnot} for
$k=0,1,\ldots,n-1$. This implies that this set of knots (the pretzel
knots $p\left(\rule{0pt}{8pt}2kn+1,1,-n\right)$ for
$k=0,1,\ldots,n-1$ with the specified colouring) is a complete set
of base-knots for $\Dn$.
\end{thm}

\begin{figure}
\begin{minipage}{2in}
\psfrag{l}[c]{\Huge$\vdots$}\psfrag{r}[c]{\Huge$\vdots$}\psfrag{k}[r]{\parbox{0.7in}{$2kn+1$\\[0.1cm]
half--twists}\ \ $\left\{\rule{0pt}{0.6in}\right.$}\psfrag{1}[l]{$\left.\rule{0pt}{0.6in}\right\}$\ \ \parbox{0.7in}{\ \quad$-n$\\[0.1cm]
half--twists}}
\psfrag{a}[c]{$ts$}\psfrag{b}[c]{$t$}\psfrag{c}[c]{$t$}\psfrag{d}[c]{$ts$}
\includegraphics[width=2in]{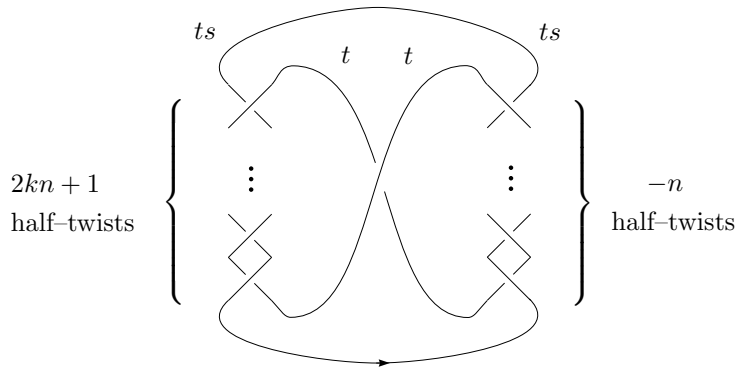}
\end{minipage}\rule{0pt}{1.1in}
\caption{\label{F:PretzBaseKnot}The base-knots in the band
projection approach.}
\end{figure}

Because too much extra notation would need to be introduced at this
point in order to explain it, an explicit construction of the
irregular branched dihedral covering spaces corresponding to the set
of base-knots of Figure \ref{T:PretzBaseKnot} is pushed off to
Section \ref{SS:SeifertLift}.\par

Having just defined a new equivalence relation, several questions
immediately arise. How many equivalence classes are there? Can they
be detected with algebraic information?\par

What we would really like is a theorem characterizing these classes
in terms of a readily computable algebraic invariant. There are many
prototypes for this in the recent literature. One example is the
result of Murakami--Nakanishi \cite{MN89}, which is closely related
to results of Matveev \cite{Mat87}, which characterizes
$\Delta$--equivalence classes of links in terms of their linking
matrices. Another is the result of Habiro \cite{Hab00} classifying
knots, all of whose finite-type invariants up to a certain degree
are equal, via surgery along tree claspers. Yet another is the work
of Naik--Stanford \cite{NS03} which links $S$--equivalence classes
of knots to double-delta moves. The influence of this point of view
on recent research should be clear.\par

In \cite[Section 6]{Mos06b} a non-trivial function from
$\Dn$--coloured knots to $\pZ$ was defined. Its value for a
$\Dn$-coloured knot in $S^3$, in terms of a Seifert matrix $\M$ and
a vector $\vec{w}$ which determines the $\Dn$-colouring $\rho$, is
given by the formula:

$$
\mathrm{cu}(K,\rho)=\frac{2(\vec{w}^{\thinspace T}\cdot \M\cdot
\vec{w})}{n}\bmod n.
$$

The value $\mathrm{cu}(K,\rho)\in\pZ$ was called the \emph{coloured
untying invariant} of $(K,\rho)$. It was proven there that
$\mathrm{cu}$ is invariant under surgery in $\ker\rho$, and so, in
particular, is constant function on $\rho$--equivalence classes (see
also \cite{LiWal08}). It was also shown there that every possible
value is realized by some $\Dn$-coloured knot. These facts imply
that the number of $\rho$--equivalence classes is at least $n$. On
the other hand, because the complete sets of base-knots in Theorem
\ref{T:UnKnotBaseKnot} and in Theorem \ref{T:PretzBaseKnot} both
have cardinality $n$, it follows that $n$ is also an upper bound for
the number of $\rho$--equivalence classes. Thus a by-product of our
constructions is:

\begin{cor}\label{C:RhoCorollary}
Two $\Dn$-coloured knots have the same coloured untying invariant if
and only if they are $\rho$--equivalent. In particular, the number
of $\rho$--equivalence classes of $\Dn$-coloured knots is $n$.
\end{cor}

For $n$ prime this was \cite[Conjecture 1]{Mos06b}, where it was
proved for $n=3$ and for $n=5$.  This conjecture was also the
subject of \cite{LiWal08}, where bordism theory was used to put an
upper bound of $2n$ on the number of $\rho$--equivalence
classes.\par

\subsection*{The view from here}
As stated at the beginning of the introduction, our motivation is to
develop a theory of quantum topology for dihedral covering spaces
and covering links. How to proceed? Many tantalizing hints can be
found in the literature.

One possible route would be to generalize recent results in the
cyclic case \cite{GK03,GK03b} due to Garoufalidis and Kricker. The
results culminate in a universal formula for the LMO invariant of a
cyclic branched cover in terms of the rational lift of the loop
expansion of the Kontsevich invariant \cite{GK04}. This rational
lift may be viewed as a version of the Kontsevich invariant coloured
by the canonical representation
$\pi_1\left(\overline{S^3-N(K)}\right)\to\mathds{Z}$.

Using the surgery presentations in this paper, one should be able to
obtain a version of these constructions where the colouring group is
$\Dn$ instead of $\mathds{Z}$. Taking the `$1$--loop part' will give
an analogue to the Alexander polynomial. The $2$--loop part should
determine the Casson(--Walker--Lescop) invariant for an irregular
dihedral covering space $M$ (which can be any $3$--manifold).

Another clue for the shape of such a theory is a mysterious formula
for the Rohlin invariant of a dihedral branched covering space that
was discovered in the seventies by Cappell and Shaneson
\cite{CS75,CS84}. Recall that the Rohlin invariant is the mod $2$
reduction of the Casson--Walker invariant, which is the unique
finite type invariant of degree $1$.\par

The theory of knot concordance has long been a blind spot for
`traditional' quantum topology. The classical invariants which
access this type of information are typically constructed from
systems of covering spaces. One of our longer term goals is to
develop sufficient technology to make contact with these
constructions.

\subsection*{Odds and ends}
The paper concludes in Section \ref{S:oddsandends} with a variety of
odds and ends which are immediate corollaries of the constructions
in the previous sections. First, the choice of a complete set of
base-knots in Theorem \ref{T:PretzBaseKnot}, which was made after
trial and error, is of course not the only one possible. Some other
choices are also worth mentioning. By choosing the twist knots in
Figure \ref{F:twistknot} as a complete set of base-knots, we can
prove that the surgery link in Theorem \ref{T:UnKnotBaseKnot} can be
chosen to have linking number zero with the component labeled by
$s$. There $m=1-\frac{(n+1)^2}{2}$ if $\frac{n+1}{2}$ is even, while
if $\frac{n+1}{2}$ is odd then $m=2-\frac{n^2+1}{2}$. Choosing the
torus knots of Figure \ref{F:torusknot} gives a picture that is easy
to lift (see \cite{Mos06b} for the $n=3$ and $n=5$ cases) but is not
a natural end-point for our algorithms. Using it we can prove that a
$3$--manifold with $\Dn$-symmetry has a surgery presentation with
$\Dn$-symmetry, extending a `visualization' result of Przytycki and
Sokolov \cite{PrzS01} and of Sakuma \cite{Sak01}. Finally, we may
choose a complete set of base-knots which differ only by the choice
of their $\Dn$-colouring, as shown in Figure \ref{F:1knot8col}.

\begin{figure}
\psfrag{l}[c]{\Huge$\vdots$}
\begin{minipage}{2in}
\psfrag{t}[c]{$t$}\psfrag{s}[c]{$ts$}\psfrag{k}[r]{\parbox{0.65in}{\tiny$kn+m$\\[0.1cm]
\ \ \
$\frac{1}{2}$--twists}\normalsize$\left\{\rule{0pt}{0.55in}\right.$}
\includegraphics[width=1.1in]{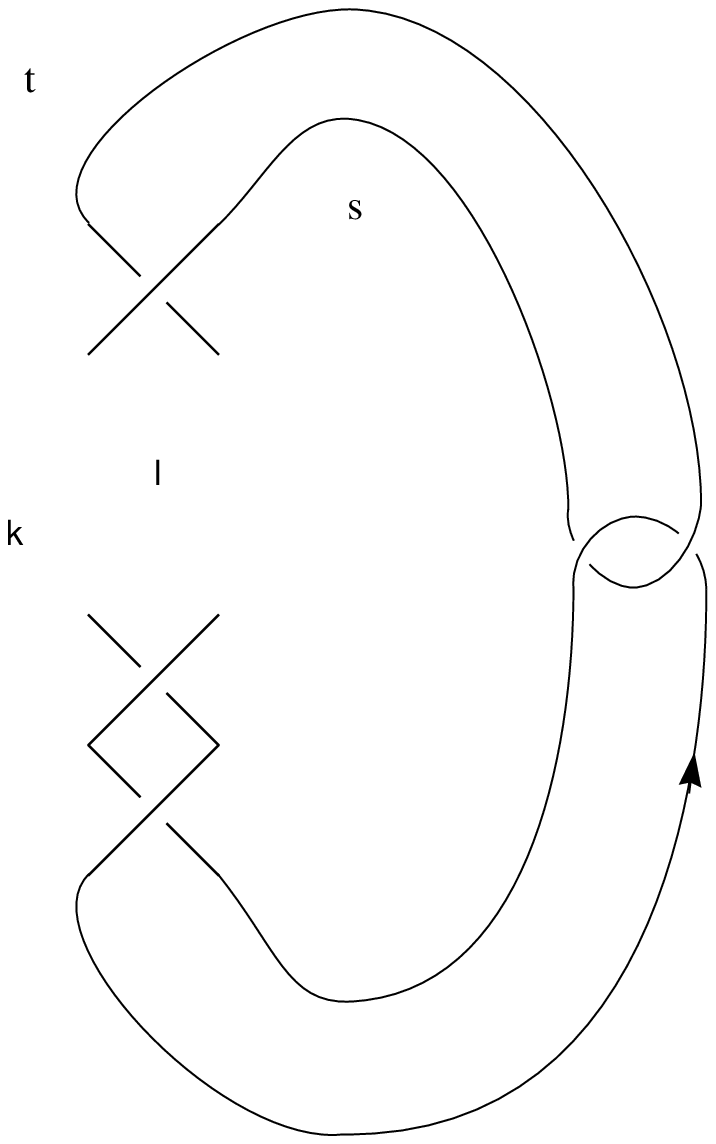}
\caption{\label{F:twistknot}}
\end{minipage}
\qquad\
\begin{minipage}{2in}
\psfrag{t}[c]{$t$}\psfrag{s}[c]{$ts$}
\psfrag{k}[r]{\parbox{0.42in}{\tiny$(2k+1)n$\\[0.1cm]
$\frac{1}{2}$--twists}\normalsize$\left\{\rule{0pt}{0.55in}\right.$}
\includegraphics[width=1.1in]{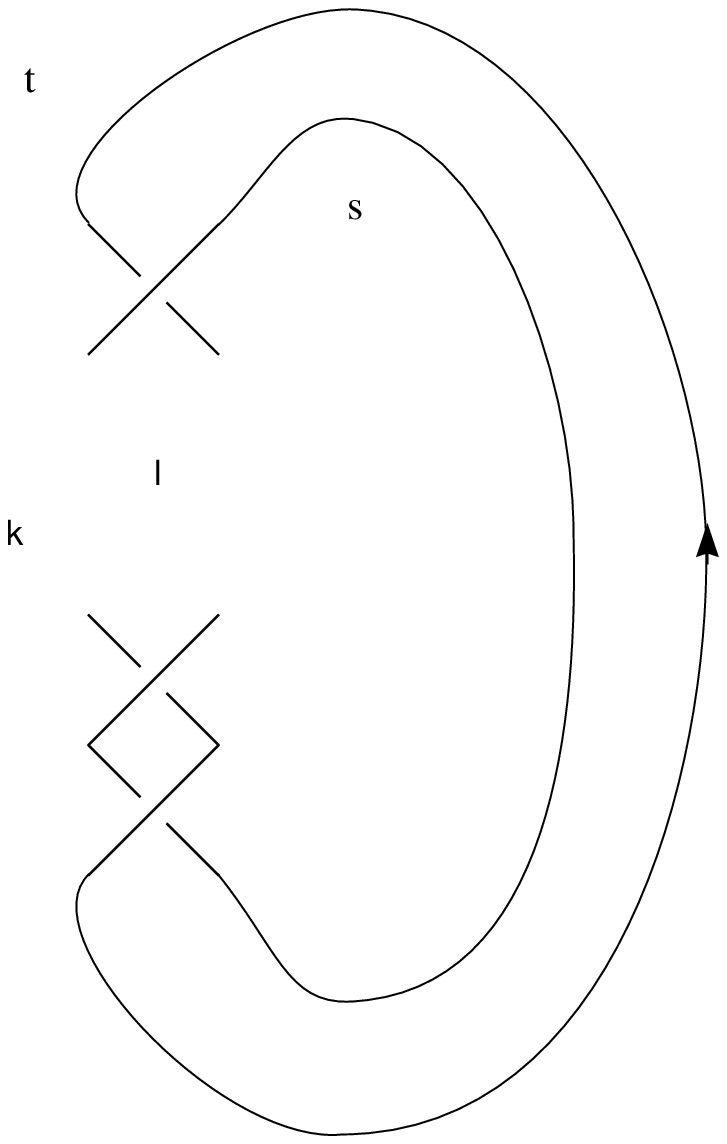}
\caption{\label{F:torusknot}}
\end{minipage}\\[0.3in]
\begin{minipage}{2in}
\psfrag{k}[r]{\parbox{0.42in}{\tiny$\ \ \ \ n$\\[0.1cm]
$\frac{1}{2}$--twists}\normalsize
$\left\{\rule{0pt}{0.55in}\right.$}
\psfrag{t}[c]{$s$}\psfrag{s}[c]{$s^k$}
\includegraphics[width=1.1in]{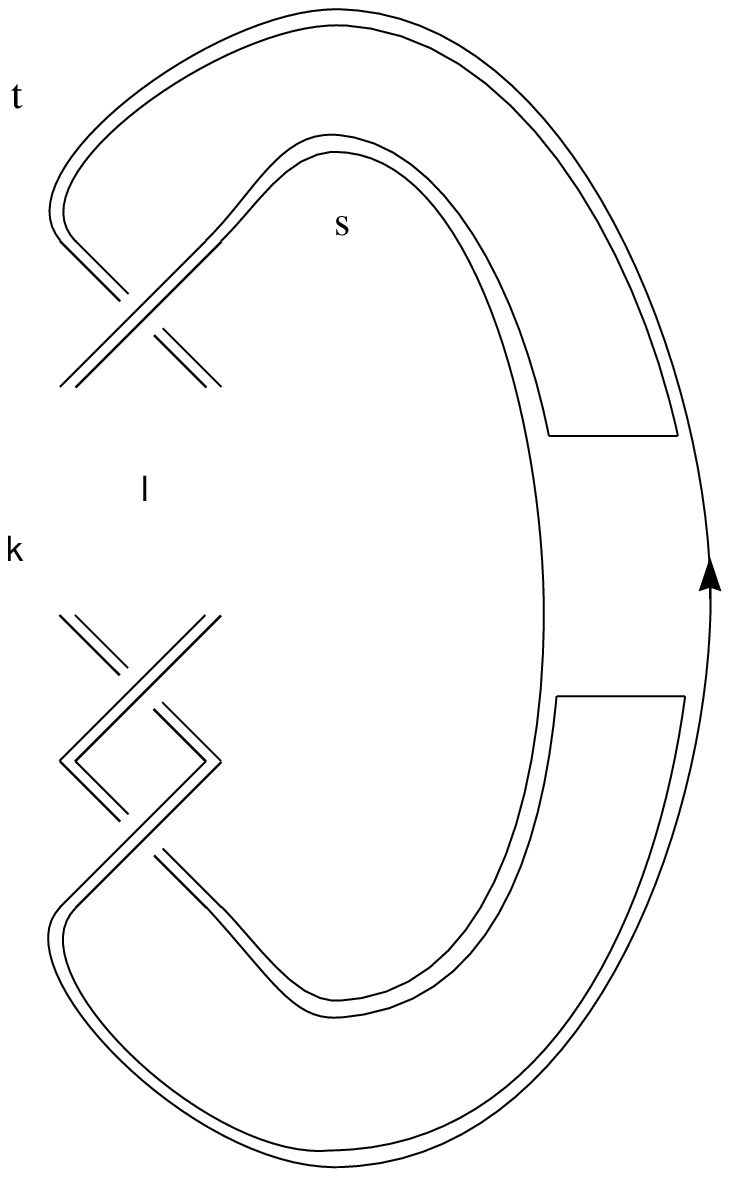}
\caption{\label{F:1knot8col}}
\end{minipage}
\end{figure}

Although the methods in this paper are elementary, the results
appear to be new. Swenton \cite{Swe04}, and independently Yamada
\cite{Yam02} for $n=3$, give quite different algorithms for
translating from dihedral covering presentations to surgery
presentations, `forgetting' the knot.


\subsection*{Some further problems}
\begin{itemize}
\item Explore the relationship between the untying approach and the
band projection approach. In particular, how can one calculate the
coloured untying invariant of a $\Dn$-coloured knot in a
$3$-manifold other than $S^3$?
\item Explore the possibility of using Goeritz surfaces instead of
Seifert surfaces in the band projection approach, giving the torus
knots as a complete set of base knots directly.
\item Find minimal complete sets of base-knots for groups other than the
dihedral group. Use these to find presentations for other classes of
covering spaces.
\item Extend the results of this paper to $\Dn$-coloured algebraically
split links (the extension to boundary links is straightforward).
\end{itemize}

\section{Dihedral branched covering spaces}\label{S:definitions}

In this section we recall the way in which $(K,\rho)$ presents a
closed $3$--manifold $M$.\par

We first recall how a \textit{monodromy representation}
characterizes an (unbranched) covering space. Let $\mathrm{pr}\co
\widetilde{X}\twoheadrightarrow X$ be an $n$--fold (unbranched)
covering space of a closed $3$--manifold $X$ with basepoint $\ast$.
An oriented loop $\ell\subset X$ based at $\ast$ lifts to a
collection of distinct loops $\ell_1,\ldots,\ell_n$ each starting
and ending at one of the $n$ preimages $\ast_1,\ldots,\ast_n$ of
$\ast$ in $\widetilde{X}$. Sending the initial point of each of
these paths to its endpoint gives a permutation of
$\ast_1,\ldots,\ast_n$, inducing a representation

$$\pi_1(X,\ast)\rightarrow \mathrm{Sym}\left(\mathrm{pr}^{-1}(\ast)\right)$$

\noindent which is unique up to relabeling lifts of the basepoint.
Choosing a different basepoint in $X$ modifies the representation
via some bijection $\mathrm{pr}^{-1}(\ast)\simeq
\mathrm{pr}^{-1}(\ast^\prime)$.\par

The theory of covering spaces tells us that two covering spaces are
equivalent (that is, homeomorphic by a homeomorphism respecting the
covering map) if and only if their monodromy representations are the
the same (after some relabeling). Thus we can specify a covering
space by giving a representation $\pi_1(X,\ast)\rightarrow
\mathrm{Sym}\left(\mathrm{pr}^{-1}(\ast)\right)$.\par

From an intuitive cut-and-paste point of view it is natural to
present a covering space by means of its monodromy. This allows one
to construct it by cutting the base space into cells, taking the
appropriate number of copies of each cell, and gluing them together
according to the representation. The following example, which plays
a part in the proof of Theorem \ref{T:untyingmethod}, is a good
illustration of this.\par

\begin{example}\label{surfexamp} Consider a genus two handlebody
equipped with base point and a representation $\rho$ from its
fundamental group onto $D_{14}$. To construct the covering space
whose monodromy group is given by this representation, we begin by
cutting the handlebody into a cell:
$$
\begin{minipage}{150pt}
\includegraphics[width=150pt]{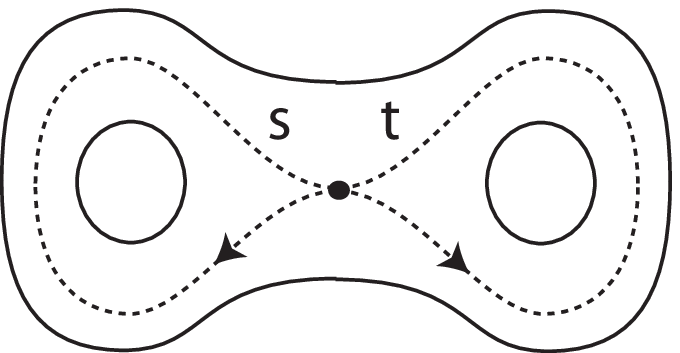}
\end{minipage}
\quad\ \ \ \begin{minipage}{30pt}
\includegraphics[width=30pt]{fluffyarrow}
\end{minipage}\quad
\begin{minipage}{150pt}
\includegraphics[width=150pt]{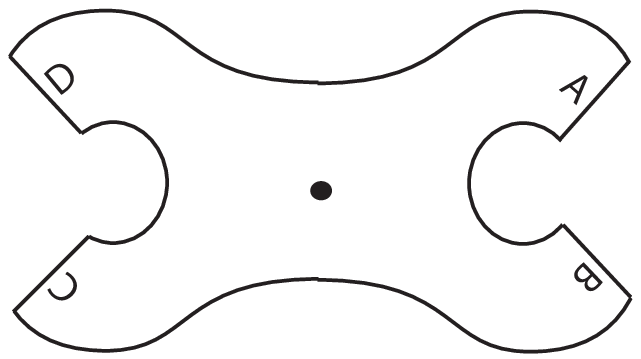}
\end{minipage}
$$
Now we take seven copies of this cell, and glue them together
according to $\rho$.
$$
\begin{minipage}{4in}
\psfrag{s}[c]{\normalsize$1$}\psfrag{w}[c]{\normalsize$2$}\psfrag{x}[c]{\normalsize$3$}\psfrag{y}[c]{\normalsize$4$}\psfrag{v}[c]{\normalsize$5$}\psfrag{u}[c]{\normalsize$6$}\psfrag{t}[c]{\normalsize$7$}%
\fs\psfrag{1}[c]{$D$}\psfrag{3}[c]{$D$}\psfrag{5}[c]{$D$}\psfrag{b}[c]{$D$}\psfrag{0}[c]{$D$}\psfrag{8}[c]{$D$}\psfrag{c}[c]{$D$}%
\psfrag{2}[c]{$C$}\psfrag{4}[c]{$C$}\psfrag{6}[c]{$C$}\psfrag{a}[c]{$C$}\psfrag{7}[c]{$C$}\psfrag{9}[c]{$C$}\psfrag{d}[c]{$C$}%
\psfrag{q}[c]{$A$}\psfrag{p}[c]{$A$}\psfrag{m}[c]{$A$}\psfrag{l}[c]{$A$}\psfrag{i}[c]{$A$}\psfrag{h}[c]{$A$}\psfrag{e}[c]{$A$}%
\psfrag{r}[c]{$B$}\psfrag{o}[c]{$B$}\psfrag{n}[c]{$B$}\psfrag{k}[c]{$B$}\psfrag{g}[c]{$B$}\psfrag{f}[c]{$B$}\psfrag{j}[c]{$B$}
\includegraphics[width=4in]{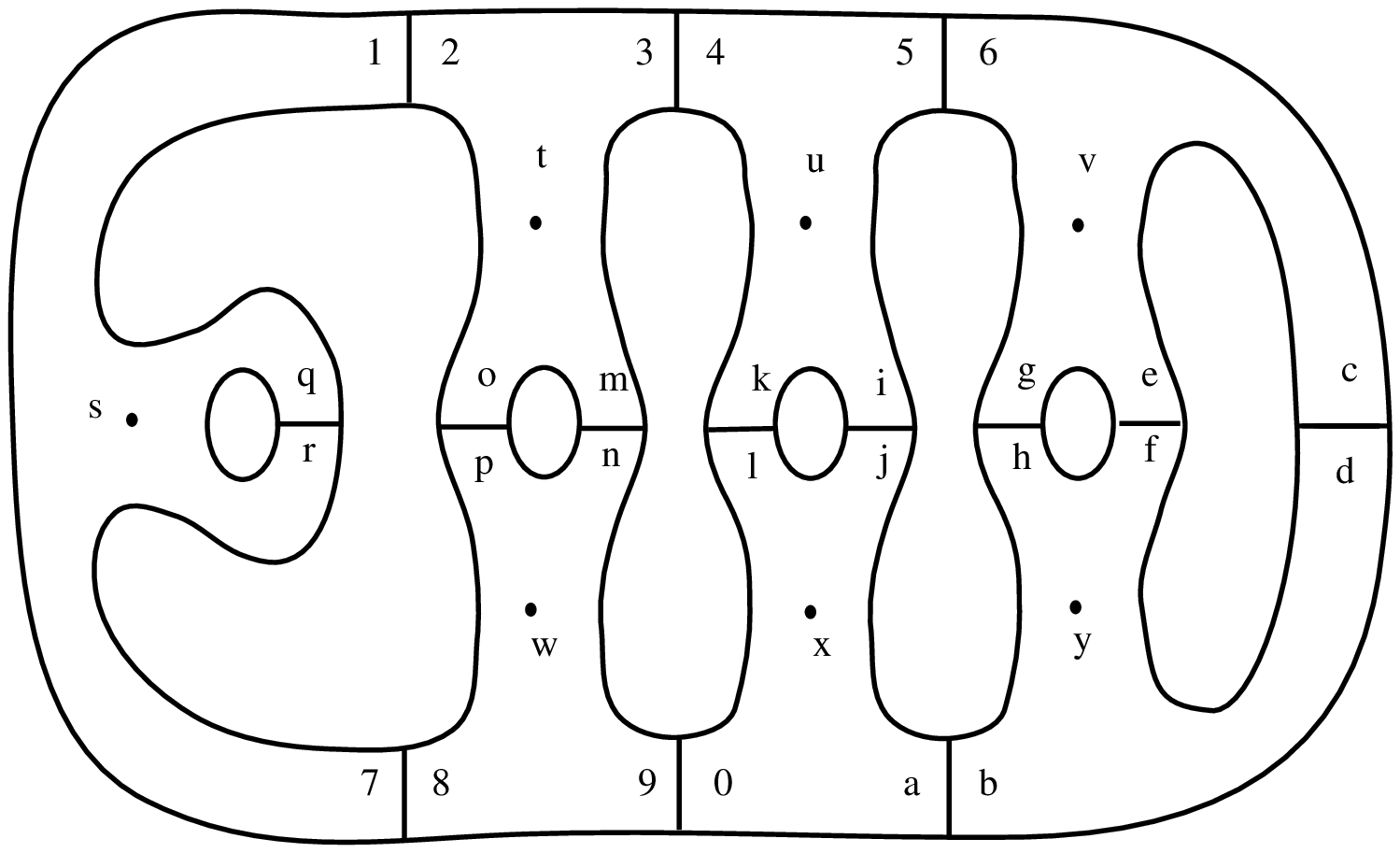}
\end{minipage}
$$
\end{example}

We now construct $M$. Begin from a $\Dn$-coloured knot $(K,\rho)$,
and consider the $n$--sheeted (unbranched) covering space
$\tilde{X}$ of the knot complement $X$ of $K$ (the closure in $S^3$
of the complement of a tubular neighbourhood $N(K)$ of $K$) with
monodromy given by $\rho$, where $\Dn$ is thought of as a subgroup
of $\mathrm{Sym}\left(\ast_1,\ldots,\ast_n\right)$.\par

Consider the boundary of this covering space. What is it? Well, the
$\Dn$-colouring $\rho$ sends a meridian to a reflection, and the
longitude may be chosen so that it is sent to $1$. It follows that
the boundary of $\widetilde{X}$ is a collection of $\frac{n+1}{2}$
tori--- $\frac{n-1}{2}$ two--sheeted coverings and $1$ one--sheeted
covering of the boundary torus $\partial N(K)$ of $X$. Glue
$\frac{n+1}{2}$ solid tori into these boundary components, longitude
to longitude, such that a meridional disc is glued into some lift of
a power of the meridian downstairs.\par

This is the desired space $M$: the branched dihedral covering space
of $S^3$ associated to the $\Dn$-coloured knot $(K,\rho)$. The cores
of the glued-in tori, with orientations induced by the orientation
of $K$, form the covering link $\tilde{K}$.

\begin{rem}
It is more usual to specify a covering space by a conjugacy class of
subgroups of $\pi_1(X)$ (corresponding to the image of
$\pi_1(\tilde{X})$ under the projection). If a covering space is
determined by a monodromy representation then the corresponding
class of subgroups is given by taking the stabilizer of a chosen
element in $\mathrm{Sym}\left(\ast_1,\ldots,\ast_n\right)$.
\end{rem}

\begin{rem}
The $3$--manifold $M$ is usually referred to as the \emph{irregular}
branched dihedral covering space associated to $(K,\rho)$ (as we
referred to it in the introduction), because it corresponds to the
preimage under $\rho$ of $\langle t \rangle$ which is not a normal
subgroup of $\Dn$. 
\end{rem}

\section{Untying approach}\label{S:untyingapproach}
This approach consists of two steps. The first is to obtain a
surgery presentation of a $\Dn$-coloured knot $(K,\rho)$ in the
complement of an unknot in a lens--space (one of the base-knots of
Theorem \ref{T:UnKnotBaseKnot}). Such a presentation is called a
\emph{\sds presentation} of $(K,\rho)$. The second step is to lift
the \sds presentation to a surgery presentation of the dihedral
branched covering space and of the covering link.

\subsection{Obtaining a \sds presentation}\label{SS:sdsPresentation}
The construction consists of three steps: use surgery to untie the
knot, perform handleslides to concentrate the non-trivial labels
onto a single surgery component, and finish with another round of
surgery to untie that surgery component. We'll also describe some
moves which put the labels and surgery curves in the resulting
diagram in a standard form.

\subsubsection{Untying the knot}\label{SSS:untie}
We can untie any knot $K$ by crossing changes, realized by surgery
on $\pm1$--framed unknots which have linking number zero with $K$.
This allows us to present $K$ as a $\pm1$--framed link $L$ in the
complement of a standard unknot $U\subset S^3$, such that surgery on
$L$ recovers $K\subset S^3$. In the following section we generalize
this procedure to $\Dn$-coloured knots.\par


Let us begin by reminding ourselves that the arcs of a knot in $S^3$
are all coloured by reflections (elements of the form $ts^a\in\Dn$).
This follows from the Wirtinger relations. Near a crossing where the
over-crossing arc is labeled $g_1$, the under-crossing arcs must be
labeled $g_2$ and either $g_1^{-1}g_2g_1$ or $g_1g_2g_1^{-1}$ for
some $g_2\in \Dn$. If any arc is labeled by a rotation then all arcs
in the knot diagram would be labeled by rotations (because
$\Cn\triangleleft \Dn$) which would contradict surjectivity of the
$\Dn$-colouring $\rho$.\par

When performing surgery, the colours of the arcs of the introduced
surgery component are induced as follows:

\begin{lem}\label{L:untielemone}
Let $g_1$ and $g_2$ be elements in $\Dn$. The local moves depicted
below induce colours on the added surgery components as shown. (The
two strands ``being twisted'' can be from the knot or from surgery
components.)

\begin{equation}
\begin{minipage}{95pt}
\psfrag{1}[c]{$g_1$}\psfrag{2}[c]{$g_2$}\psfrag{3}[c]{$g_2^{-1}g_1g_{2}$}\psfrag{4}[c]{$g_2^{-1}g_1g_2g_1^{-1}g_2$}
\includegraphics[width=95pt]{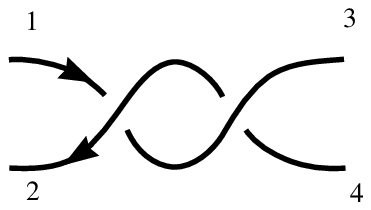}
\end{minipage} \qquad\Longleftrightarrow\qquad
\raisebox{8pt}{\begin{minipage}{140pt}
\psfrag{1}[c]{$g_1$}\psfrag{2}[c]{$g_2$}\psfrag{3}[c]{$g_2^{-1}g_1g_{2}$}\psfrag{4}[c]{$g_2^{-1}g_1g_2g_1^{-1}g_2$}\psfrag{5}[c]{$g_2g_1^{-1}$}\psfrag{6}[c]{$g_1^{-1}g_2$}
\includegraphics[width=140pt]{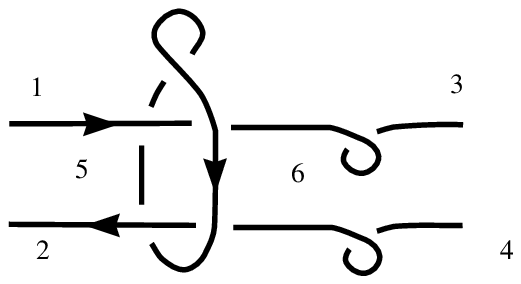}
\end{minipage}}
\end{equation}

\begin{equation}\begin{minipage}{95pt}
\psfrag{1}[c]{$g_1$}\psfrag{2}[c]{$g_1g_2g_1^{-1}$}\psfrag{3}[c]{$g_2g_1g_2^{-1}$}\psfrag{4}[c]{$g_2$}
\includegraphics[width=95pt]{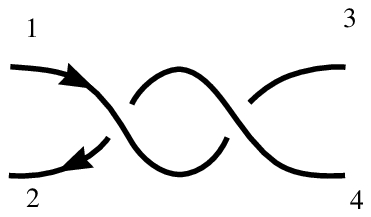}
\end{minipage}\qquad \Longleftrightarrow\qquad
\raisebox{8pt}{\begin{minipage}{140pt}
\psfrag{1}[c]{$g_1$}\psfrag{2}[c]{$g_1g_2g_1^{-1}$}\psfrag{3}[c]{$g_2g_1g_2^{-1}$}\psfrag{4}[c]{$g_2$}\psfrag{5}[c]{$g_1^{2}g_2^{-1}g_1^{-1}$}\psfrag{6}[c]{$g_1g_2^{-1}$}
\includegraphics[width=140pt]{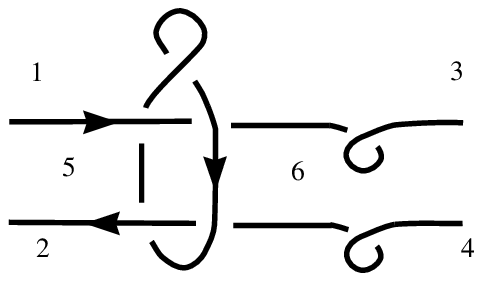}
\end{minipage}}
\end{equation}
\end{lem}

\begin{proof}
The precise claim is that there is a PL--homeomorphism $h$ between
the two spaces, taking the knot in one space onto the knot in the
other space, such that the pulled-back representation of the knot
group is as shown. The homeomorphism $h$ is to cut
$\overline{S^3-T}$ along a disc spanning $T$ a tubular neighbourhood
of the introduced surgery component, do a $2\pi$ twist in the
appropriate direction, then reglue the disc and the solid torus.

The label on an arc of the right-hand diagram is determined by the
image under $h$ of a path representing the appropriate element of
the fundamental group. For example, we obtained the label
$g_1^{-1}g_2$ in the first local move by finding the image under $h$
of the Wirtinger generator corresponding to the appropriate meridian
of the introduced surgery curve:

$$
\begin{minipage}{120pt}
\includegraphics[width=120pt]{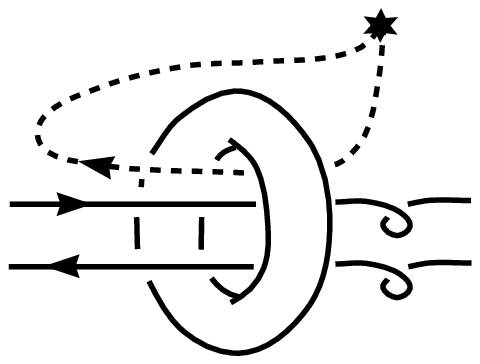}
\end{minipage}\qquad\raisebox{-12pt}{$\Longleftrightarrow$}\qquad
\begin{minipage}{150pt}
\includegraphics[width=150pt]{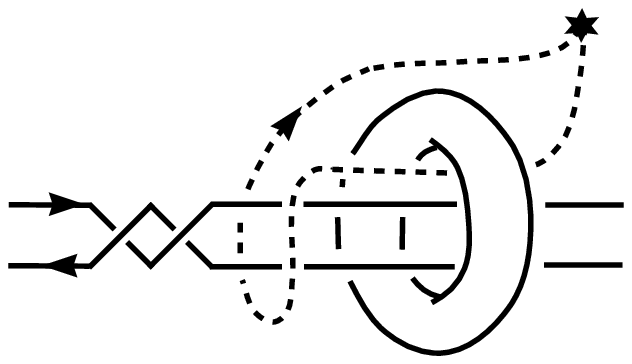}
\end{minipage}$$
\end{proof}

We'll use this lemma to untie $\Dn$-coloured knots. In that case
we'll have $g_1=ts^a$ and $g_2=ts^b$ for some $a,b\in \pZ$, so that
$g_1g_2^{-1}=s^{a+b}$, and the typical move will look like:

\begin{equation*}
\begin{minipage}{95pt}
\psfrag{1}[c]{$ts^a$}\psfrag{2}[c]{$ts^b$}\psfrag{3}[c]{$ts^{2a-b}$}\psfrag{4}[c]{$ts^{3b-2a}$}
\includegraphics[width=95pt]{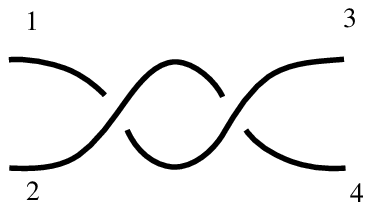}
\end{minipage} \qquad\Longleftrightarrow\qquad
\raisebox{8pt}{\begin{minipage}{120pt}
\psfrag{1}[c]{$ts^a$}\psfrag{2}[c]{$ts^b$}\psfrag{3}[c]{$ts^{2a-b}$}\psfrag{4}[c]{$ts^{3b-2a}$}\psfrag{5}[c]{$s^{b-a}$}\psfrag{6}[c]{$s^{a-b}$}
\includegraphics[width=120pt]{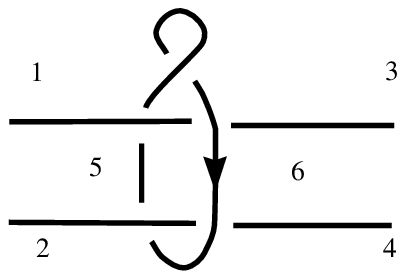}
\end{minipage}}
\end{equation*}

If some meridian of $K$ maps to a reflection $ts^a$, then because
$\Cn$ is a normal subgroup of $\Dn$ all meridians of $K$ map to that
same reflection $ts^a$. Arcs in $L$ are labeled by rotations, and
conjugation by a reflection $ts^a$ of a rotation $s^b$ maps it to
$s^{-b}$. Therefore for any component $C_i$ of $L$ there exists
$j\in\pZ$ such that all arcs of $C_i$ are labeled either $s^j$ or
$s^{-j}$.\par

Because $\rho$ is surjective, $s$ is generated by labels of the arcs
of the knot $K$. Thus there exists an element in
$\pi_1\left(\overline{S^3-N(K)}\right)$ which is represented by a
curve $C$ which has a meridian which maps to $s$. Because $s$ is a
rotation, this curve passes under an even number of arcs of $K$, and
we may choose such a curve to have linking number zero with $K$
because the labels on $K$'s arcs all have order $2$. Perform surgery
on $C$, and set $C_1\ass C$. Now untie the resulting knot by
crossing changes, realized by $\pm1$--framed surgeries along unknots
which have linking number zero with the knot. We obtain surgery
presentation $L$ for $K$ in the complement of the unknot $U$, for
which an arc of $C_1$ is labeled $s$, and so all arcs of $C_1$ are
labeled either $s$ or $s^{-1}$. We have shown the following lemma:

\begin{lem}\label{L:distinguishedcomponent}
The surgery presentation $L$ may be chosen such that the arcs of
$C_1$ are labeled $s$ and $s^{-1}$.
\end{lem}

We call $C_1$ the \emph{distinguished surgery component}.

\subsubsection{Handleslides}
A surgery component whose arcs are all labeled $1$ (the identity in
$\Dn$) is said to be \emph{in $\ker\rho$}. The second step of the
construction is to perform handleslides so as to arrange that every
surgery component except for one distinguished component is in
$\ker\rho$. The following lemma tells us how labels transform under
handleslides.

\begin{lem}\label{L:handleslidelemma}
Two diagrams that differ by one of the moves shown below present
equivalent $\Dn$-coloured knots. (The displayed components are
surgery components.)

\begin{align*}
\raisebox{-0.85cm}{\scalebox{0.85}{\includegraphics{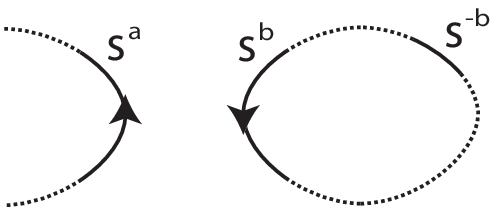}}}
\qquad &\Longleftrightarrow \qquad
\raisebox{-0.85cm}{\scalebox{0.85}{\includegraphics{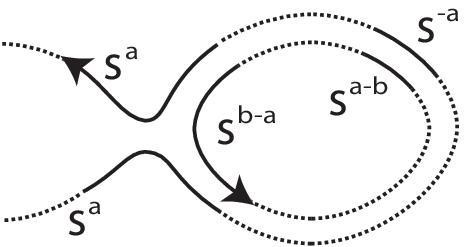}}}\\[0.1cm]
\raisebox{-0.85cm}{\scalebox{0.85}{\includegraphics{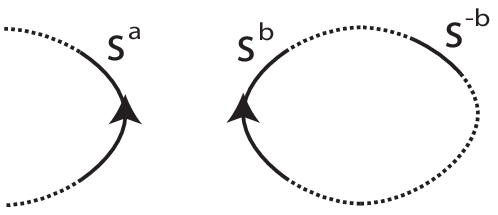}}}
\qquad &\Longleftrightarrow \qquad
\raisebox{-0.85cm}{\scalebox{0.85}{\includegraphics{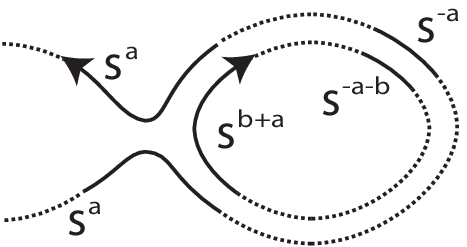}}}
\end{align*}
\end{lem}

\begin{proof}
When two diagrams are related by a handleslide, the corresponding
spaces are related by a PL--homeomorphism which is the identity
outside a genus two handlebody containing the two involved
components and the `path' of the slide.\par

To observe how labels transform: pick a curve representing some arc
in the right-hand diagram, isotope the curve so that it lies outside
the genus two handlebody corresponding to a handle slide which will
take us to the left-hand diagram, then read off what that curve maps
to in the left-hand diagram.\par

For example, the label $s^{b-a}$, above can be obtained as shown
below:

\begin{align*}
\raisebox{-0.85cm}{\scalebox{0.95}{\includegraphics{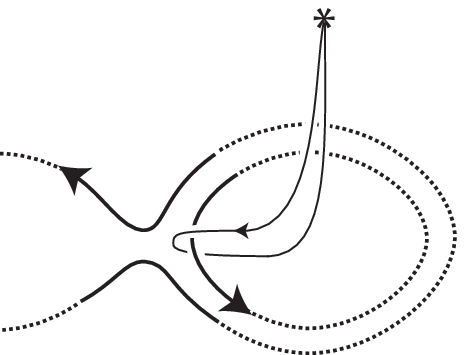}}} \ \ &
\approx & \ \
\raisebox{-0.85cm}{\scalebox{0.95}{\includegraphics{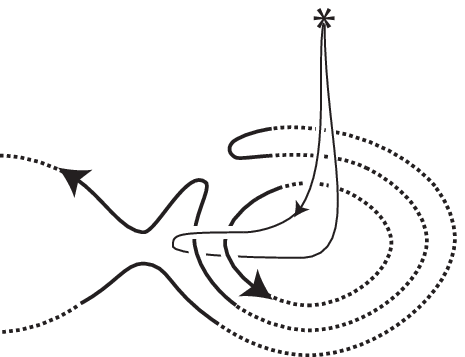}}}
\\[0.1cm]
\ \ & \approx & \ \
\raisebox{-0.85cm}{\scalebox{0.95}{\includegraphics{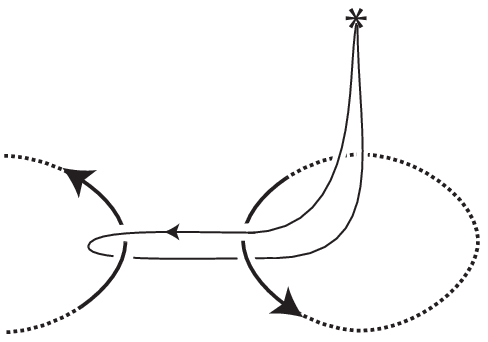}}}
\end{align*}

Note further that once we know what the label on one of the arcs of
a component is, the labels on all of its arcs are determined by the
fact that they must induce a well-defined representation onto $\Dn$.
\end{proof}

By Lemma \ref{L:handleslidelemma} we may repeatedly perform
handleslides until all of the surgery components are in $\ker\rho$
except $C_1$ (the distinguished surgery component). By Lemma
\ref{L:distinguishedcomponent} we may assume $C_1$ has an arc
labeled $s$. For each $1<i\leq \mu$ let $a_i\in\pZ$ be an element
such that some arc of $C_i$ is labeled $s^{a_i}$. The effect of
sliding $C_1$ over $C_i$ is to replace $a_i$ by $a_i-1$ (or to
$a_i+1$ depending on which version of the handleslide is used). Thus
sliding $C_1$ over $C_i$ repeatedly $a_i$ times (or $-a_i$ times)
kills the labels of the arcs of $C_i$. Repeat for all
$i=2,\ldots,\mu$.

\begin{rem}
Readers who try some examples will find that this second step can
add significant complexity to the construction. However things are
not so bad when $n=3$. The reason is that there will only ever be a
single handleslide required to kill the label on a surgery
component, because $1+2=0 \bmod 3$ or $1-1=0 \bmod 3$. Note further
that in this situation the surgery components will remain framed
unknots after the handleslides.
\end{rem}

\subsubsection{Putting the presentation into a standard form}
After the first two steps we have a diagram where:
\begin{itemize}
\item{The knot $K$ has been untied and is in its standard
position $U$.}
\item{There are a number of surgery components, each of
which has linking number zero with the knot.}
\item{Every
surgery component, except one, is in $\ker\rho$, \textit{i.e.} has
all of its arcs labeled $1\in \Dn$.}
\item{The remaining component $C_1$ has each of its
arcs labeled either $s$ or $s^{-1}$.}
\end{itemize}

The final step is to add extra surgery components so that the two
component sublink $U\cup C_1$ becomes a standard two component
unlink. We will require that the surgery components introduced to
make this happen are in $\ker\rho$.

In the neighbourhood of a crossing in $C_1$, either all arcs will be
labeled $s$, in which case we can reverse the crossing by:

\begin{equation}\label{E:changeofcrossingsA}
\psfrag{1}[c]{$s$}\psfrag{2}[c]{$s$}\psfrag{3}[c]{$s$}
\raisebox{-0.9cm}{\scalebox{0.9}{\includegraphics{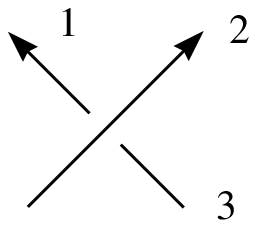}}}
\Longleftrightarrow\ \psfrag{4}[c]{$s$}\psfrag{5}[c]{$1$}
\raisebox{-0.9cm}{\scalebox{0.9}{\includegraphics{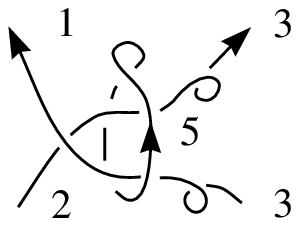}}}\ \
\ \  \mbox{and}\ \ \ \
\raisebox{-0.9cm}{\scalebox{0.9}{\includegraphics{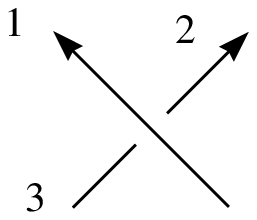}}}
\Longleftrightarrow\
\raisebox{-0.9cm}{\scalebox{0.9}{\includegraphics{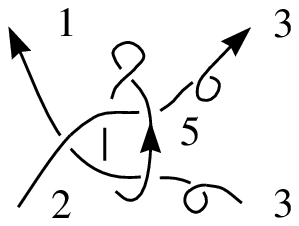}}}
\end{equation}

\noindent or the crossing will have one incident arc labeled $s$ and
another incident arc labeled $s^{-1}$, which can be dealt with by:

\begin{equation}\label{E:changeofcrossingsB}
\psfrag{1}[c]{$s$}\psfrag{2}[c]{$s^{-1}$}\psfrag{3}[c]{$s$}
\raisebox{-0.9cm}{\scalebox{0.9}{\includegraphics{untiecolourA}}}
\Longleftrightarrow\
\psfrag{2}[c]{$s^{-1}$}\psfrag{4}[c]{$s$}\psfrag{5}[c]{$1$}\psfrag{3}[c]{$s^{-1}$}
\raisebox{-0.9cm}{\scalebox{0.9}{\includegraphics{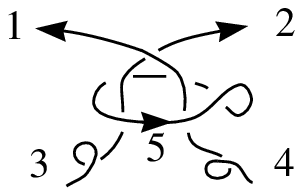}}}\
\ \ \  \mbox{and}\ \ \ \
\raisebox{-0.9cm}{\scalebox{0.9}{\includegraphics{untiecolourC}}}
\Longleftrightarrow\
\raisebox{-0.9cm}{\scalebox{0.9}{\includegraphics{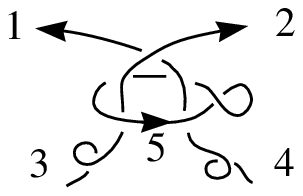}}}
\end{equation}

\noindent Thus, we can reverse any crossing on $C_1$ by surgery in
$\ker\rho$, untying $C_1$ and unlinking it from $U$.\par

Notice that the framing of the distinguished component must end up a
multiple of $n$. This is because labels induce a well-defined
representation of $\pi_1(\overline{S^3-N(K)})$ onto $\Dn$, so
contractible curves map to $1$. The framing curve of every surgery
component (in particular, the distinguished component) bounds a disc
in the corresponding torus being glued in and is thus contractible.
Since the distinguished component is labeled $s$ and is disjoint
from the knot (so its longitude maps to $1$), its framing must
vanish modulo $n$.\par

It is possible to introduce extra surgery components into the
presentation which will change that framing by $n^2$. It follows
that $k$ may be chosen so that $0\leq k < n$. To do this, coil the
distinguished surgery component into $n$ parallel strands:

\begin{equation}\label{E:dothetwistA}
\qquad\ \ \ \ \begin{minipage}{250pt}
\psfrag{s}[c]{$s$}\psfrag{t}[c]{$t$}\psfrag{p}[r]{$n$
strands}\psfrag{f}[r]{framing$=kn$}
\includegraphics[width=250pt]{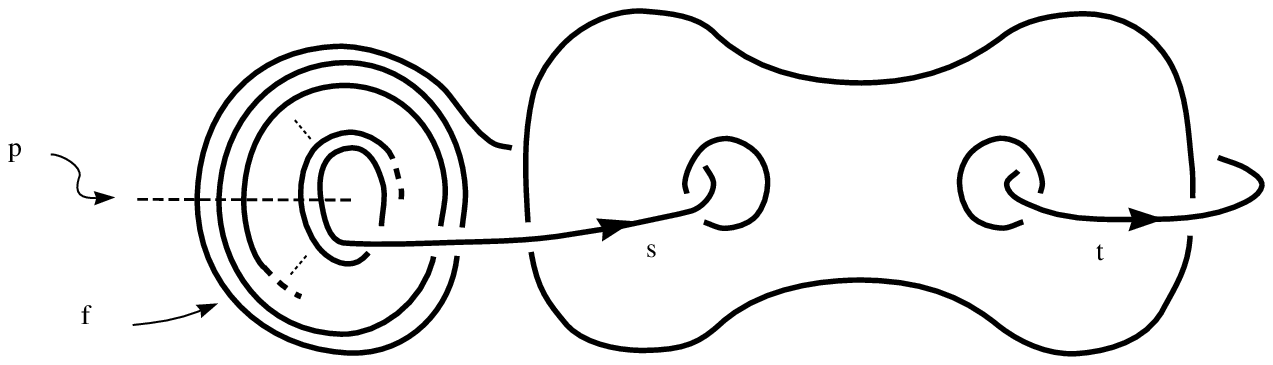}
\end{minipage}
\end{equation}

\noindent Then add a $\pm 1$--framed surgery component. (Choose $+1$
to increase framing by $n^2$, and $-1$ to decrease framing by
$n^2$.)

\begin{equation}\label{E:dothetwistB}
\qquad\ \ \ \
\begin{minipage}{250pt}
\psfrag{s}[c]{$s$}\psfrag{t}[c]{$t$}\psfrag{p}[r]{$n$
strands}\psfrag{f}[r]{framing\thinspace$=kn-n^2$}\psfrag{1}[c]{\fs$2\pi$
twist}
\includegraphics[width=250pt]{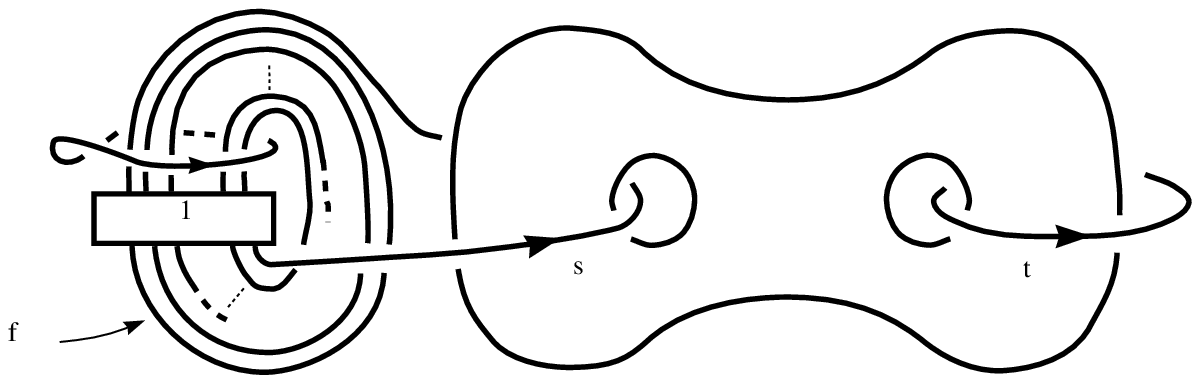}
\end{minipage}
\end{equation}

The distinguished surgery component may now be tied in a knot, but
we can untie it using surgery in $\ker\rho$, as shown in Equation
\ref{E:changeofcrossingsA}. Observe that such moves do not change
the framing of the distinguished surgery component because the
linking number of the introduced surgery components with $C_1$ is
zero.\par

All arcs of $U$ are labeled by some reflection $ts^a\in \Dn$, but by
the ambient isotopy of Figure \ref{F:innerauto} we may conjugate
this label by $s$, so that the label on the arcs of $U$ becomes
$ts^{a-2}$. Repeating $\frac{a}{2}\bmod n$ times, we obtain a
presentation for $(K,\rho)$ in which all arcs of $U$ are labeled
$t$.

\begin{figure}
\psfrag{s}{$s$}
\includegraphics[width=200pt]{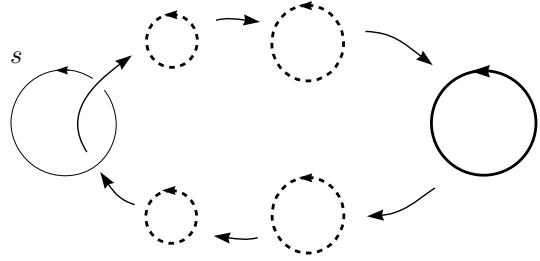}
\caption{\label{F:innerauto}Ambient isotopy of the surgery picture
to conjugate the label on $U$ by $s$.}
\end{figure}

To summarize:

\begin{prop}
Any $\Dn$-coloured knot $(K,\rho)$ has an \sds presentation,
\textit{i.e.} it has an surgery presentation $L=C_1\cup\cdots\cup
C_\mu$ such that:
\begin{itemize}
\item The distinguished surgery component $C_1$ has all its arcs labeled $s$
and has framing $kn$ with $0\leq k<n$.
\item All the other components $C_2,\ldots,C_\mu$ are in $\ker\rho$.
\item All arcs of $U$ are labeled $t$.
\item $C_1\cup U$ is the standard $2$--component unlink.
\end{itemize}
\end{prop}

\begin{figure}[h]
\psfrag{s}[c]{$s$}\psfrag{t}[c]{$t$}\psfrag{k}[c]{$U$}\psfrag{f}[r]{framing\thinspace$=kn$}
\qquad\includegraphics[width=3.2in]{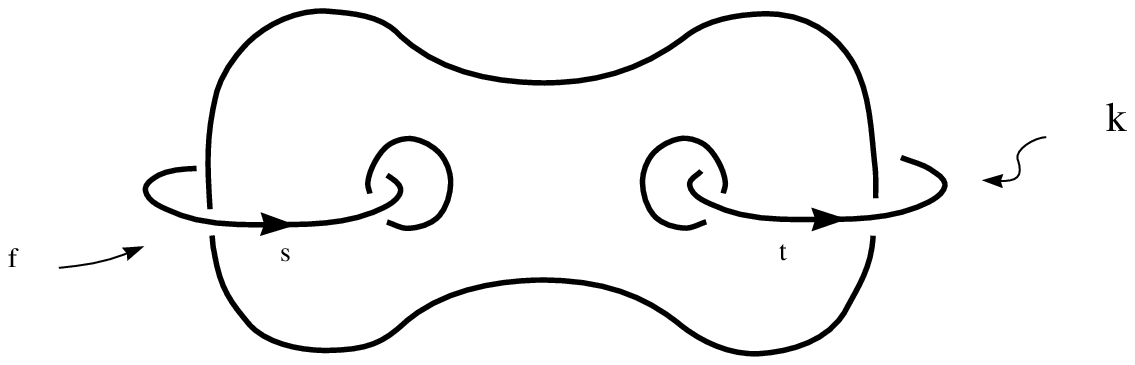}
\caption{\label{F:characteristicpres} A \sds presentation.}
\end{figure}

In Section \ref{S:lknumzero} we will additionally show that a \sds
presentation may be chosen such that the components
$C_2,\ldots,C_\mu$ in $\ker\rho$ all have linking number zero with
the distinguished component $C_1$.

\subsubsection{Example: The $D_{14}$-coloured $5_2$
knot}\label{SSS:fivetwoexample}

As an example, let's see how we obtain a \sds presentation of a
$D_{14}$-coloured $5_2$ knot.

\begin{multline*}
\begin{minipage}{111pt}
\psfrag{a}[c]{$t$}\psfrag{b}[c]{$ts^5$}\psfrag{c}[c]{$ts^3$}\psfrag{d}[c]{$ts$}\psfrag{e}[c]{$ts^2$}
\includegraphics[width=111pt]{fivetwoAA}
\end{minipage}\ \ \ \overset{\text{Surgery}}{\begin{minipage}{30pt}
\includegraphics[width=30pt]{fluffyarrow}
\end{minipage}}\ \ \ \
\begin{minipage}{150pt}
\psfrag{a}[c]{$ts^2$}\psfrag{b}[c]{$t$}\psfrag{c}[c]{$ts^5$}\psfrag{d}[c]{$ts^2$}\psfrag{e}[c]{$ts^4$}\psfrag{f}[c]{$ts^6$}
\psfrag{1}[c]{$s^{-1}$}\psfrag{2}[c]{$s$}
\includegraphics[width=130pt]{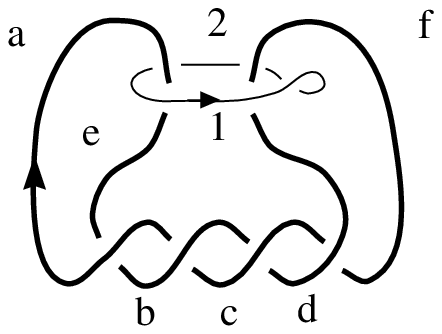}
\end{minipage}\\
\overset{\text{Isotopy}}{\begin{minipage}{30pt}
\includegraphics[width=30pt]{fluffyarrow}
\end{minipage}}\quad
\begin{minipage}{130pt}
\psfrag{a}[c]{$s$}\psfrag{b}[c]{$s$}\psfrag{c}[c]{$s^{-1}$}\psfrag{d}[c]{$s$}\psfrag{e}[c]{$s^{-1}$}\psfrag{f}[c]{$s^{-1}$}
\psfrag{1}[c]{$ts^{2}$}\psfrag{2}[c]{$t$}
\includegraphics[width=125pt]{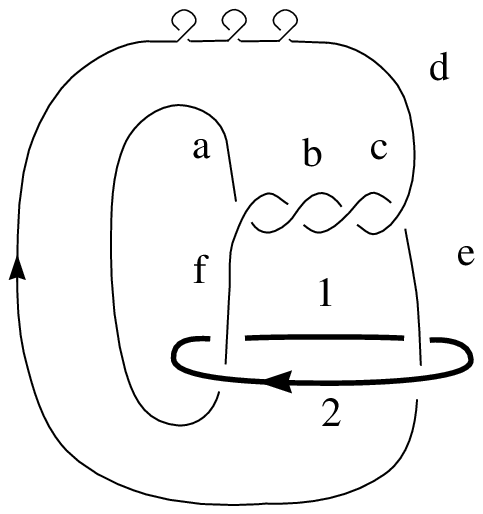}
\end{minipage}\quad
\overset{\text{Surgery}}{\begin{minipage}{30pt}
\includegraphics[width=30pt]{fluffyarrow}
\end{minipage}}\quad
\begin{minipage}{130pt}
\psfrag{a}[c]{$s$}\psfrag{b}[c]{$s^{-1}$}\psfrag{c}[c]{$s$}\psfrag{d}[c]{$s^{-1}$}\psfrag{e}[c]{$s$}\psfrag{f}[c]{$s^{-1}$}\psfrag{g}[c]{}\psfrag{h}[c]{$s^{-1}$}
\psfrag{1}[c]{$ts^{2}$}\psfrag{2}[c]{$t$}\psfrag{3}[c]{$1$}\psfrag{4}[c]{$1$}
\includegraphics[width=130pt]{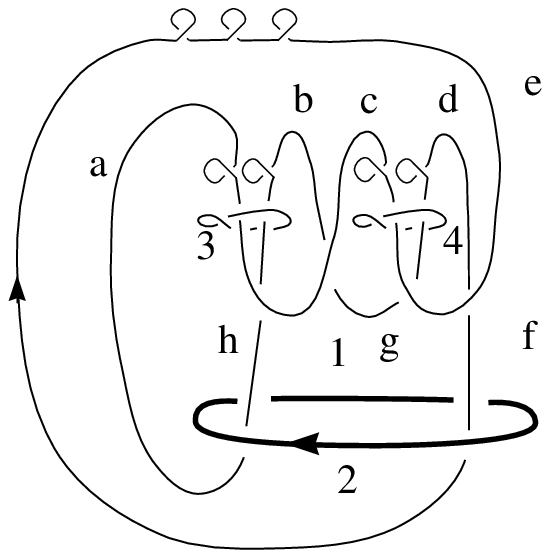}
\end{minipage}\\\ \\
\overset{\text{Isotopy}}{\begin{minipage}{30pt}
\includegraphics[width=30pt]{fluffyarrow}
\end{minipage}}\qquad\qquad\qquad\quad
\begin{minipage}{3.2in}
\psfrag{S}[c]{$s$}\psfrag{t}[c]{$t$}\psfrag{O}[c]{$U$}\psfrag{f}[r]{framing\thinspace$=-7$}
\includegraphics[width=3.2in]{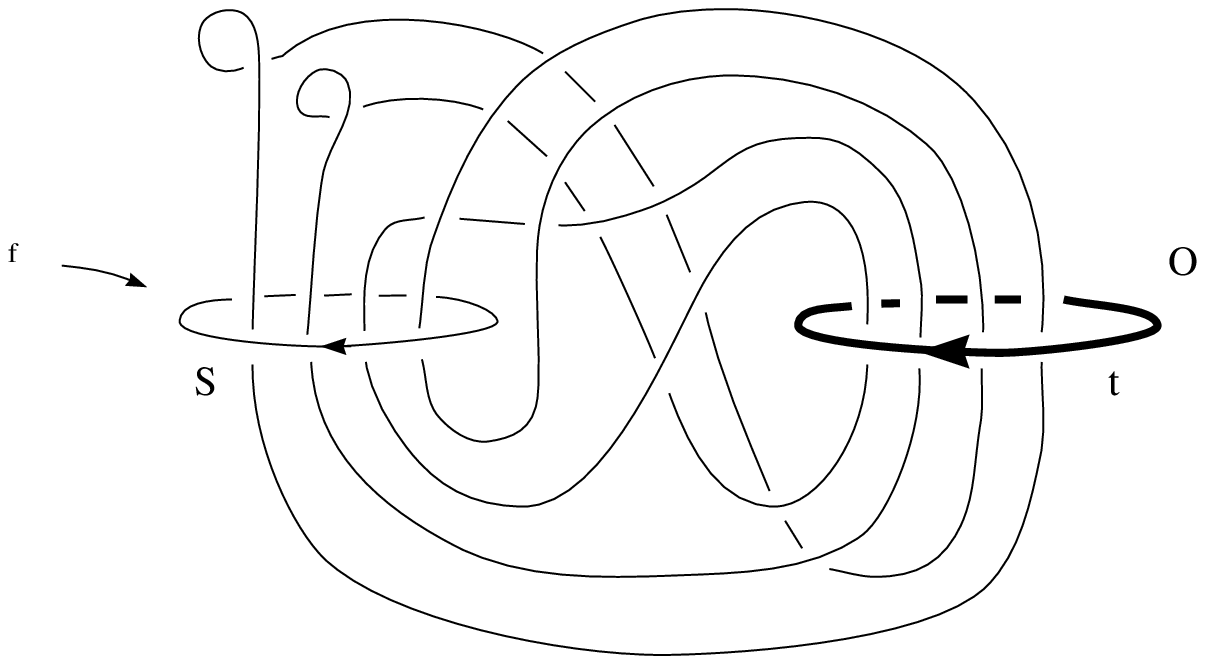}
\end{minipage}
\end{multline*}

\subsection{Constructing the cover}\label{S:coverconstruct}

Take a \sds presentation of some $\Dn$-coloured knot $(K,\rho)$. It
consists of a framed link $L$ in a genus two handlebody $H$,
embedded into a link in the way shown in Figure
\ref{F:characteristicpres}. Our goal in this section is to lift this
picture to a surgery presentation of $M$, the $n$--fold dihedral
covering space of $S^3$ branched over the knot $K$ whose monodromy
is given by $\rho$.

Our starting point is Figure \ref{F:Xconstructor}, which tells us
how to use the \sds presentation to construct the knot complement
$X\ass\overline{S^3-N(K)}$. This is achieved by doing surgery on
$L$, attaching $2$--handles to the curves $A$ and $B$, and finishing
by attaching a ball to the resulting $S^2$ boundary component.

\begin{figure}[h]
\psfrag{m}[c]{$m$}\psfrag{l}[c]{$l$}\psfrag{A}[c]{$A$}\psfrag{B}[c]{$B$}\psfrag{p}[r]{$n$
strands}
\includegraphics[width=3.2in]{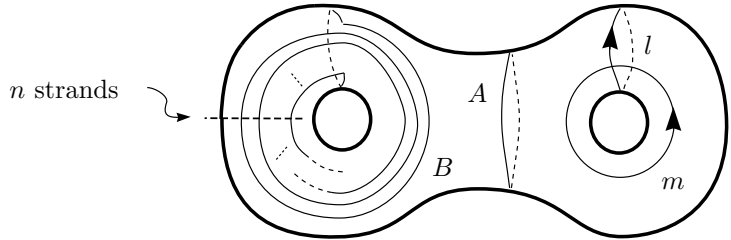}
\caption{How to construct $X\ass\
\overline{S^3-N(K)}$.\label{F:Xconstructor}}
\end{figure}

\noindent The knot complement $X$ comes equipped with a
representation $\rho\co \pi_1\left(X\right)\twoheadrightarrow \Dn\,$
determined by the labels $s$ and $t$.\par

On the boundary of $H$ we have also marked the meridian $m$ and a
choice of longitude $l$ of $K$. This data will be referred to below
as the \emph{peripheral markings}. We recover $S^3$ with the knot
$K$ embedded in it by gluing a solid torus $N(K)$, displayed in
Figure \ref{F:neighbourhoodofK}, into the boundary of $X$ (a torus),
so as to match up the curves $m$ and $l$.

\begin{figure}[h]
\scalebox{0.7}{\includegraphics{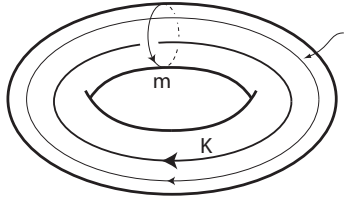}}
\caption{\label{F:neighbourhoodofK}The solid torus $N(K)$ with
embedded knot $K$.}
\end{figure}

With these preliminaries in hand, we can now describe the
construction of $M$. The first step is to construct
$\widetilde{X}_\rho$, which is defined to be the (unbranched)
covering space of $X$ whose monodromy is specified by $\rho$. The
following steps construct $\widetilde{X}_\rho$.
\begin{enumerate}
\item{Take $\widetilde{H}_\rho$, the $n$--fold covering space of
$H$ whose monodromy is specified by $\rho$. Lift the surgery link in
$H$ to $\widetilde{H}_\rho$ and do surgery on that link.}
\item{Lift $A$ and $B$, the $2$--handle attaching circles, to systems
of curves $\{A_i\}_{i=1}^n$ and $\{B_i\}_{i=1}^n$ on
$\widetilde{H}_\rho$.
} \item{Attach $2$--handles to these systems of curves.}
\item{Attach a ball to each of the $n$ resulting $S^2$ boundary
components.}
\end{enumerate}

Figure \ref{F:systemlift} shows how the attaching circles and
peripheral markings lift to $\widetilde{H}_\rho$, in the special
case that $n=7$. The general case is clear from this picture.

Consider now the boundary of $\widetilde{X}_\rho$, the space we have
just constructed. Inspecting Figure \ref{F:systemlift} we observe
that it consists of $\frac{n+1}{2}$ tori:
\[
\partial\left(\widetilde{X}_\phi\right) = T_1\sqcup T_2 \sqcup
\ldots \sqcup T_{\frac{n+1}{2}}.
\]
The torus $T_1$ is marked as shown in Figure \ref{F:othertori} on
the left. Under the restriction of the covering map
$\widetilde{X}_\rho \rightarrow X$ to this boundary component, $T_1$
is a one--sheeted covering of $\partial N(K)$.
The other tori, $T_i$ where $i$ runs from $2$ to $\frac{n+1}{2}$,
are marked as shown in Figure \ref{F:othertori} on the right. These
tori give two--sheeted coverings of $\partial N(K)$.
\begin{figure}[h]
\[\scalebox{0.7}{\includegraphics{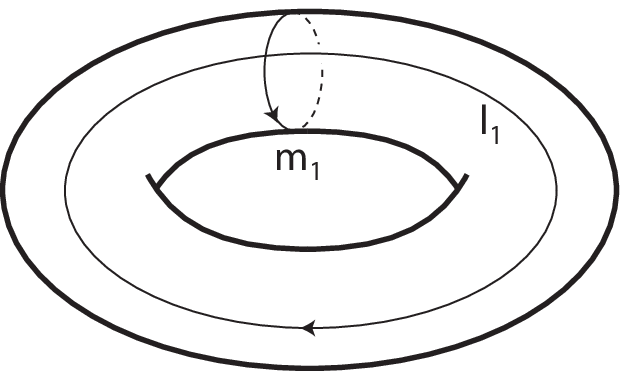}}\qquad\ \ \scalebox{0.7}{\includegraphics{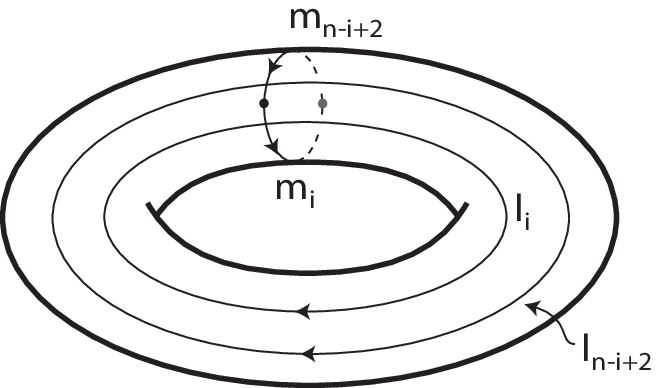}}\]
\caption{\label{F:othertori}The torus $T_1$ (on the left) and a
torus $T_i$ for some $2\leq i\leq \frac{n+1}{2}$ (on the right),
together with their markings.}
\end{figure}

The \textbf{branched} irregular dihedral covering space $M$,
together with the covering link
$\left\{\widetilde{K}_i\right\}_{i=1}^{\frac{n+1}{2}}$, is obtained
from $\widetilde{X}_\rho$ by:
\begin{enumerate}
\item{Gluing a copy of $N(K)$ into $T_1$ so as to match $m$ to
$m_1$ and $l$ to $l_1$.} \item{For each $i$ such that $2\leq i\leq
\frac{n+1}{2}$, gluing a copy of $N(K)$ into $T_i$ so as to match
$m$ to the curve $m_im_{n-i+2}$, and $l$ to either $l_i$ or
$l_{n-i+2}$.}
\end{enumerate}
This completes the construction of $M$.

\begin{figure}[t]
\psfrag{1}[c]{\LARGE$kn$}\psfrag{2}[c]{\LARGE$kn$}\psfrag{3}[c]{\LARGE$kn$}\psfrag{4}[c]{\LARGE$kn$}\psfrag{5}[c]{\LARGE$kn$}\psfrag{6}[c]{\LARGE$kn$}\psfrag{7}[c]{\LARGE$kn$}%
\psfrag{a}[c]{\large$A_1$}\psfrag{b}[c]{\large$A_2$}\psfrag{c}[c]{\large$A_3$}\psfrag{d}[c]{\large$A_4$}\psfrag{e}[c]{\large$A_5$}\psfrag{f}[c]{\large$A_6$}\psfrag{g}[c]{\large$A_7$}%
\psfrag{A}[c]{$m_1$}\psfrag{B}[c]{$m_2$}\psfrag{C}[c]{$m_3$}\psfrag{D}[c]{$m_4$}\psfrag{Z}[c]{$m_5$}\psfrag{Y}[c]{$m_6$}\psfrag{X}[c]{$m_7$}%
\psfrag{E}[c]{$I_1$}\psfrag{F}[c]{$I_2$}\psfrag{G}[c]{$I_3$}\psfrag{H}[c]{$I_4$}\psfrag{I}[c]{$I_5$}\psfrag{J}[c]{$I_6$}\psfrag{K}[c]{$I_7$}%
\scalebox{0.7}{\includegraphics{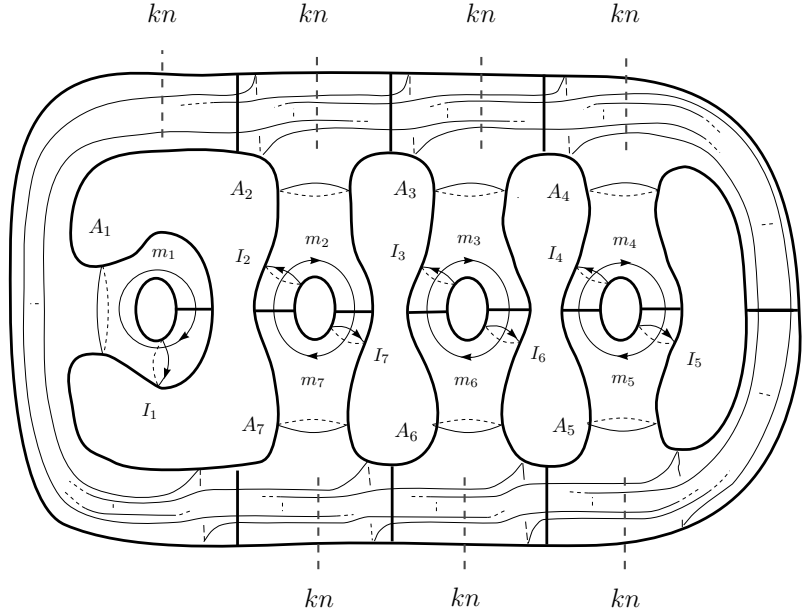}}
\caption{\label{F:systemlift} The lifts of the attaching circles and
peripheral markings to $\widetilde{H}_\rho$, in the case that
$n=7$.}
\end{figure}

Our task is to turn the construction we have just detailed into a
surgery presentation for $M$. Consider the sequence below, where the
index $i$ runs from $2$ to $\frac{n+1}{2}$, and $j=n-i+2$.

\begin{multline*}
\begin{minipage}{145pt}
\psfrag{c}[c]{$A_j$}\psfrag{f}[c]{$A_i$}\psfrag{J}[c]{$I_j$}\psfrag{G}[c]{$I_i$}\psfrag{C}[c]{$m_j$}\psfrag{Y}[c]{$m_i$}
\includegraphics[width=145pt]{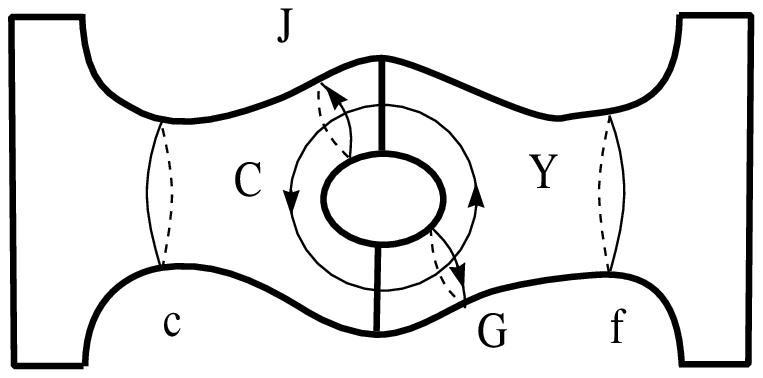}
\end{minipage}\quad\
\begin{minipage}{20pt}
\includegraphics[width=20pt]{fluffyarrow}
\end{minipage}\quad\
\begin{minipage}{145pt}
\psfrag{c}[c]{$A_j$}\psfrag{f}[c]{$A_i$}\psfrag{J}[c]{$I_j$}\psfrag{G}[c]{$I_i$}\psfrag{C}[c]{$m_j$}\psfrag{Y}[c]{$m_i$}
\includegraphics[width=145pt]{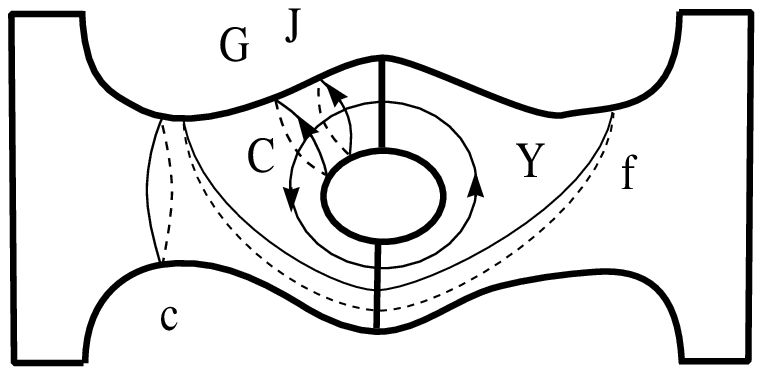}
\end{minipage}\\\ \\
\begin{minipage}{20pt}
\includegraphics[width=20pt]{fluffyarrow}
\end{minipage}\quad
\begin{minipage}{145pt}
\psfrag{c}[c]{$A_j$}\psfrag{f}[c]{}\psfrag{J}[c]{$I_j$}\psfrag{G}[c]{$I_i$}\psfrag{C}[c]{$m_j$}\psfrag{Y}[c]{$m_i$}
\includegraphics[width=145pt]{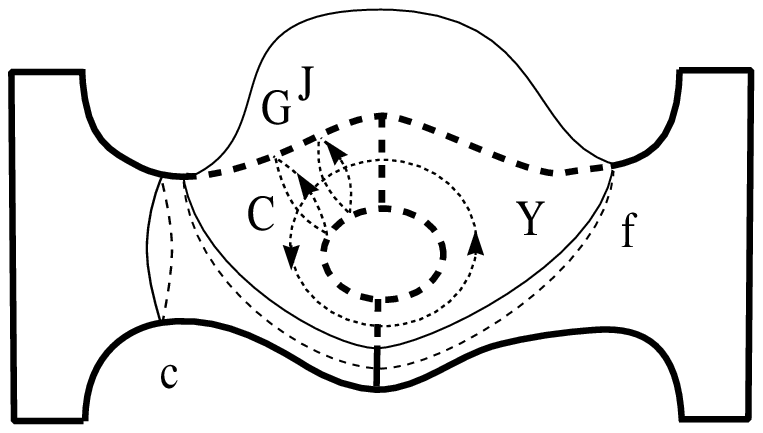}
\end{minipage}\quad\begin{minipage}{20pt}
\includegraphics[width=20pt]{fluffyarrow}
\end{minipage}\quad
\begin{minipage}{145pt}
\psfrag{c}[c]{$A_j$}\psfrag{k}[c]{$\tilde{K}_i$}
\includegraphics[width=145pt]{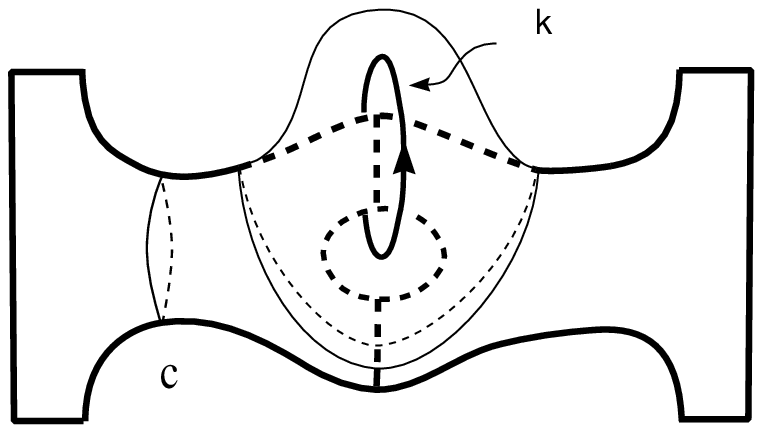}
\end{minipage}
\end{multline*}\\

The first move is to slide the attaching circle $A_i$ over the
attaching circle $A_j$. Before we do that, we'll get the longitude
marking $l_i$ out of the way by sliding it over $A_j$ first. Next we
attach a $2$--handle to $A_i$. Observe that the result can be
embedded in $S^3$, and that the torus $T_i$ is now embedded in this
diagram. Glue a copy of $N(K)$ into $T_i$ in the required way
(matching $m$ to $m_im_j$). In a similar way we can immediately
attach a $2$--handle to $A_1$ and glue a copy of $N(K)$ into $T_1$.
These are the three steps which we carried out in the sequence
above.\par


After the above sequence, if we attach $2$--handles to to the
circles $A_{\frac{n+3}{2}}$ through $A_n$ and another $2$--handle to
$B_1$, then the boundary of the space is a copy of $S^2$ (it is
connected and of genus $0$). Call this boundary $Y$. A $2$--sphere
$S^2$ can only bound a ball (Sch\"{o}nflies Theorem) so plugging $Y$
with a $3$--ball right away is the same as attaching $2$--handles to
$B_2,\ldots,B_{n}\subset Y$ and then plugging what is left of the
boundary with $3$--balls. In other words, we can discard
$B_2,\ldots,B_n$ without changing the result.\par

In the same way, we can add extra attaching circles for $2$--handles
into $Y$ without changing the result. Let's then attach $2$--handles
into $Y$ to cut the complement in $S^3$ of the handlebody into solid
tori, in the way indicated in Figure \ref{F:finalfigure}. The
attaching circles of the extra $2$--handles are labeled
$E_1,\ldots,E_{\frac{n+1}{2}}$ in the figure.

We are done. The space constructed is in the complement of a
$\frac{n+1}{2}$ component unlink in the three--sphere, and attaching
the remaining $2$--handles and balls is equivalent to doing surgery
on that unlink, in precisely the way detailed in Theorem
\ref{T:untyingmethod}.\par

To illustrate with an example, the surgery presentation for the
dihedral branched covering space and covering link for the
$D_{14}$-coloured $5_2$ knot considered in Section
\ref{SSS:fivetwoexample} is as given in Figure \ref{F:fivetwofinal}.

\begin{figure}[h]
\psfrag{1}[c]{\small$\tilde{U}_1$}\psfrag{2}[c]{\small$\tilde{U}_2$}\psfrag{3}[c]{\small$\tilde{U}_3$}\psfrag{4}[c]{\small$\tilde{U}_4$}
\psfrag{a}[c]{\Small$0$}\psfrag{b}[c]{\Small$0$}\psfrag{c}[c]{\Small$0$}\psfrag{h}[l]{framing\
$=-7$}
\includegraphics[width=250pt]{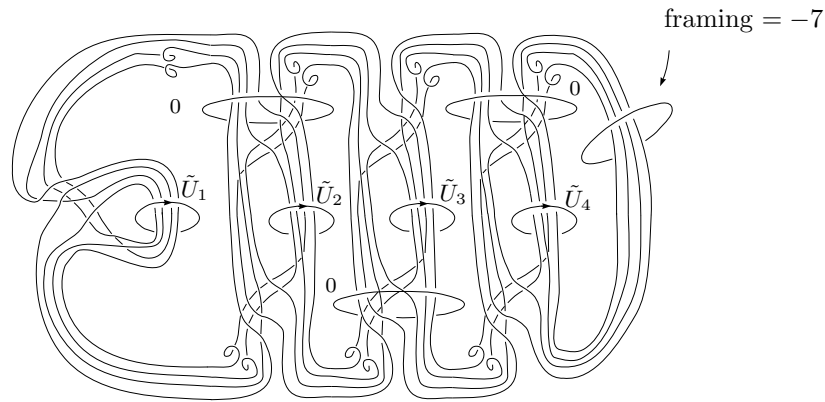}
\caption{\label{F:fivetwofinal} A surgery presentation for the
dihedral covering space and covering link of the $D_{14}$-coloured
$5_2$ knot of Section \ref{SSS:fivetwoexample}.}
\end{figure}

\begin{figure}
\psfrag{1}[c]{\LARGE$k$\
strands}\psfrag{2}[c]{\Large$A_5$}\psfrag{3}[c]{\Large$A_4$}\psfrag{a}[c]{\Large$E_1$}\psfrag{b}[c]{\Large$E_2$}\psfrag{c}[c]{\Large$B_1$}
\scalebox{0.72}{\includegraphics{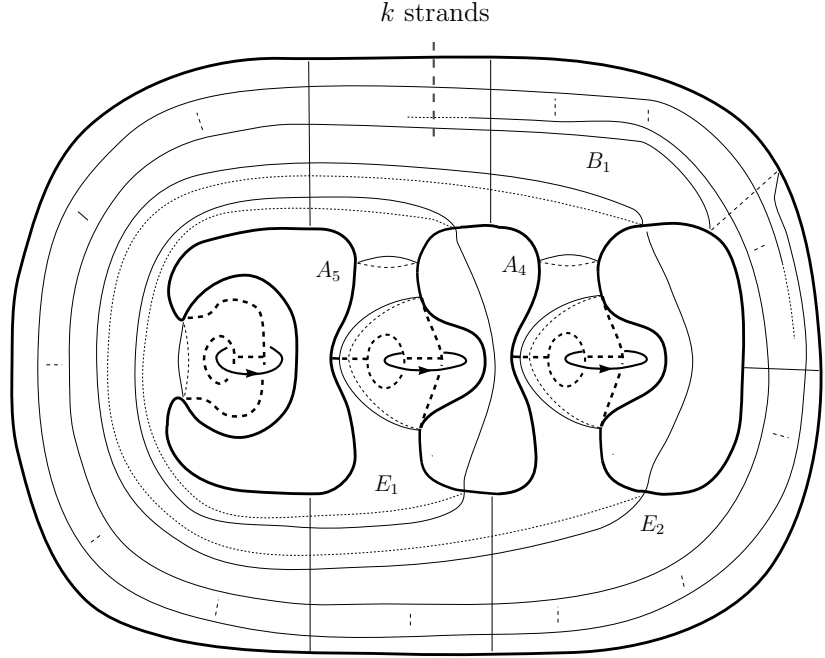}}
\caption{\label{F:finalfigure} The final diagram - after we have
discarded $B_2$ through $B_n$, and attached extra $2$--handles $E_1$
thorough $E_{\frac{n+1}{2}}$ so as to cut the complement of the
handlebody into solid tori. }
\end{figure}

\newpage
\section{Band projection approach}\label{S:Seifertapproach}

In this approach we obtain a surgery presentation of a
$\Dn$-coloured knot $(K,\rho)$ as a link $L$ consisting of
$\pm1$--framed unknotted surgery components in $\ker\rho$ which live
in the complement of an element in a complete set of base-knots in
$S^3$. We then lift this presentation to a surgery presentation of
the branched dihedral cover $M$. A by-product of this approach is a
proof of a conjecture that two $\Dn$-coloured knots are
$\rho$--equivalent if and only if they have the same coloured
untying invariant.

\subsection{Tools}
\subsubsection{The surface data}\label{SSS:ColouringVector}
Take a $\Dn$-coloured knot $(K,\rho)$ and choose $F$ a Seifert
surface for $K$. Let $x_{1},\ldots,x_{2g}$ be a basis of $H_{1}(F)$
and let $(\xi_{1},\ldots,\xi_{2g})$ be the associated basis for
$H_{1}(S^{3}-F)$ uniquely characterized by the condition that
$\mathrm{Link}(x_i,\xi_j)=\delta_{ij}$ (see \textit{e.g.}
\cite[Definition 13.2]{BZ03}). A curve representing $\xi_i$
under-crosses an even number of arcs of the knot diagram, therefore
the representative $\tilde{\xi}_i$ of $\xi_i$ in
$\pi_1\left(S^3-K\right)$ is mapped by $\rho$ to the product of an
even number of reflections, \textit{i.e.} a rotation
$s^a\in\Cn\subset \Dn$. This image is independent of which
representative $\tilde{\xi}_i$ of $\xi_i$ we chose because $\Cn$ is
commutative. Let the \emph{colouring vector} of $(K,\rho)$
associated to the basis $\set{x_1,\ldots,x_{2g}}$ of $H_1(F)$ be the
vector of these images
$$
\vec{v}\ass (v_1,\ldots,v_{2g})^T \ass
\left(\rho(\xi_1),\ldots,\rho(\xi_{2g})\right)^T\in
\left(\Cn\right)^{2g}.
$$
\noindent The colouring vector determines $\rho$ restricted to
$\pi_1(S^3-F)$, and so, via the HNN construction over the Seifert
surface, determines $\rho$ up to an inner automorphism of $\Dn$ (the
details of the construction are recalled in the proof of Lemma
\ref{L:VMV} below). Actually, a small trick shows that every such
inner automorphism is realized by an isotopy of the diagram, so the
colouring vector determines $\rho$ uniquely \cite[Proof of Lemma
4]{Mos06b}.\par

Let $\tau^{\pm}$ denote the pushoff from $F$ in the direction of its
positive (negative) normal (as determined by the orientation of the
knot), and let $\M=\left(\mathrm{Link}(\tau^- x_i,x_j)\right)_{1\leq
i,j\leq 2g}$ be a Seifert matrix for $K$ with respect to
$\left\{x_1,\ldots,x_{2g}\right\}$.

We call the pair $(\M,\vec{v})$ the \emph{surface data} for the
$\Dn$-coloured knot $(K,\rho)$ corresponding to a choice of Seifert
surface $F$ and a choice of basis for $H_1(F)$. The surface data
satisfies the following property:

\begin{lem}\label{L:VMV}
Let $\vec{w}\ass \left(w_1,\ldots,w_{2g}\right)^T\in\mathds{Z}^{2g}$
be a vector of integers satisfying $v_i= s^{w_i}$. For any vector of
integers $\vec{z}\ass (z_1,\ldots,z_{2g})^T\in\mathds{Z}^{2g}$ we
have
\begin{equation*}
\vec{z}^{\ T}\cdot (\M+\M^T)\cdot \vec{w} \equiv0\bmod n,
\end{equation*}
\noindent and in particular
\begin{equation*}
\vec{w}^{\ T}\cdot \M\cdot \vec{w} \equiv0\bmod n.
\end{equation*}
\end{lem}

\begin{proof}
The proof is essentially the same as \cite[proof of Proposition
1.1]{CS84}. Because $(\xi_{1},\ldots,\xi_{2g})$ is a basis for
$H_{1}(S^{3}-F)$ and because $\Cn$ is abelian, the colouring vector
$\vec{v}$ determines the map $\bar{\rho}\co \pi_1(S^3-F)\to \Cn$
induced by $\rho$ via the condition
$\bar{\rho}(y)=s^{\mathrm{Link}(y,\alpha)}$, where
\begin{equation*}
\alpha\ass \sum_{i=1}^{2g}w_{i}x_i.
\end{equation*}

The extension of $\pi_1(S^3-F)$ to $\pi_1(S^3-K)$ is given by adding
a generator $m$ corresponding to a choice of meridian of $K$, modulo
the relation
$$
m\cdot \tau^+ z\cdot m^{-1}= \tau^{-}z
$$
\noindent for all $z\in \pi_1(F)$, corresponding to the fact that
the path $m\cdot z\cdot m^{-1}z^{-1}$ is contractible in
$\pi_1(S^3-K)$ (the HNN construction).\par

Because we know that $\bar{\rho}$ extends to $\rho$ and that
$\rho(m)$ is a reflection, it follows that
$$
\rho(\tau^- z)=\rho(m\cdot \tau^+ z\cdot m^{-1})=
\rho(m)\cdot\rho(\tau^+ z)\cdot\rho(m)= \rho(-\tau^+ z)
$$

\noindent Therefore

$$\mathrm{Link}(\tau^+ z,\alpha)= \mathrm{Link}(-\tau^-
z,\alpha) \bmod n$$

 \noindent The term on the right equals
$-\mathrm{Link}(\tau^+\alpha,z)$. Therefore
$$
z\cdot (L_\M+L_{\M^T})\cdot\alpha= 0\bmod n
$$
where $L_\M$ and $L_{\M^T}$ are the linking pairings of $\M$ and of
$\M^T$ correspondingly in $S^3$. Setting $z=\sum_{i=1}^{2g}z_{i}x_i$
gives
\begin{equation*}
\vec{z}^{\ T}\cdot(\M+\M^T)\cdot \vec{w}\equiv 0\bmod n
\end{equation*}
\noindent
and setting $z=\alpha$ gives
\begin{equation*}
\vec{w}^{\ T}\cdot(\M+\M^T)\cdot \vec{w}=2\vec{w}^{\ T}\cdot \M\cdot
\vec{w} \equiv0\bmod n.
\end{equation*}
\end{proof}

\begin{rem}
The vectors $\vec{w}$ and $\vec{w}\bmod n$ are called the
$p$--colouring vector in \cite{LiWal08} and in \cite{Mos06b}
respectively. When $\alpha$ is represented by a simple closed curve,
that curve is called a \emph{mod $p$ characteristic knot of
$(K,\rho)$} in \cite{CS84}.
\end{rem}

\subsubsection{The coloured untying invariant}

In \cite[Section 6]{Mos06b} it was shown that the following
expression

$$
\mathrm{cu}(K,\rho)=\frac{2(\vec{w}^{\thinspace T}\cdot \M\cdot
\vec{w})}{n}\bmod n
$$
\noindent depends neither on the choice of Seifert surface $F$ nor
on the choice of basis for $H_1(F)$. Hence it is an invariant of
$\Dn$-coloured knots. It is also shown that this is a non-trivial
$\pZ$--valued invariant of $\Dn$-coloured knots in $S^3$ which is
constant on $\rho$--equivalence classes. A homological version of
this invariant seems to provide a generalization to $\Dn$-coloured
knots in more general $3$--manifolds \cite{Mos06b,LiWal08}.\par

The culmination of this section is to show that \textbf{two knots
are $\rho$--equivalent if and only if they have the same untying
invariant.}

\subsubsection{Band Projection}
Any knot has a \emph{band projection} (see for instance
\cite[Proposition 8.2]{BZ03}). This is a projection of the following
form: 


\begin{figure}[htbp]
\psfrag{1}[c]{$B_{1}$}\psfrag{2}[c]{$B_{2}$}\psfrag{3}[c]{$B_3$}
\psfrag{4}[c]{$B_{4}$}\psfrag{5}[c]{$B_{2g-1}$}\psfrag{6}[c]{$B_{2g}$}\psfrag{D}{\large$D^2$}
\psfrag{d}[c]{\Huge$\mathbf{\ldots}$}\psfrag{x}{$x_{1}\quad$}\psfrag{c}{$\xi_{1}$}\psfrag{e}{$\xi_2$}\psfrag{y}{$x_2$}
\includegraphics[width=4.5in]{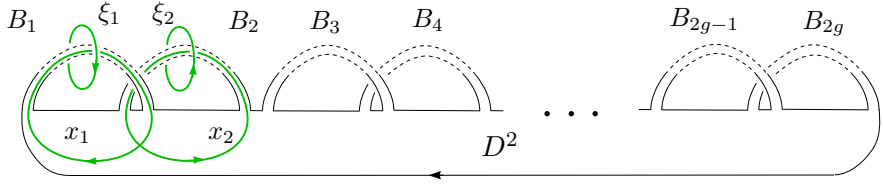}
\caption{A band projection of a knot.} \label{F:bandproj}
\end{figure}

Pairs of bands $B_{2i-1}$ and $B_{2i}$ for $i=1,\ldots,g$ will be
called \emph{twin bands}. We may choose a band projection such that
that knot is oriented as shown in the figure.\par

A knot in band projection comes equipped with a canonical choice of
a Seifert surface $F$ and a choice of basis for $H_{1}(F)$: let
$x_{1},\ldots,x_{2g}$ be elements of $H_1(F)$ such that for each
$1\leq i\leq 2g$ the class $x_i$ is represented by a curve in $F$
which threads once through the band $B_i$, with orientations as
determined by Figure \ref{F:bandproj}. Recall that the associated
basis $\xi_{1},\ldots,\xi_{2g}$  for $H_1(S^3-F)$ is determined by
the condition $\mathrm{Link}(x_i,\xi_j)=\delta_{ij}$. In this case
the class $\xi_i$ is represented by the appropriately oriented
boundary of a small disc which the band intersects the interior of
transversely as shown in Figure \ref{F:bandproj}.

The surface data of a knot in band projection refers to the Seifert
matrix and colouring vector for this canonical choice of basis.

\subsubsection{Band slides}
At the heart of this approach are moves which allow us to realize
algebraic manipulations of the surface data by ambient isotopies
which modify the choice of band projection of a fixed $\Dn$-coloured
knot.

We say that some band projection is obtained from another by doing a
\emph{band slide} of band $B_{2i-1}$ counterclockwise over band
$B_{2i}$ if it is obtained by the following sequence of ambient
isotopies:
\begin{equation*}\label{E:bandslide}
\begin{minipage}{80pt}
\includegraphics[width=80pt]{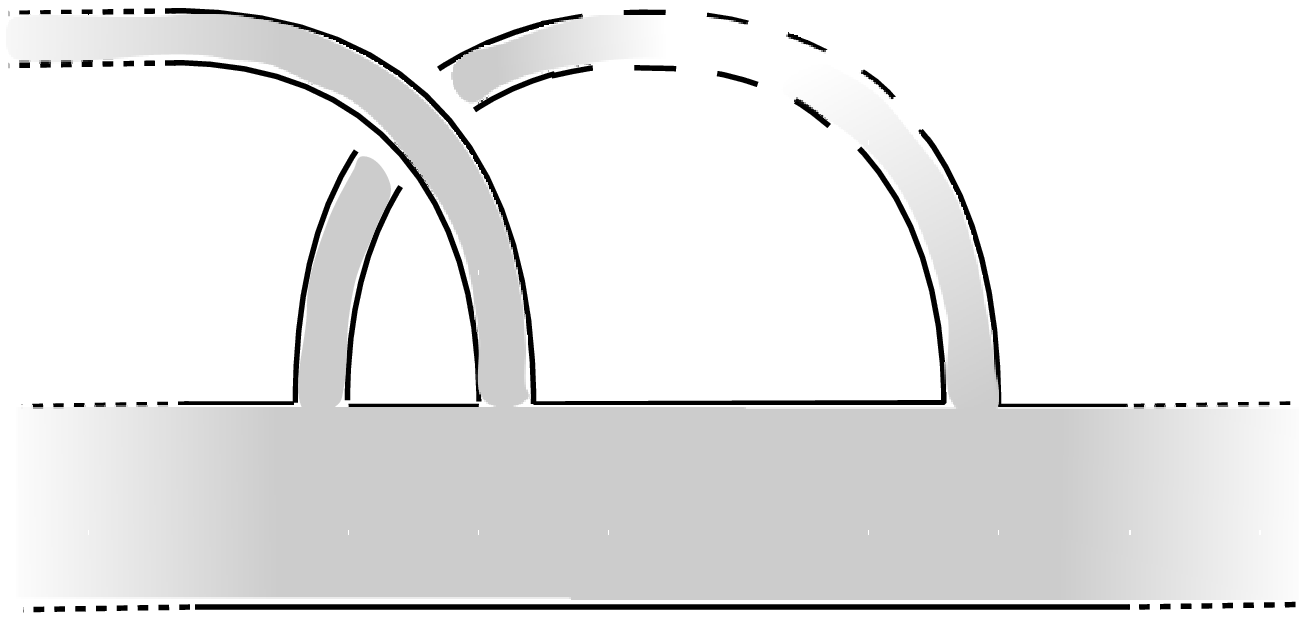}
\end{minipage}\quad\Too\hspace{7pt}
\begin{minipage}{80pt}
\includegraphics[width=80pt]{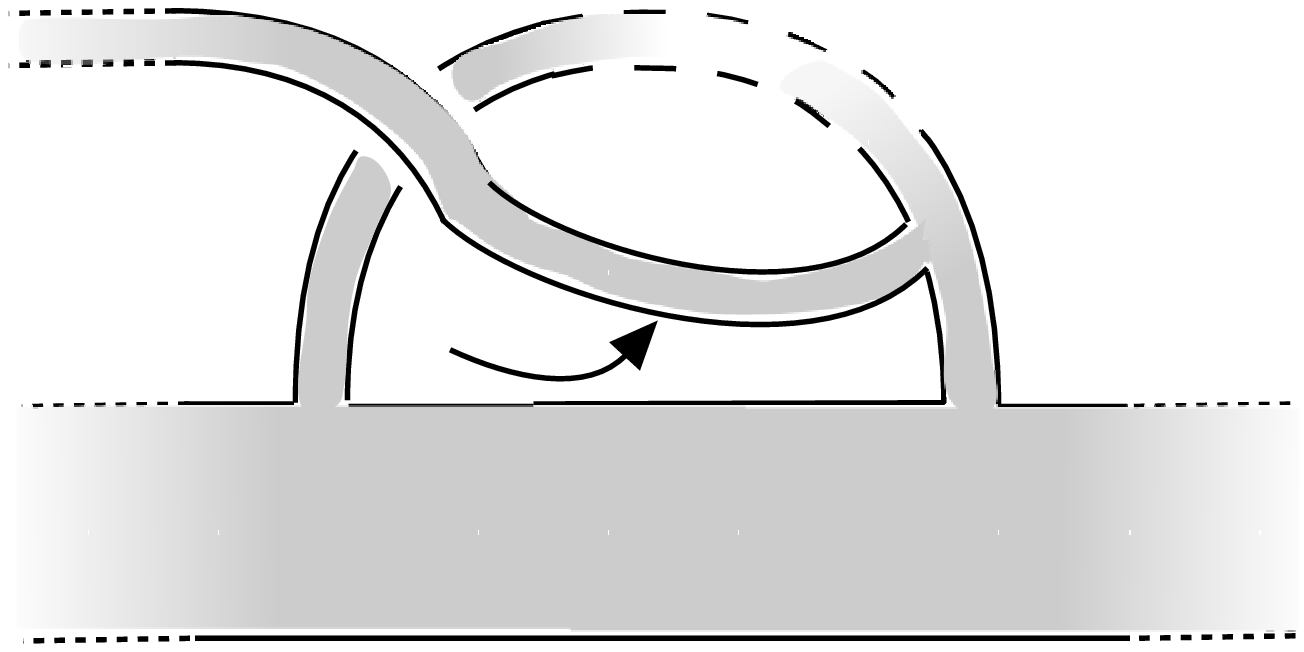}
\end{minipage}\quad\Too\hspace{7pt}
\begin{minipage}{80pt}
\includegraphics[width=80pt]{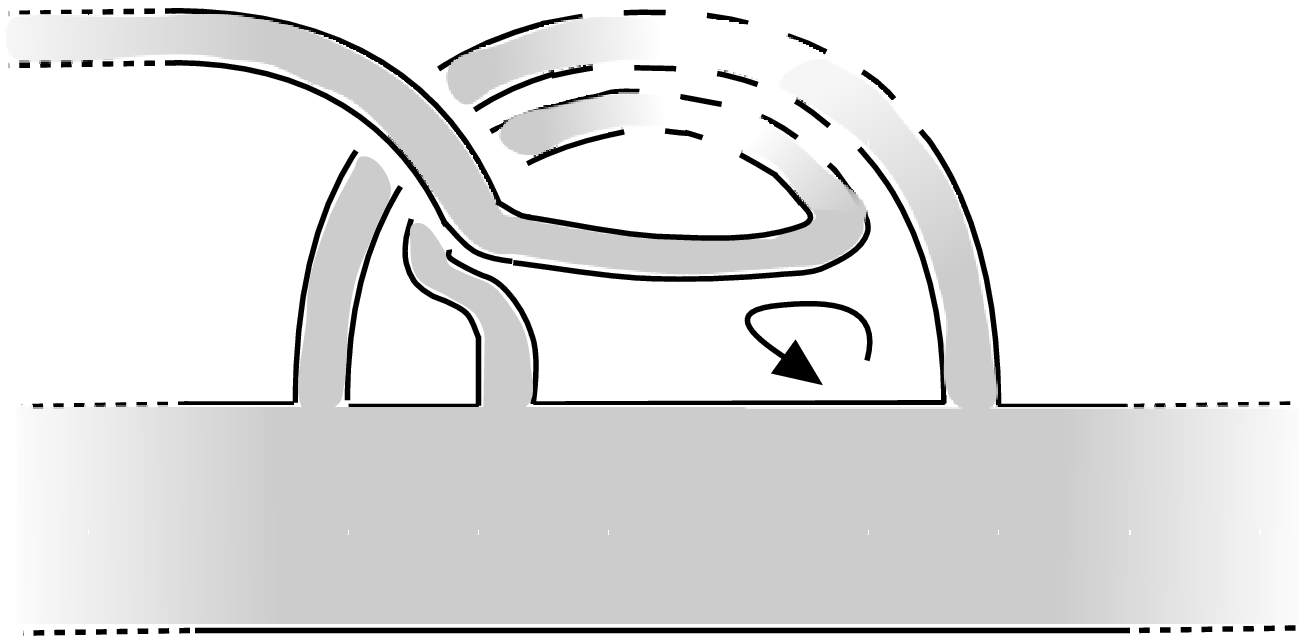}
\end{minipage}
\end{equation*}
\noindent Similarly we can slide $B_{2i-1}$ clockwise over $B_{2i}$,
and we can slide $B_{2i}$ over $B_{2i-1}$ both clockwise and
counterclockwise.\par

These moves fix $F$ but change the choice of basis for $H_1(F)$, and
so will change the surface data. The effect on the choice of basis
is:
\begin{itemize}
\item Sliding $B_{2i-1}$ counterclockwise (respectively clockwise) over $B_{2i}$:
$$\left(x_{1},\ldots,x_{2i-1},x_{2i},\ldots,x_{2g}\right)\mapsto
\left(x_{1},\ldots,x_{2i-1},x_{2i}\pm
x_{2i-1},\ldots,x_{2g}\right)$$

\item Sliding $B_{2i}$ counterclockwise (respectively clockwise) over $B_{2i-1}$:
$$\left(x_{1},\ldots,x_{2i-1},x_{2i},\ldots,x_{2g}\right)\mapsto
\left(x_{1},\ldots,x_{2i-1}\pm x_{2i},x_{2i},\ldots,x_{2g}\right)$$
\end{itemize}

\noindent And the corresponding effect on the colouring vector is as
follows:

\begin{itemize}
\item Sliding $B_{2i-1}$ counterclockwise (respectively clockwise) over
$B_{2i}$:
$$\left(v_{1},\ldots,v_{2i-1},v_{2i},\ldots,v_{2g}\right)\mapsto
\left(v_{1},\ldots,v_{2i-1},v_{2i}\cdot
v_{2i-1}^{\mp1},\ldots,v_{2g}\right)$$
\item Sliding $B_{2i}$ counterclockwise (respectively clockwise) over $B_{2i-1}$:
$$\left(v_{1},\ldots,v_{2i-1},v_{2i},\ldots,v_{2g}\right)\mapsto
\left(v_{1},\ldots,v_{2i-1}\cdot
v_{2i}^{\mp1},v_{2i},\ldots,v_{2g}\right)$$
\end{itemize}

\noindent The corresponding effects on the Seifert matrix are
$\M\mapsto \left(P^{\pm}_{(2i-1,2i)}\right) \M
\left(P^{\pm}_{(2i-1,2i)}\right)^T$ and $\M\mapsto
\left(P^{\pm}_{(2i,2i-1)}\right) \M
\left(P^{\pm}_{(2i,2i-1)}\right)^T$ for $P^{\pm}_{j,k}\ass I\pm
E_{j,k}$.

\begin{example}
Let $(K,\rho)$ be a $\Dn$-coloured genus one knot for which

\begin{equation}
(\M,\vec{v})=\
\genusoneknot{a_{11}}{a_{12}}{a_{21}}{a_{22}}{v_1}{v_2}
\end{equation}
\noindent with respect to a given basis of $H_1(F)$. The effect of
  band sliding $B_1$ over $B_2$ counterclockwise is as follows:

\begin{equation}
\left(\M,\vec{v}\right)\mapsto
\genusoneknot{a_{11}+a_{12}+a_{21}+a_{22}}{a_{12}+a_{22}}{a_{21}+a_{22}}{a_{22}}{v_1}{v_2\cdot
v_1^{-1}}
\end{equation}
\end{example}

The following two lemmas are crucial in this approach. They show how
much freedom band slides give us to engineer the colouring vector.

\begin{lem}\label{L:slidingtwins}
For twin bands $B_{2i-1}$ and $B_{2i}$ for which either $v_{2i-1}$
or $v_{2i}$ generates $\Cn$,   band slides allow us to transform the
pair $(v_{2i-1},v_{2i})$ to any other pair
$(v_{2i-1}^\prime,v_{2i}^\prime)$  for which either
$v_{2i-1}^\prime$ or $v_{2i}^\prime$ generates $\Cn$.
\end{lem}

\begin{proof}
Assume without the limitation of generality that $v_{2i-1}$
generates $\Cn$. Then by sliding $B_{2i-1}$ over $B_{2i}$ an
appropriate number of times, we can transform $v_{2i}$ into a
generator of $\Cn$ (in fact into any element of $\Cn$). We can
therefore assume that both $v_{2i-1}$ and $v_{2i}$ generate $\Cn$.
Symmetrically, we can assume that both $v_{2i-1}^\prime$ and
$v_{2i}^\prime$ generate $\Cn$. Slide $B_{2i-1}$ over $B_{2i}$ until
the corresponding entry in the colouring vector becomes
$v_{2i}^\prime$ and then sliding $B_{2i}$ over $B_{2i-1}$ until the
corresponding entry in the colouring vector becomes
$v_{2i-1}^\prime$.
\end{proof}

\begin{lem}\label{L:killevenentries}
For any pair of twin bands $B_{2i-1}$ and $B_{2i}$, by band slides
we can obtain a band projection which induces a colouring vector
such that either $v_{2i-1}$ vanishes, or $v_{2i}$ vanishes, as
desired.
\end{lem}

\begin{proof}
Equip $\Cn$ with a total ordering, and for each $i=1,\ldots,g$ slide
$B_{2i}$ over $B_{2i-1}$ if $v_{2i-1}\geq v_{2i}$, or $B_{2i-1}$
over $B_{2i}$ otherwise. We obtain a pair
$(v_{2i-1}^\prime,v_{2i}^\prime)$ which is smaller than
$(v_{2i-1},v_{2i})$ in the lexical ordering induced by the ordering
on $\Cn$. Repeat until we kill $v_{2i}$ (in which case we're
finished) or $v_{2i-1}$.\par

Now that we have obtained a colouring vector with $v_{2i-1}=1$, if
we want a colouring vector with $v_{2i}=1$, exchange the $2i$'th and
the $(2i-1)$'st entries in the colouring vector by a sequence of
band slides corresponding to the following operations on entries of
the colouring vector:
$$(1,v_{2i})\to (v_{2i},v_{2i})\to (v_{2i},1)$$
\noindent Analogously, if $v_{2i}=1$ and we want $v_{2i-1}$ to
vanish, we can reverse the sequence of band slides above.
\end{proof}

\subsection{Reduction of Genus}
The goal of this section is to show that any $\Dn$-coloured knot
$(K,\rho)$ of genus $g$ is $\rho$--equivalent to a $\Dn$-coloured
knot of genus $1$. The proof consists of three steps. We first show
that for any $(K,\rho)$ we may choose a band projection such that
the induced colouring vector has its first entry equal to $s$. The
second step is to arrange every other entry to be $1$. Having
prepared such a band projection the final step is to reduce genus by
$\rho$--equivalences.

\subsubsection{Step 1: Engineer a band projection such that $v_1=s$.}\label{SSS:Step1}

If $n$ is prime, engineering a band projection such that $v_1=s$ is
straightforward (Lemma \ref{L:slidingtwins}), and one may proceed
directly to Step 2. If $n$ is composite, however, a more involved
argument may be required.\par

Our strategy is to construct the desired band projection directly,
by finding an appropriate \textit{cut system}. For the purpose of
this section's discussion we'll formalize a few terms.

\begin{defn}
A \emph{cut} on some Seifert surface for some knot $K$ is a simple
non-separating oriented curve lying on the surface whose two
boundary points lie on $K$.
\end{defn}

\begin{defn}
Consider some cut $C$ on some Seifert surface $F$. The \emph{ring
around $C$} is a particular simple closed oriented curve in $S^3-F$,
constructed in the following way. Seifert surfaces are bi-collared,
so we may thicken $F$ in $S^3$ to $F\times [-1,1]$. The original
surface $F$ is regarded as occupying the $0$--slice of this
cylinder. Let the boundary points of $C$ be $C_0$ and $C_1$ (so that
$C$ runs from $C_0$ to $C_1$). The ring around $C$ is now the loop
which starts at $C_0\times \{1\}$, follows the curve $C\times 1$ to
$C_1\times \{1\}$, loops around $K$ to $C_1 \times \{-1\}$ via the
path $\gamma_1$ shown in Figure \ref{F:finaldiag}, returns along
$C\times \{-1\}$ to $C_0 \times \{-1\}$, then loops back around $K$
to its starting point using the obvious path $\gamma_0$.
\end{defn}

\begin{figure}
\psfrag{k}[c]{$K$}\psfrag{f}[c]{$F$}
\begin{minipage}{120pt}
\psfrag{c}[c]{$C$}
\includegraphics[width=120pt]{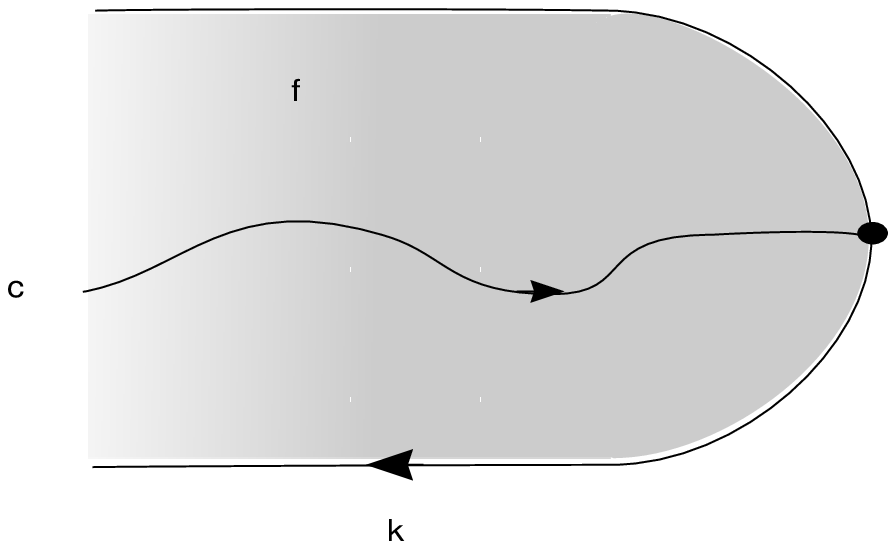}
\end{minipage}
\quad\
\begin{minipage}{23pt}
\includegraphics[width=23pt]{fluffyarrow}
\end{minipage}\qquad\qquad
\begin{minipage}{150pt}
\psfrag{g}[c]{$\gamma_1$}\psfrag{0}[r]{\SMALL$C\times\{-1\}$}\psfrag{1}[r]{\SMALL$C\times\{1\}$}\psfrag{2}[l]{\SMALL$C_1\times\{1\}$}\psfrag{3}[l]{\SMALL$C_1\times\{-1\}$}
\includegraphics[width=150pt]{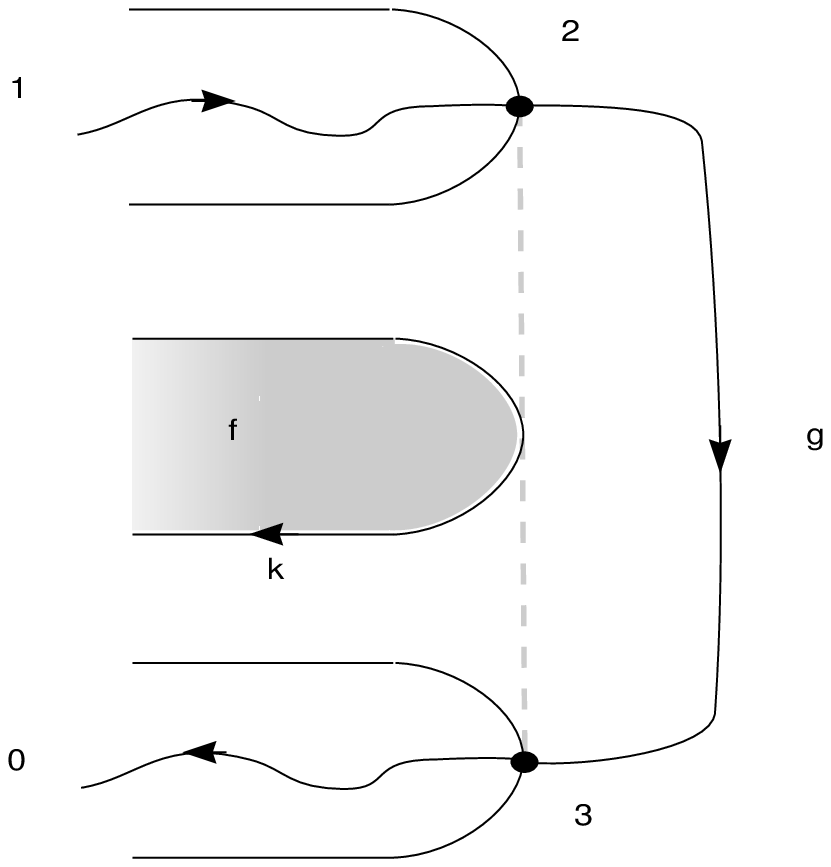}
\end{minipage}
\caption{\label{F:finaldiag}The path $\gamma_1$.}
\end{figure}

So given a cut on a Seifert surface, we may take the ring around it,
which now evaluates in the representation $\rho$ to give a
well-defined element of $C_n$.\par

The constructions which resolve this step can now be described by
the following two lemmas.

\begin{lem}\label{L:AndrewCut1}
Consider a $D_{2n}$-coloured knot $(K,\rho)$, and a Seifert surface
$F$ for $K$. If there exists a cut $C$ on the surface whose
corresponding ring evaluates to $s$, then the knot has a band
projection whose corresponding colouring vector has its first entry,
$v_1$, equal to $s$.
\end{lem}

\begin{lem}\label{L:AndrewCut2}
Every Seifert surface $F$ of a $D_{2n}$-coloured knot $(K,\rho)$ has
a cut on it whose corresponding ring evaluates to $s$.
\end{lem}

We'll explain the proofs of these lemmas in turn.\par

\begin{proof}[Proof of Lemma \ref{L:AndrewCut1}]
This proof is essentially a re-reading of the standard manipulations
that show that every Seifert surface has a band projection (see
\textit{e.g.} \cite[Chapter 6]{ST34}).\par

A system of cuts on $F$, $C_1$ through $C_{2g}$, is called a
\emph{cut system} if when we remove the bands coming from the
regular neighbourhoods of the cuts, we are left with a disc. If we
have a cut system on $F$, then the disc that remains after we remove
the bands from it has its boundary marked with $2g$ pairs of
intervals, corresponding to the two sides that are created when an
arc is cut open. Label these intervals using $B_1$ through $B_{2g}$,
say, depending on which cut an interval came from. (So, in
particular, each label will appear twice.) If we have chosen our
cuts so that these labels appear in the usual ``product of
commutators'' order, then an ambient isotopy which takes this disc
into a standard unknotted disc position will carry the original
Seifert surface into standard band-projection position. Furthermore,
that ambient isotopy will also carry the rings around the cuts to
the rings around the ``standard'' cuts of a knot in band projection
(see Figure \ref{F:bandcuts}), which are the usual $\xi_i$'s.\par

\begin{figure}
\psfrag{1}[c]{$B_{1}$}\psfrag{2}[c]{$B_{2}$}\psfrag{3}[c]{$B_3$}
\psfrag{4}[c]{$B_{4}$}\psfrag{5}[c]{$B_{2g-1}$}\psfrag{6}[c]{$B_{2g}$}\psfrag{D}{\large$D^2$}
\psfrag{d}[c]{\Huge$\mathbf{\ldots}$}
\includegraphics[width=4.5in]{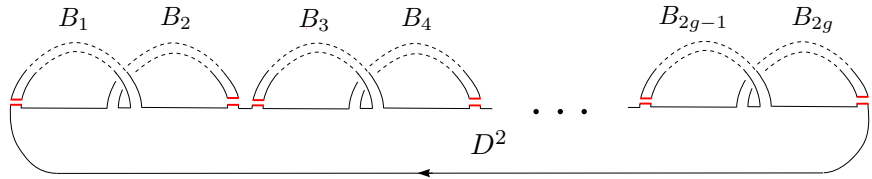}
\caption{\label{F:bandcuts} The ``standard'' cuts of a band
projection.}
\end{figure}

So our only task is to show that any given cut $C_1$ may be
completed to a cut system, $C_1$ through $C_{2g}$, marking the disc
in the desired ``product of commutators'' order. This is a standard
manipulation.
\end{proof}

\begin{proof}[Proof of Lemma \ref{L:AndrewCut2}] Begin with any band
projection of the given $D_{2n}$-coloured knot. Using Lemma
\ref{L:killevenentries} kill even numbered entries in the colouring
vector by band slides. \par

Next we'll introduce the collection of cuts amongst which we'll find
our desired cut. To every vector $(a_1,\ldots,a_g)\in \mathbb{Z}^g$
associate a cut in the way illustrated by Figure \ref{F:Andrewcut}.

\begin{figure}[h]
$$
\begin{minipage}{3in}
\includegraphics[width=3in]{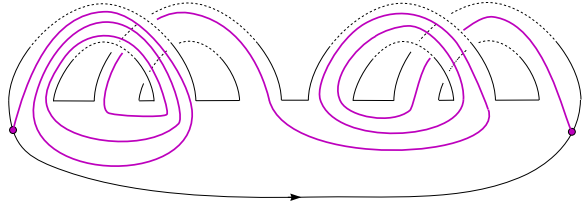}
\end{minipage}
$$
\caption{\label{F:Andrewcut}The cut for $g=2$ and
$(a_1,a_2)=(3,-2)$}
\end{figure}

We claim that we can choose the vector $(a_1,\ldots,a_g)$ so that
the ring around the corresponding cut evaluates under $\rho$ to
$s$.\par

Our next task, then, is to determine how the image under $\rho$ of
the ring around one of these cuts depends on the given vector. Well,
observe that this ring is homologous in $H_1(S^3-F)$ to
$$
\sum_{i=1}^g\left(a_i(\tau^+x_{2i-1}-\tau^-x_{2i-1})+(\tau^-x_{2i}-\tau^+x_{2i})\right),
$$
where, recall, $\tau^{\pm}x$ denotes the push-off from the Seifert
surface of a curve $x$ in the positive (resp. negative) direction.
Furthermore, note for all $i$ that $\tau^+x_{2i-1}-\tau^-x_{2i-1}$
is homologous to $\xi_{2i}$ and $\tau^-x_{2i}-\tau^+x_{2i}$ is
homologous to $\xi_{2i-1}$. Thus the ring around the cut
corresponding to the vector $(a_1,\ldots,a_g)$ evaluates under
$\rho$ to
$$
\left(v_{2}\right)^{a_1}\left(v_{4}\right)^{a_2}\ldots\left(v_{2g}\right)^{a_g}.
$$

To finish the proof we ask: can we choose the vector
$(a_1,\ldots,a_g)$ so that this expression evaluates to $s$? The
answer is yes, because we assumed that $\rho$ was surjective. (Here
are some quick details: Because $\rho$ is surjective, there will be
some curve $\psi$ in the complement of $K$ mapping to $s$. Note that
it will have to link $K$ an even number of times. So we can write
$\psi$ as some product
$$
\gamma^{k_1}\psi_1\gamma^{k_2}\psi_2\ldots\gamma^{k_j}\psi_j,
$$
where $\gamma$ is some fixed loop based at the base-point $\star$
which intersects the Seifert surface exactly once, in the positive
direction, where each $\psi_i$ is a loop based at $\star$ in the
complement of the Seifert surface, and where $\sum_i k_i = 0$.\par

Each of these factors $\psi_i$ is mapped under $\rho$ to the element
of $\Cn$ given by the formula:
$$
\rho(\psi_i) =
\left(v_{2}\right)^{\mathrm{Link}(\psi_i,x_2)}\left(v_{4}\right)^{\mathrm{Link}(\psi_i,x_4)}\ldots\left(v_{2g}\right)^{\mathrm{Link}(\psi_i,x_{2g})}.
$$
So $\rho$ of the above product of curves, which equals $s$ by the
choice of $\psi$, gives the desired expression for $s$.)
\end{proof}

\subsubsection{Step 2: Kill $v_i$ for $i>1$}\label{SSS:Step2}

First, kill $v_{2i}$ for $i=1,\ldots,g$ by Lemma
\ref{L:killevenentries} (note that this leaves $v_1$ untouched
because $s$ generates $\Cn$). If $v_3=s^a$ and if $a>0$, first
exchange $v_1$ and $v_2$ by band slides using Lemma
\ref{L:slidingtwins}, then slide bands as follows:

\begin{gather*}
\scriptsize \psfrag{a}[r]{$1$}\psfrag{b}{$s$}
\begin{minipage}{160pt}
\psfrag{c}{$s^a$}\psfrag{d}{$1$}
\includegraphics[width=160pt]{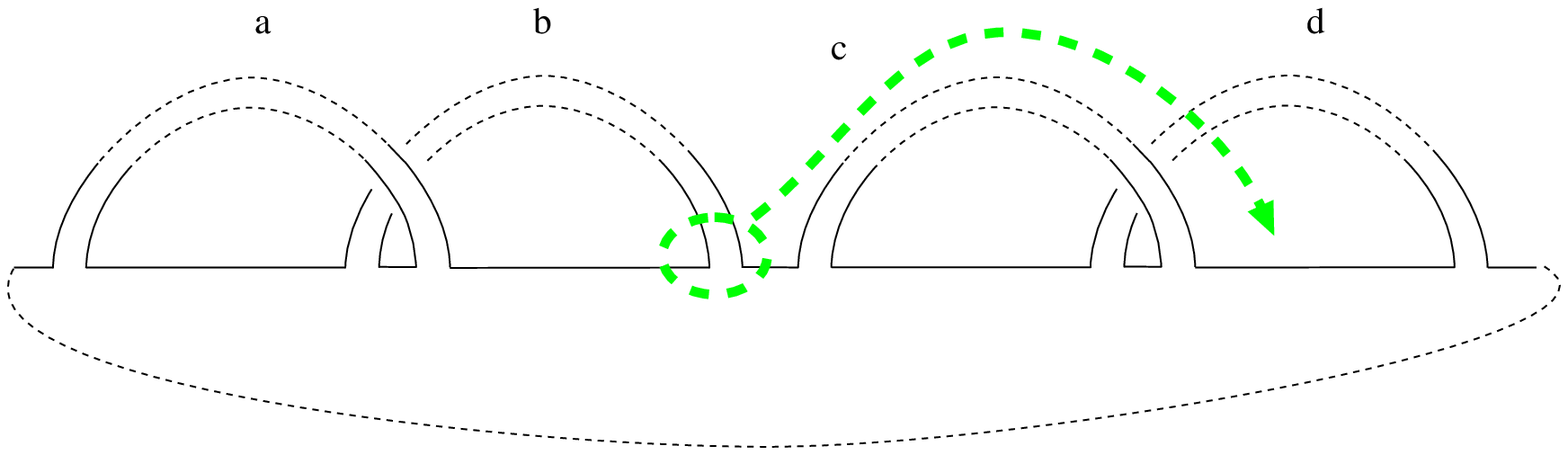}
\end{minipage}\quad \begin{minipage}{23pt}
\includegraphics[width=23pt]{fluffyarrow}
\end{minipage}\quad
\begin{minipage}{160pt}
\includegraphics[width=160pt]{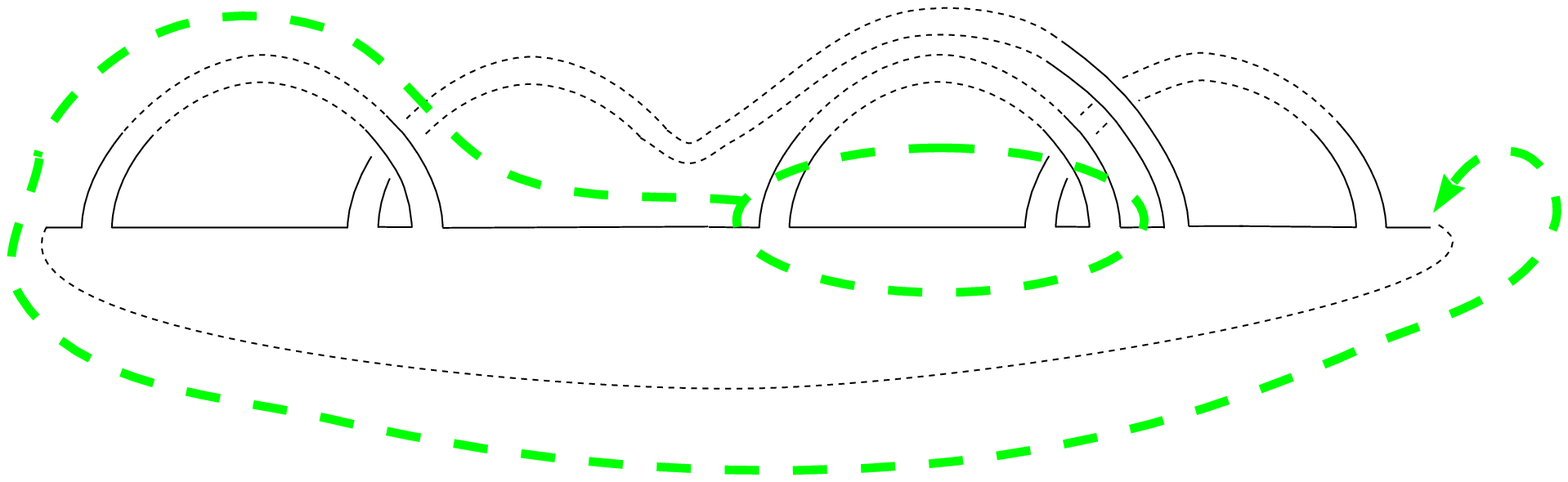}
\end{minipage}\\
\begin{minipage}{23pt}
\includegraphics[width=23pt]{fluffyarrow}
\end{minipage}\qquad
\begin{minipage}{160pt}\scriptsize
\psfrag{a}[r]{$1$}\psfrag{b}{$s$}\psfrag{d}{$s^{a+1}$}\psfrag{c}{$1$}
\includegraphics[width=160pt]{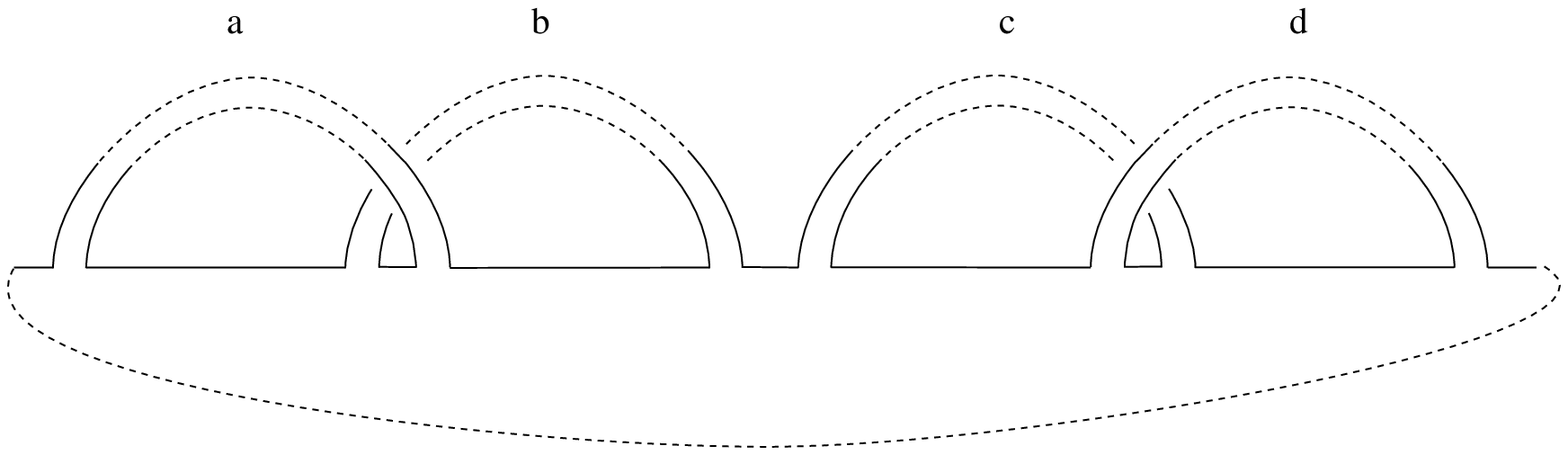}
\end{minipage}\rule{100pt}{0pt}
\end{gather*}

The second arrow is obtained by sliding both attaching segments of
$B_{3}$ and the left attaching segment of $B_{4}$ all the way around
the knot counterclockwise. This does not effect the colouring vector
because $v_{4}=1$.\par

Repeat the above steps $n-a$ times. After this step (if we switch
back $v_1$ and $v_2$), $v_3$ which has been killed while the rest of
the colouring vector has been unchanged. Now slide $B_3$ and $B_4$
over $B_5$ and $B_6$:

\begin{multline*}
\begin{minipage}{175pt}
\includegraphics[width=175pt]{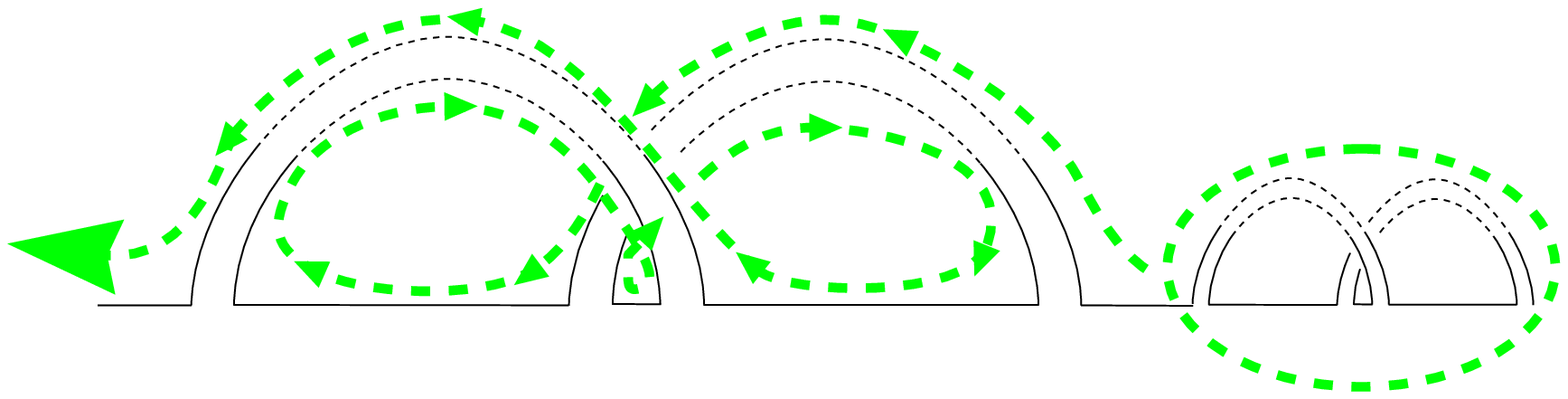}
\end{minipage}\ \ \ \begin{minipage}{23pt}
\includegraphics[width=23pt]{fluffyarrow}
\end{minipage}\quad
\begin{minipage}{135pt}
\includegraphics[width=125pt]{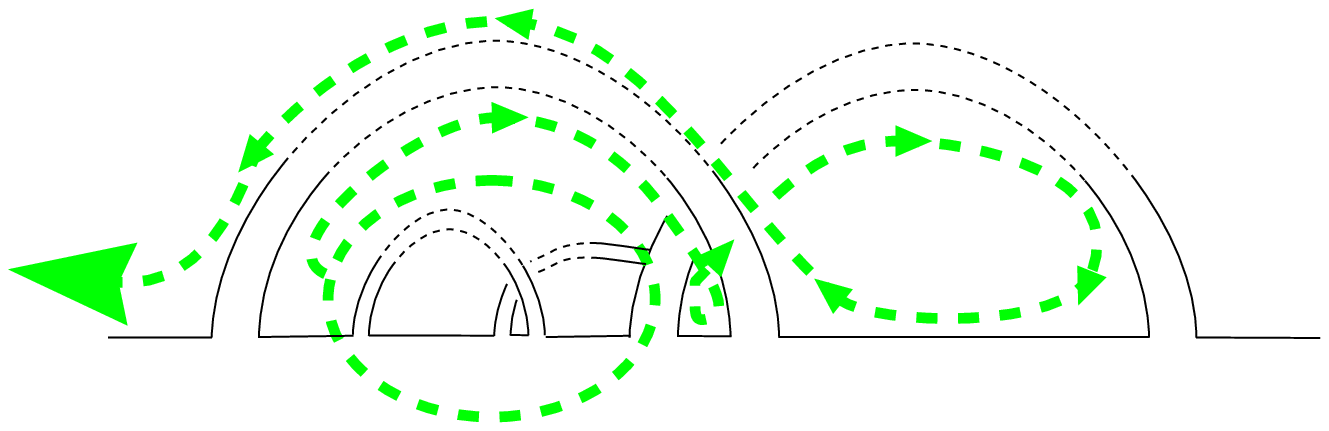}
\end{minipage}\\
\begin{minipage}{23pt}
\includegraphics[width=23pt]{fluffyarrow}
\end{minipage}\qquad
\begin{minipage}{200pt}
\includegraphics[width=200pt]{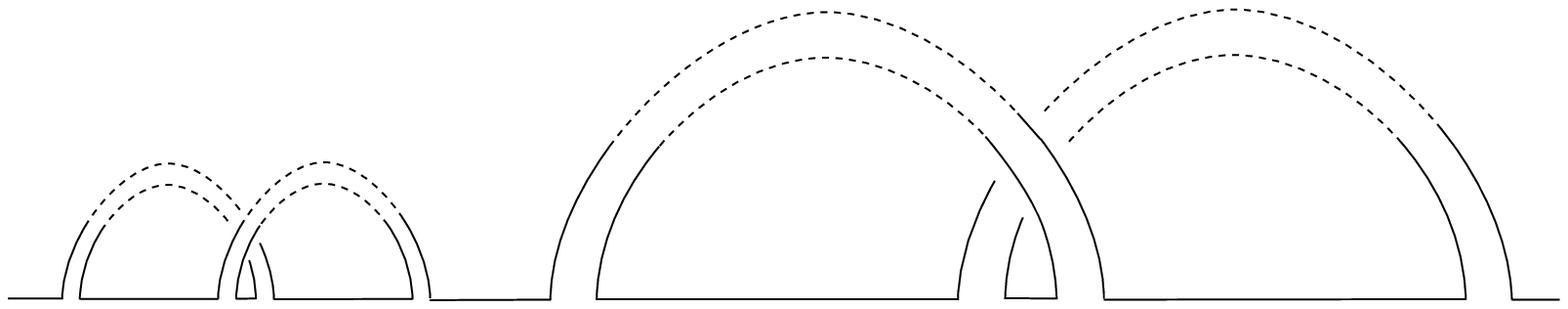}
\end{minipage}
\end{multline*}
\noindent and repeat the sequence of slides which we used to kill
$v_3$ in order to kill $v_5$.\par

Repeat all steps above to kill $v_{2i+1}$ for all
$i=1,2,\ldots,g-1$, and the colouring vector becomes
$\vec{v}=\left(s,1,\ldots,1\right)^T$ as required.

\subsubsection{Step 3: Surgery to reduce genus}

Now that we have a band projection with respect to which
$\vec{v}=\left(s,1,\ldots,1\right)^T$, we can reduce genus by
surgery. Assume that $g>1$. By surgery we trivialize bands $B_i$ for
$i>2$, starting from the right.\par

If $B_i$ links with $B_{2g}$ for some $i<2g$, we may isotopy $B_i$
to make sure it passes first over and then under $B_{2g}$ with
respect to the orientation of $x_{2g}$:

$$
\psfrag{1}[r]{$B_{2g}$}\psfrag{2}{$B_i$}
\begin{minipage}{100pt}
\includegraphics[width=100pt]{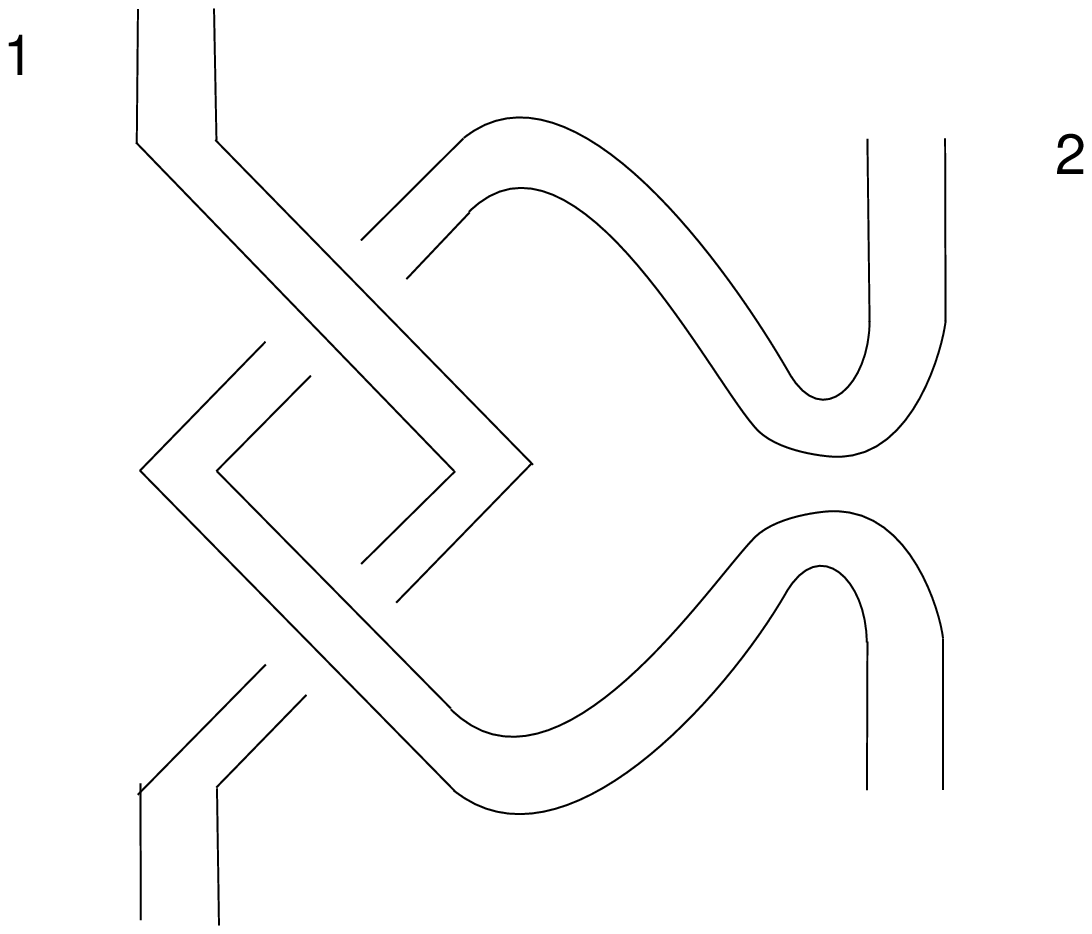}
\end{minipage}\qquad\Longleftrightarrow\qquad
\begin{minipage}{100pt}
\includegraphics[width=100pt]{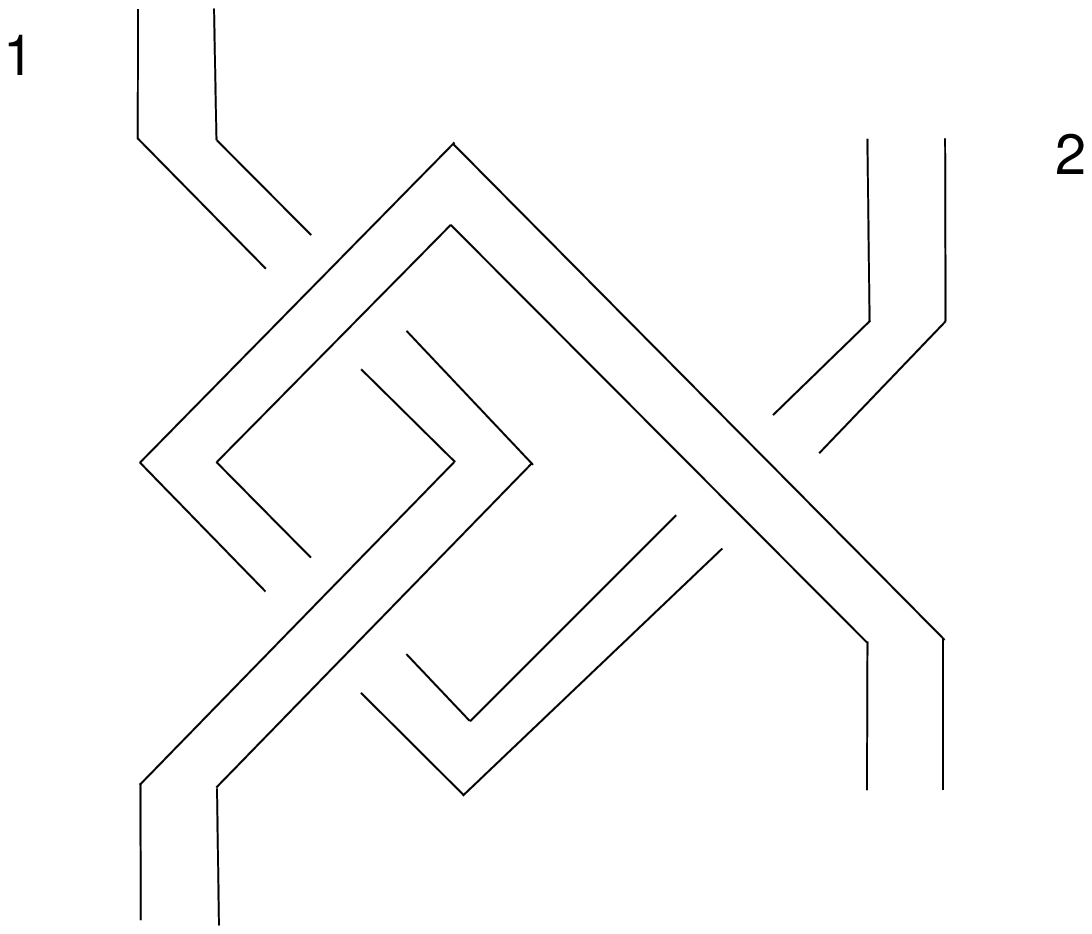}
\end{minipage}
$$\\


Denote by $\bar{C}$ and by $\bar{D}$ collective linkage of other
bands with $B_{2g}$ and $B_{2g-1}$ correspondingly. Consider Figure
\ref{F:genred1}.

\begin{figure}[h]
\psfrag{1}{\Small$\delta_{4g}$}\psfrag{2}{\Small$\delta_{4g-1}$}\psfrag{3}{\Small$\delta_{4g-2}$}\psfrag{4}{\Small$\delta_{4g-3}$}\psfrag{5}{\Small$\delta_{4g-4}$}
\psfrag{c}{$\bar{C}$}\psfrag{d}{$\bar{D}$}
\psfrag{C}{\fs$C_1$}\psfrag{D}[c]{\fs$C_2$}
\includegraphics[width=200pt]{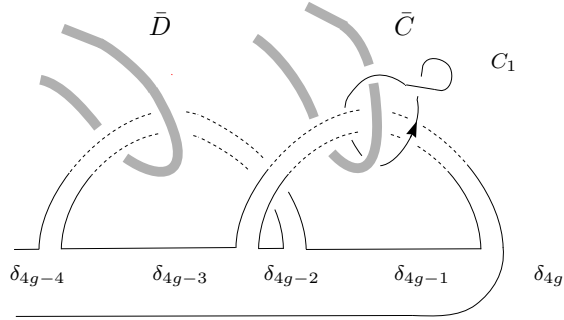}
\caption{\label{F:genred1}Unlinking $\bar{C}$ from $B_{2g}$.}
\end{figure}

Denoting by $\rho(\delta_i)$ the $\rho$--image of the Wirtinger
generator corresponding to $\delta_i$, we have
\begin{align*}
v_1&=\rho(\delta_{4g}\cdot\delta_{4g-1})=\rho(\delta_{4g-3}\cdot\delta_{4g-2})=1\\
v_2&=\rho(\delta_{4g-1}\cdot\delta_{4g-2})=\rho(\delta_{4g-4}\cdot\delta_{4g-3})=1.
\end{align*}

\noindent Therefore $
\rho(\delta_{4g})=\rho(\delta_{4g-1})=\rho(\delta_{4g-2})=
\rho(\delta_{4g-3})=\rho(\delta_{4g-4})$. We also know that
conjugation of $\rho(\delta_{1})$ by the $\rho$--image of all the
arcs in $\bar{C}$ which cross over $B_{2g}$ gives
$\rho(\delta_{4g-3})$, which is equal to $\rho(\delta_{4g})$. In
other words, conjugation of $\rho(\delta_{4g})$, which equals $ts^a$
for some $a\in\pZ$, by the $\rho$--image of a $+1$--framed component
$C_1^\prime$ which loops once around $\bar{C}$ equals
$\rho(\delta_{4g})$. The component $C_1^\prime$ is in $\pi_1(S^3-F)$
(let's allow ourselves to confuse curves with the homotopy classes
which they represent) so $\rho(C_1^\prime)=s^b$ for some $b\in\pZ$
and therefore
$$
ts^a= s^{-b}\cdot ts^a\cdot s^b= ts^{a+2b}
$$
\noindent which is possible only if $b=0$. Thus $C_1^\prime$ is in
$\ker\rho$, and a $+1$--framed component $C_1$ which loops once
around $\bar{C}$ and once around $B_{2g}$ is also in $\ker\rho$ (see
Figure \ref{F:genred1}). By performing surgery along $C_1$ we may
unlink $\bar{C}$ and $B_{2g}$. Thus after this step $B_{2g}$ won't
be linked with any other bands.\par

Next, using the fact that $v_{2g}=1$, untwist $B_{2g}$ by surgery on
$\pm 1$--framed components which ring $B_{2g}$:

\begin{equation}\label{E:RR}
\psfrag{a}{\ \Small$ts^a$}\psfrag{b}[lB]{\Small$ts^a$}
\begin{minipage}{60pt}
\includegraphics[height=60pt]{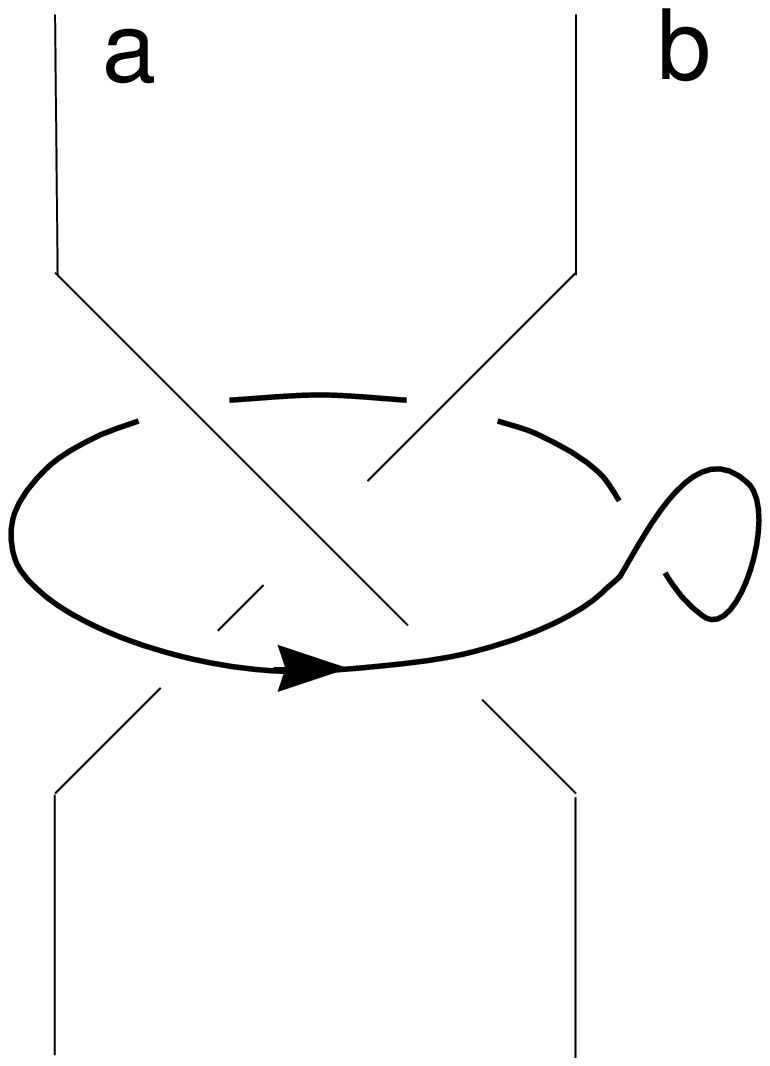}
\end{minipage}\quad \Longleftrightarrow \qquad
\begin{minipage}{60pt}
\includegraphics[height=60pt]{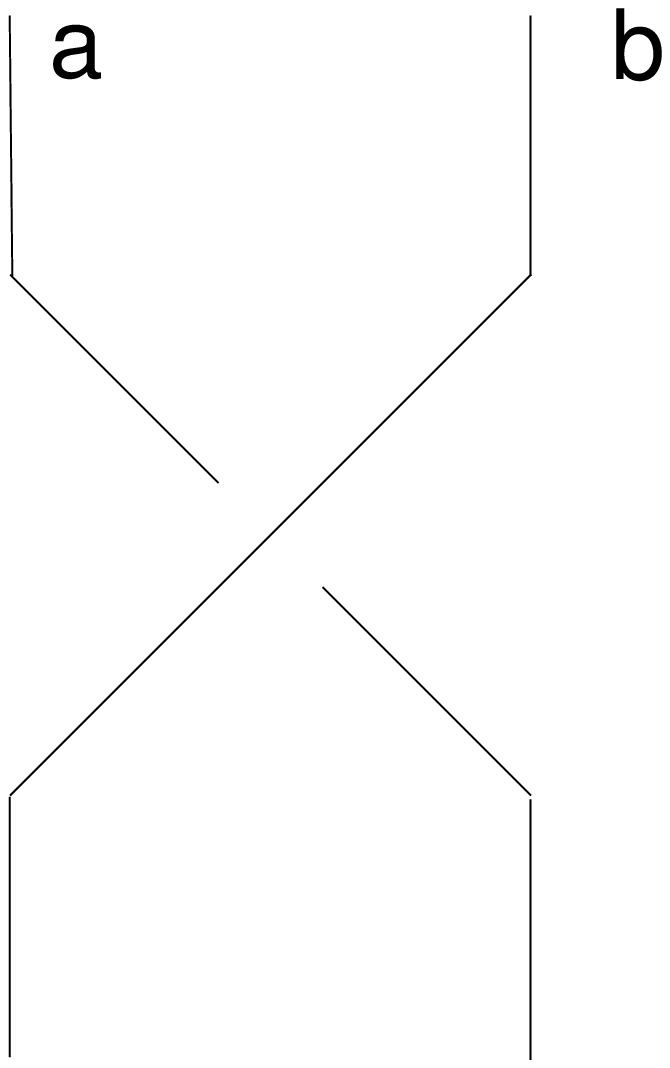}
\end{minipage}
\end{equation}

Finally, if $B_{2g}$ is knotted then it can be untied by surgery in
$\ker\rho$.

\begin{equation}\label{E:untiebands}
\psfrag{C}[lc]{$+1$}
\begin{minipage}{80pt}
\includegraphics[height=80pt]{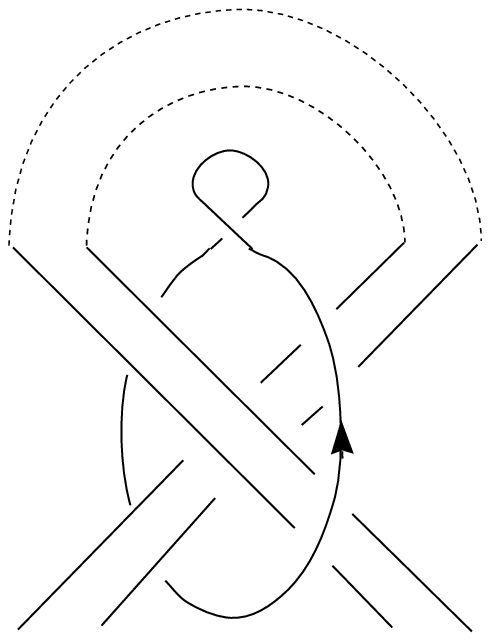}
\end{minipage}\qquad\Longleftrightarrow \qquad
\begin{minipage}{80pt}
\includegraphics[height=80pt]{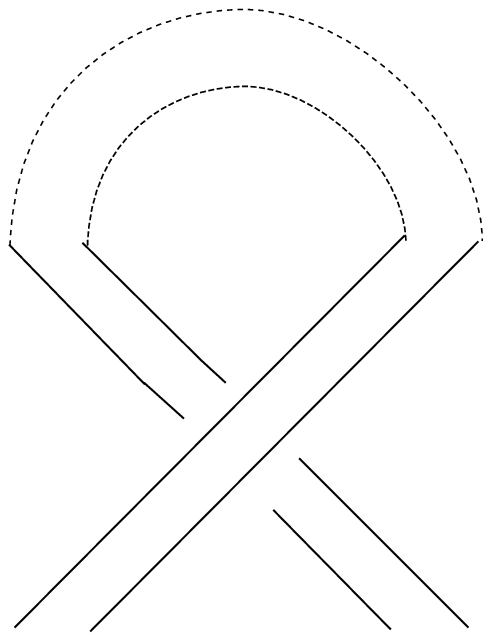}
\end{minipage}
\end{equation}

The resulting diagram represents a $\Dn$-coloured knot of lower
genus that the one we began with, as $B_{2g}$ and $B_{2g-1}$ unravel
and vanish (see Figure \ref{F:genred2}).

\begin{figure}[h]
\psfrag{1}{\Small$\delta_{4g}$}\psfrag{2}{\Small$\delta_{4g-1}$}\psfrag{3}{\Small$\delta_{4g-2}$}\psfrag{4}{\Small$\delta_{4g-3}$}
\psfrag{5}{\Small$\delta_{4g-4}$}
\psfrag{c}{$\bar{C}$}\psfrag{d}{$\bar{D}$}
\psfrag{C}{\fs$C_1$}\psfrag{D}[c]{\fs$C_2$}
\includegraphics[width=180pt]{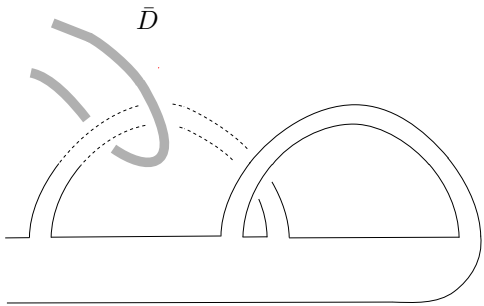}
\caption{\label{F:genred2}After untwisting $B_{2g}$, untying it, and
unlinking everything from it.}
\end{figure}

\subsection{Genus one knots}

In the last section we saw that any $\Dn$-coloured knot $(K,\rho)$
is $\rho$--equivalent to a $\Dn$-coloured genus one knot
$(K^\prime,\rho^\prime)$. Let $(K^\prime,\rho^\prime)$ be given in
band projection, with respect to which it has surface data:

\begin{equation} (\M,\vec{v})\ass\
\genusoneknot{a_{11}}{a_{12}}{a_{12}+1}{a_{22}}{v_1}{v_2}
\end{equation}

We may change any crossing between a band and itself by
$\rho$--equivalence (Equation \ref{E:untiebands}). We may thus take
$B_1$ and $B_2$ to be unknotted. Since two-component homotopy links
are uniquely characterized by their linking number (\textit{i.e.}
$\pi_1(N(x_1))\simeq\mathds{Z}$), we may view $(\M,\vec{v})$ as
giving rise to a unique $\Dn$-coloured knot. From now on we shall do
so, and $B_1$ and $B_2$ will always be assumed to be unknotted.\par

By Lemma \ref{L:slidingtwins} we may take $(v_1,v_2)=(s,1)$, fixing
the colouring vector (surjectivity of $\rho$ implies non-vanishing
of the colouring vector). Now that $v_2=1$ we may add or subtract
twists from $B_2$ by $\rho$--equivalence as in \ref{E:RR}. Therefore
we may set $a_{22}$ to any integer we please. Let's set it to
$\frac{1-n}{2}$.\par

Since $\vec{w}^{\ T} \cdot \M\cdot \vec{w}= 0\bmod n$ (Lemma
\ref{L:VMV}, recalling that $\vec{w}=(w_1,\ldots,w_{2g})$ with
$v_i=s^{w_i}$), we also know that $a_{11}=0\bmod n$. We may add or
subtract $n^2$ full twists in $B_1$ by surgery on a unit--framed
component which rings $n$ times around $B_1$, thus setting
$a_{11}=kn$ for some $0\leq k\leq n-1$. This is illustrated below in
the case $n=3$ (with the number of full twists indicated near the
bands), where the second surgery is there to keep $B_1$ unknotted.

\begin{equation}\label{E:add9twist}
\begin{minipage}{100pt}
\psfrag{a}{$-3$}\psfrag{b}{$-1$}
\includegraphics[width=100pt]{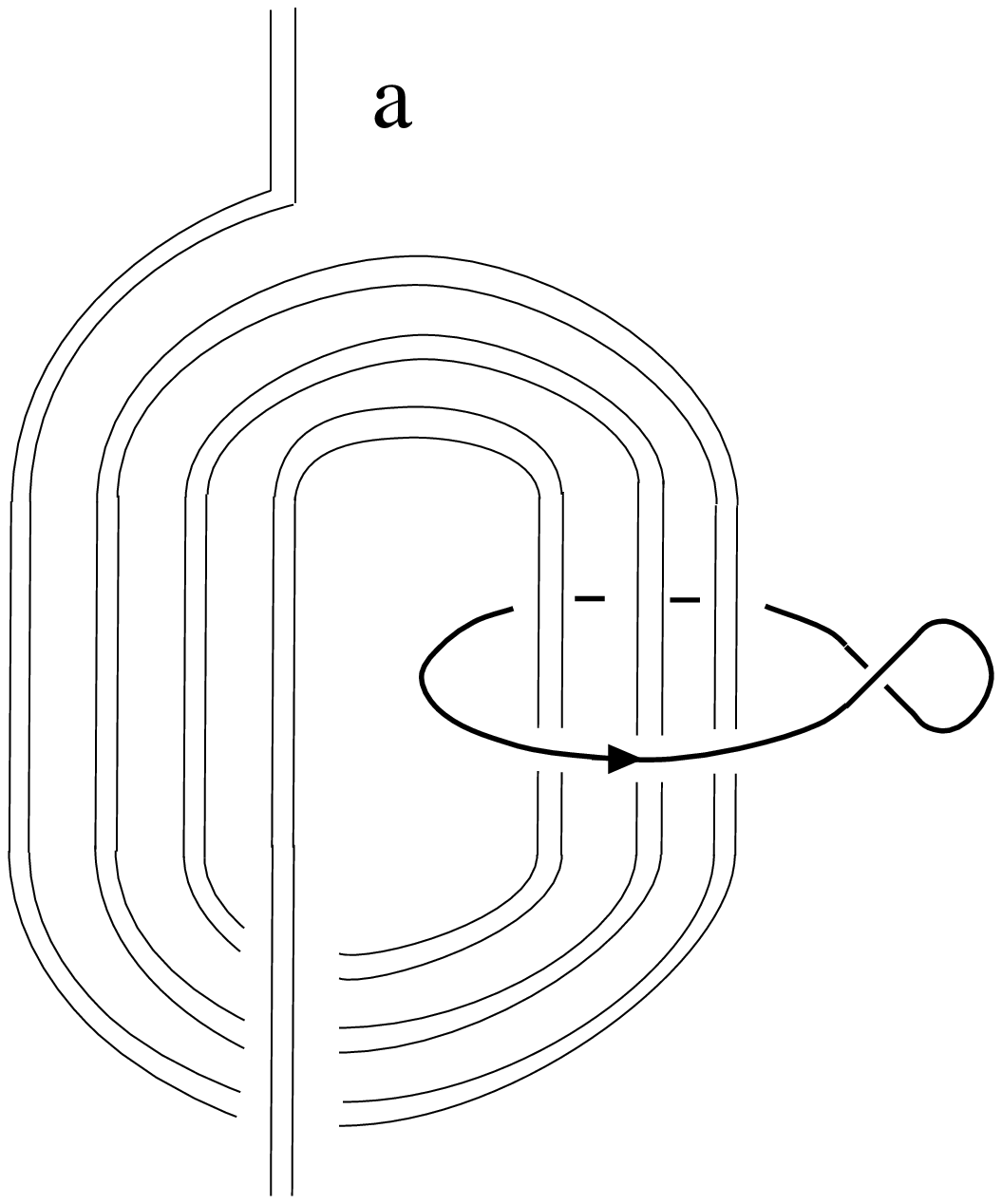}
\end{minipage}\quad\ \seq\hspace{15pt}
\begin{minipage}{100pt}
\psfrag{a}{$0$}\psfrag{b}{$-1$}
\includegraphics[width=85pt]{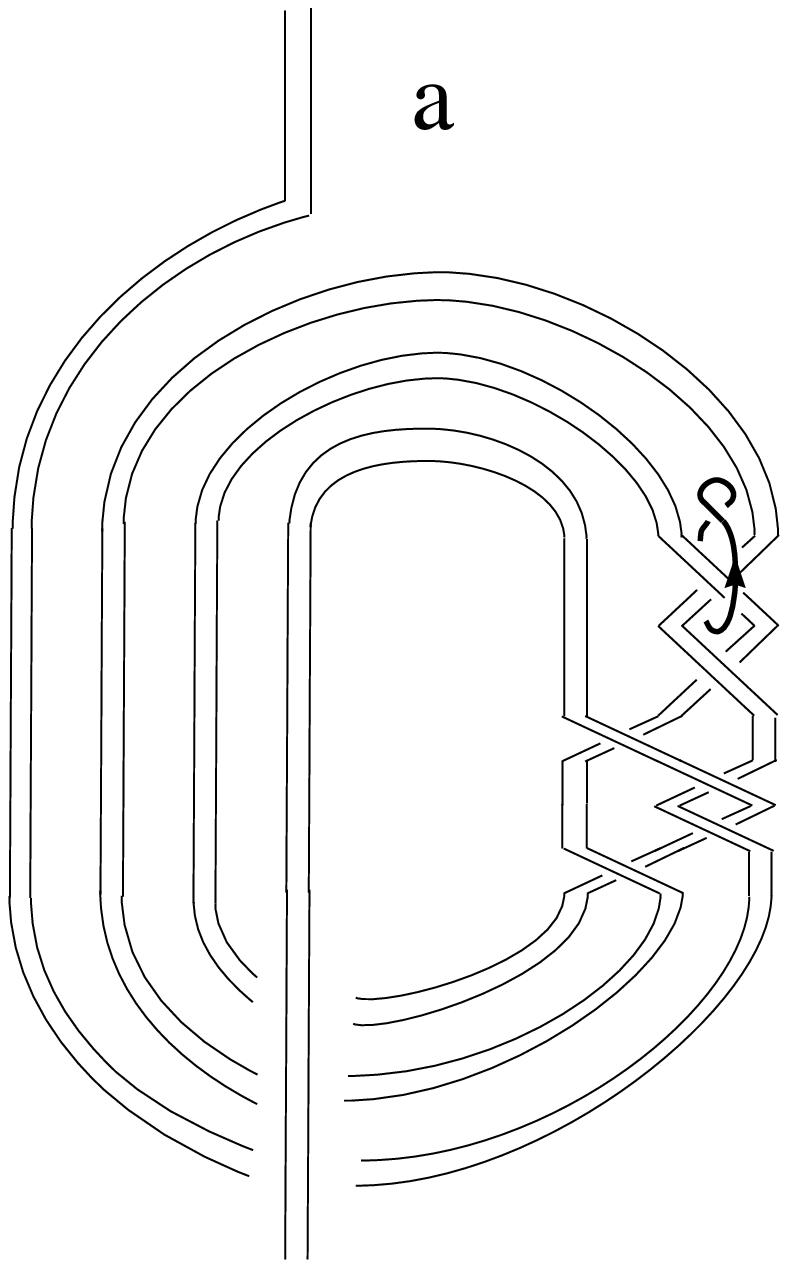}
\end{minipage}\quad\seq\hspace{15pt}\quad
\begin{minipage}{100pt}
\psfrag{a}{$9$}
\includegraphics[width=18.5pt]{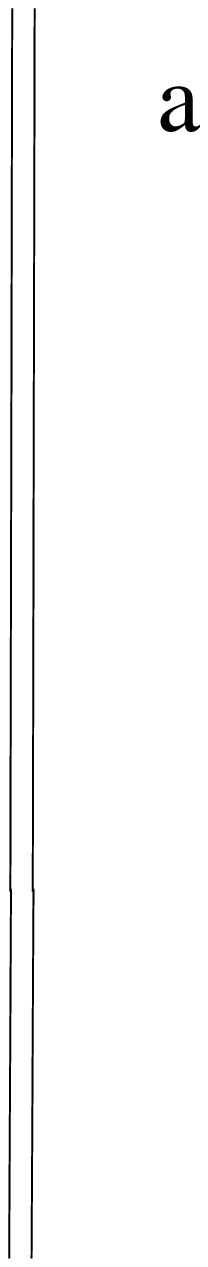}
\end{minipage}
\end{equation}

Next, we may add or subtract $n$ from $a_{12}$ by $\pm1$--framed
surgery on a component $C$ which has no self-intersections in the
projection to the plane which gives the band projection of
$K^\prime$, and for which $\mathrm{Link}(C,B_1)=n$ and
$\mathrm{Link}(C,B_2)=1$ ($C$ is in $\ker\rho$). The surgery is
illustrated below in the case $n=3$:

\begin{equation}
\begin{minipage}{100pt}
\psfrag{a}{$-3$}\psfrag{c}{$0$} \psfrag{b}{$-1$}
\includegraphics[width=100pt]{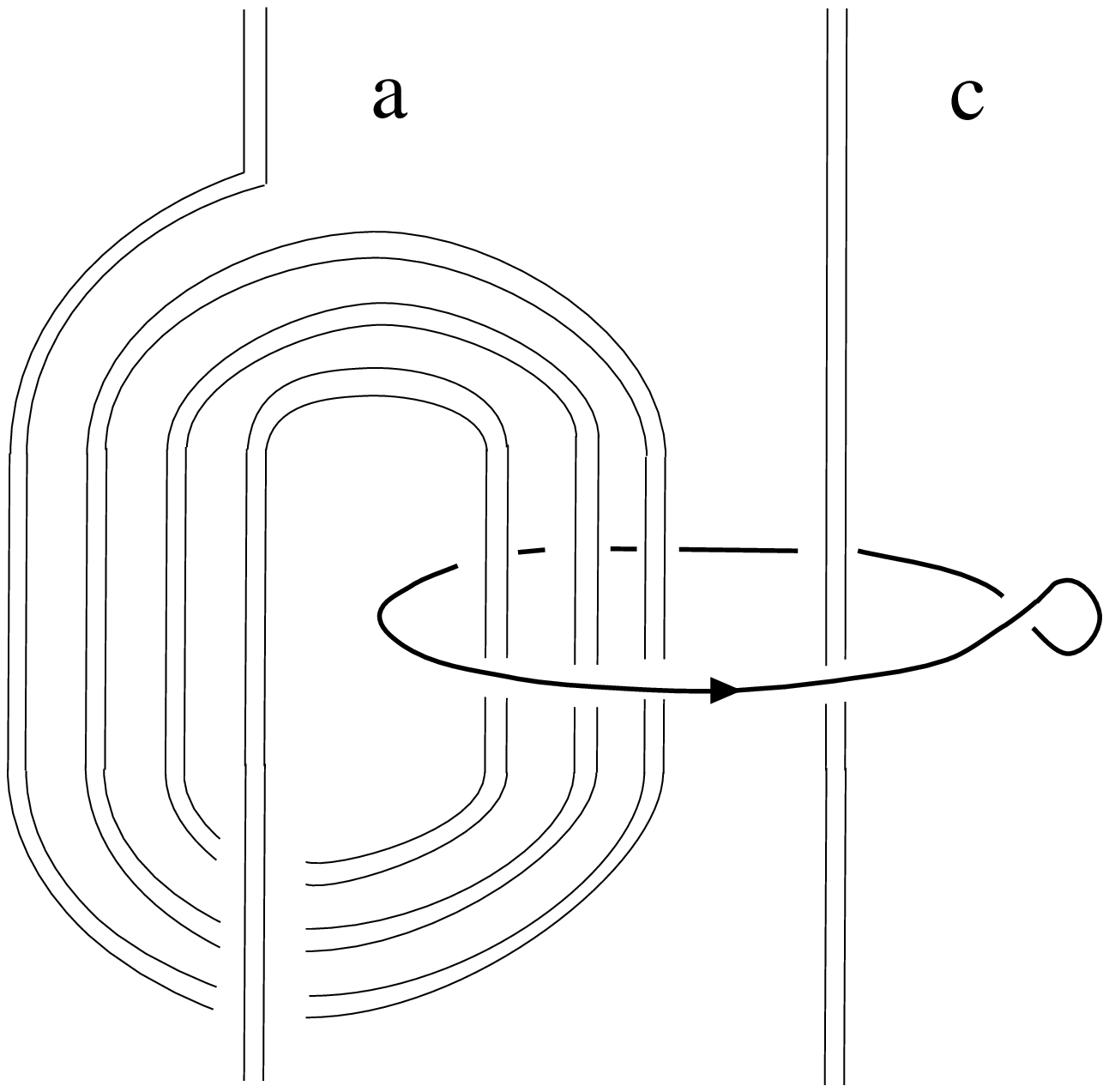}
\end{minipage}\qquad\seq\hspace{12pt}
\begin{minipage}{100pt}
\psfrag{a}{$0$}\psfrag{b}{$1$}\psfrag{c}{$-1$}
\includegraphics[width=90pt]{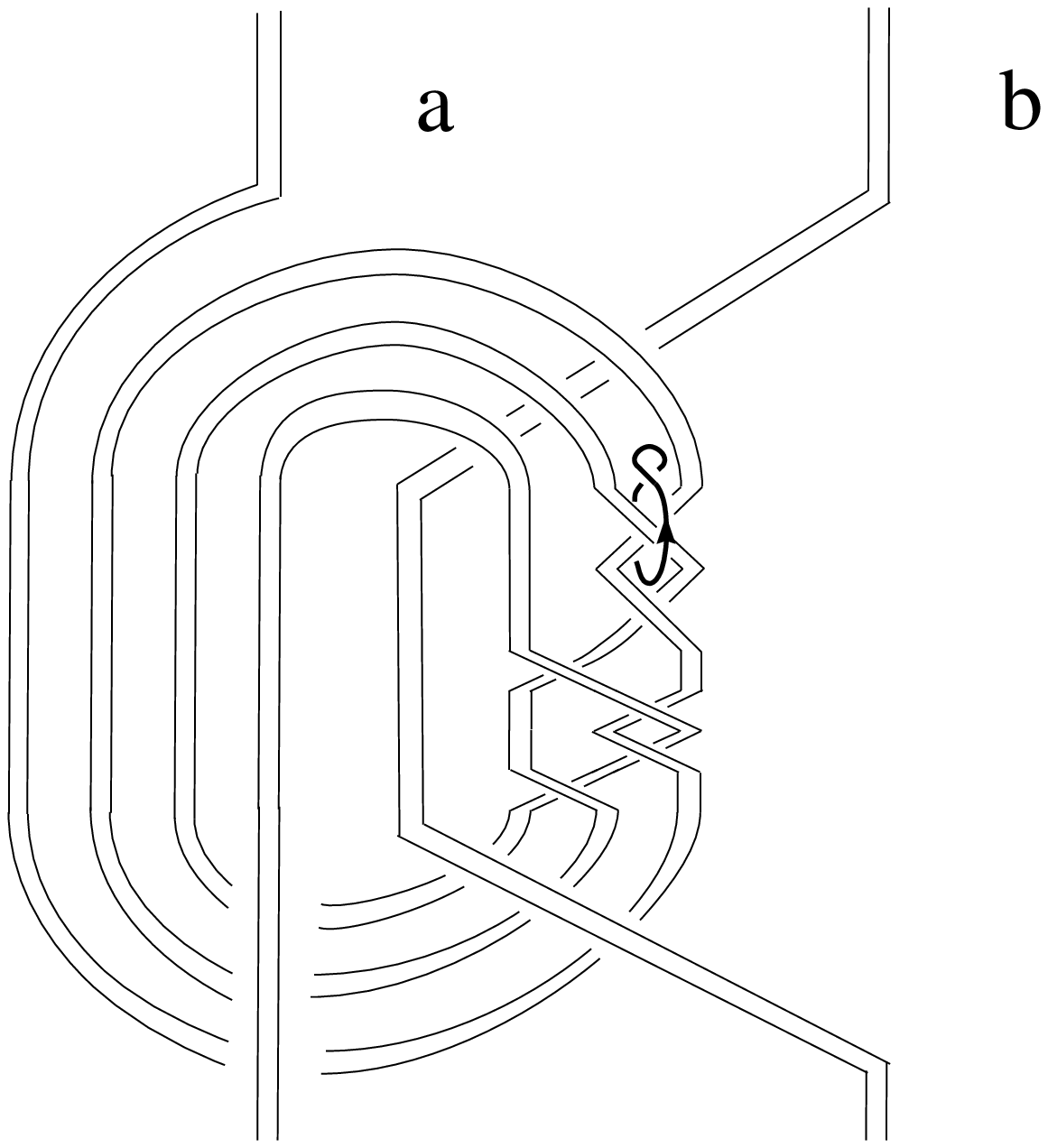}
\end{minipage}\seq\hspace{13pt}
\begin{minipage}{100pt}
\psfrag{a}{$9$}\psfrag{b}{$1$}
\includegraphics[width=68pt]{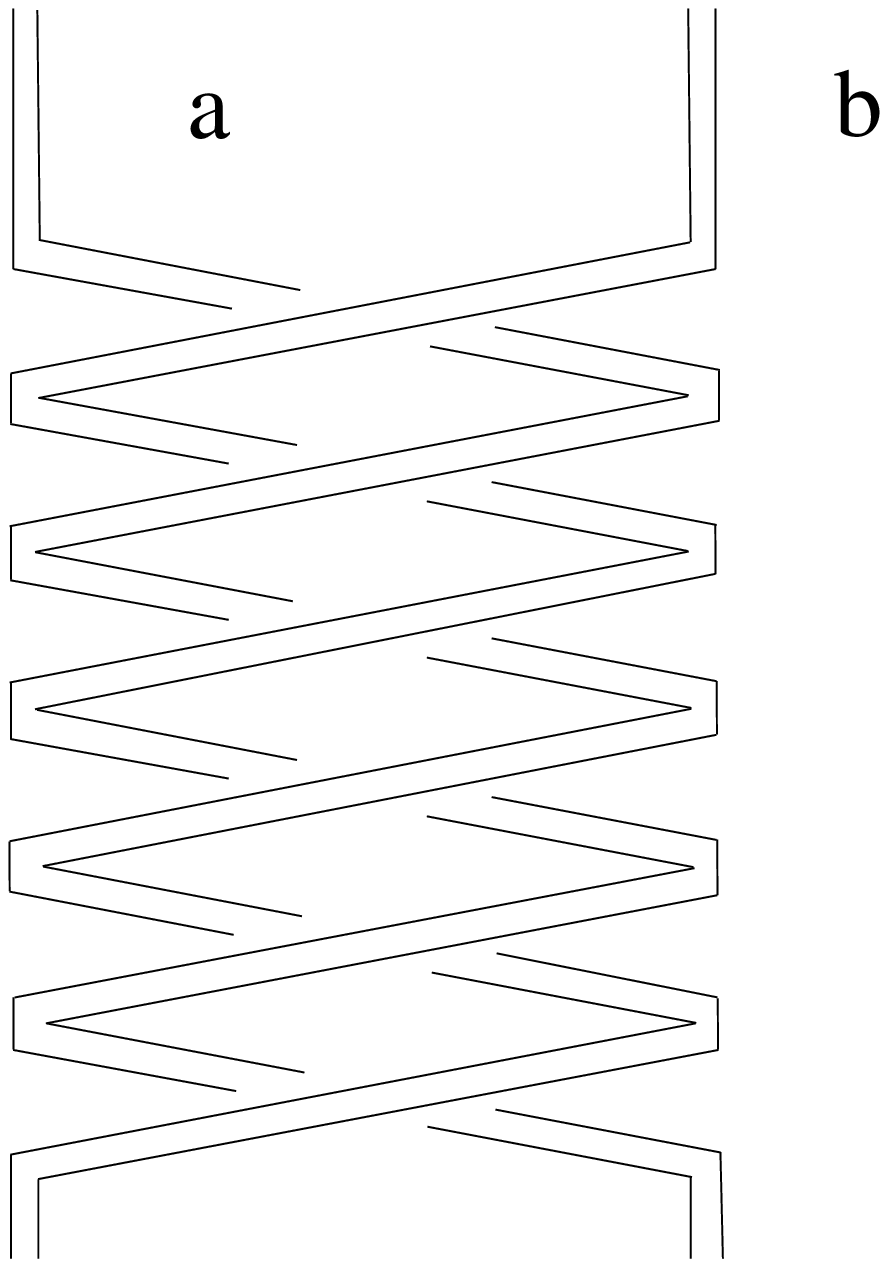}
\end{minipage}
\end{equation}

Cancel the twists added in $B_1$ and $B_2$ by surgery on a
unit--framed component which rings $n$ times around $B_1$ as in
\ref{E:add9twist}, and surgery on a unit--framed component which
rings once times around $B_2$ as in \ref{E:RR}.\par

Thus we may choose $0\leq a_{12}<n$. Because for $\vec{x}=(0,1)^T$
we have $\vec{x}^{\ T} \cdot \M\cdot \vec{w}= 0\bmod n$ (Lemma
\ref{L:VMV}) and because also $a_{12}+1=a_{21}$, this implies that
we may choose $a_{12}=\frac{n-1}{2}$.\par

We are almost there--- we have shown that any $\Dn$-coloured knot is
equivalent to a $\Dn$-coloured knot for which there exists a
coloured Seifert matrix of the form:
\begin{equation}\label{E:finalform}
(\M,\vec{v})=
\genusoneknot{kn}{\frac{n-1}{2}}{\frac{n+1}{2}}{\frac{1-n}{2}}{s}{1}
\end{equation}

\noindent with $k=0,\ldots,n-1$.\par

To simplify one step further, to get the easiest knot to lift that
we can, perform one additional band slide:

$$
(\M,\vec{v})=
\genusoneknot{kn}{\frac{n-1}{2}}{\frac{n+1}{2}}{\frac{1-n}{2}}{s}{1}\mapsto
\genusoneknot{(kn+\frac{n+1}{2})}{0}{1}{\frac{1-n}{2}}{s}{s^{-1}}
$$

The surface data $(\M,\vec{v})$ uniquely determines a genus one
$\Dn$-coloured knot if we assume unknotted bands (as we do). We
denote this knot $(\K_k,\rho_k)$ (see Figure \ref{F:baseknot}, where
the thick line
``\raisebox{1.5pt}{\includegraphics[width=30pt]{crayonline}}''
denotes an arbitrary number of strands). It is the pretzel knot
$p\left(\rule{0pt}{9pt}2kn+1,-1,-n\right)$ with a certain
$\Dn$-colouring. In this section we have shown that any knot is
$\rho$--equivalent to such a knot for some $k=0,\ldots,n-1$, proving
Theorem \ref{T:PretzBaseKnot}.

\begin{figure}
\begin{minipage}{2.4in}
\psfrag{l}[c]{\Huge$\vdots$}\psfrag{r}[c]{\Huge$\vdots$}\psfrag{k}[r]{\parbox{0.7in}{$2kn+1$\\[0.1cm]
half--twists}\ \ $\left\{\rule{0pt}{0.7in}\right.$}\psfrag{m}[l]{$\left.\rule{0pt}{0.7in}\right\}$\ \ \parbox{0.7in}{\ \quad$-n$\\[0.1cm]
half--twists}}
\psfrag{1}[c]{$ts$}\psfrag{2}[c]{$t$}\psfrag{3}[c]{$t$}\psfrag{4}[c]{$ts$}\psfrag{L}[c]{\Large$T$}
\includegraphics[width=2.4in]{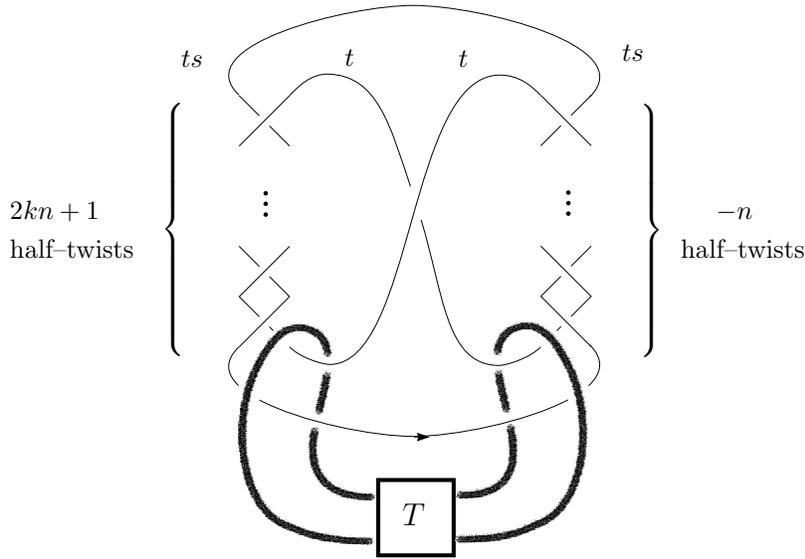}
\end{minipage}\rule{0pt}{1.5in}
\caption{\label{F:baseknot}The $\Dn$-coloured knot $(\K_k,\rho_k)$
with the surgery link in its complement.}
\end{figure}

\subsection{Constructing the cover.}\label{SS:SeifertLift}

In the previous section we proved that for any $D_{2n}$-coloured
knot $(K,\rho)$ with coloured untying invariant $0\leq k<n$ there
exists a $\pm1$--framed link $L$ in the complement of the
$\Dn$-coloured pretzel knot $(\K_k,\rho_k)$ of Figure
\ref{F:baseknot}, whose components are unknotted and in $\ker\rho$,
and surgery along which recovers $(K,\rho)$. The goal of this
section is to construct the branched dihedral covering space and
covering link corresponding to this data, and to lift the surgery
information to this cover.\par


\subsubsection{Language and Notation}
Coordinates in $\mathds{R}^3\subset S^3$ will be employed to
explicitly describe configurations of objects in $3$--space. Denote
by $\Sigma \subset \mathds{R}^2$ the surface arising from a
sufficiently large disc when small discs centred at the points
$(-2,0)$, $(-1,0)$, $(1,0)$ and $(2,0)$ are removed. The surgery
link $L$ will lie inside $\Sigma\times [0,1]\subset \mathds{R}^3$.
We will think of $L$ as being the closure of a $\pm1$
framed tangle $T$ 
such that diagram of $L$ arising from the projection onto
$\Sigma\times \{0\}$ is as pictured in Figure \ref{F:FIGA}.\par

\begin{figure}[h]
\psfrag{T}[c]{\Large$T$}\psfrag{1}[c]{$x_1$}\psfrag{2}[c]{$x_2$}
\includegraphics[width=300pt]{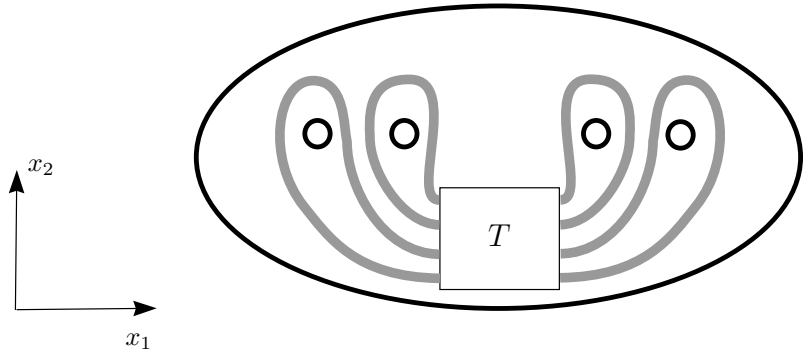}
\caption{\label{F:FIGA}A diagram of the surgery link
$L\subset\Sigma\times[0,1]$.}
\end{figure}

The knot $\K_k$ over which we'll be taking a branched dihedral cover
can be assumed to live in
$\mathds{R}^3-\left(\Sigma\times[0,1]\right)$, as pictured in Figure
\ref{F:FIGB} (using the convention that the coordinate $x_2$
increases \textit{into} the page).\par

\begin{figure}[h]
\psfrag{S}[c]{\small$\Sigma\!\times\![0,1]$}\psfrag{l}[c]{$\cdots$}\psfrag{r}[c]{$\cdots$}\psfrag{k}[c]{$2kn+1$
half-twists}\psfrag{m}[c]{$-n$ half-twists}
\psfrag{1}[c]{$x_1$}\psfrag{3}[c]{$x_3$}\psfrag{a}[c]{\rotatebox{270}{$\left\{\rule{0pt}{50pt}\right.$}}\psfrag{b}[c]{\rotatebox{270}{$\left.\rule{0pt}{50pt}\right\}$}}
\includegraphics[width=360pt]{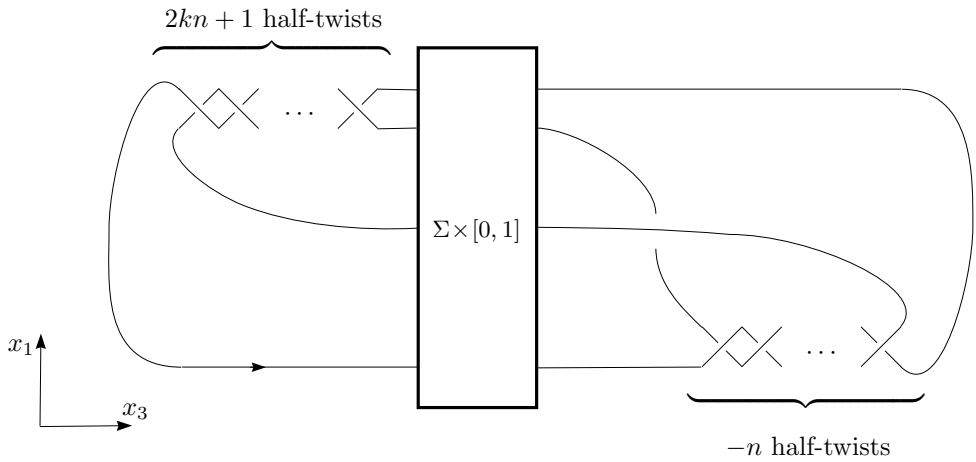}
\caption{\label{F:FIGB}The cylinder $\Sigma\times[0,1]$ sitting
inside the knot complement $\mathds{R}^3-\K_k$.}
\end{figure}

The $\Dn$-colouring $\rho$ induces a representation from
$\pi_1\left(\Sigma\times[0,1]\right)$ into $\Dn$, which we shall
also call $\rho$ by abuse of notation. To describe this
representation, choose a base point for $\Sigma\times[0,1]$ lying on
the surface $\Sigma\times\{0\}$, and specify the images of
generators for  with respect to this basepoint as shown in Figure
\ref{F:FIGC}. The two `outer' generators map to $ts$, while the two
`inner' generators map to $t$.\par

\begin{figure}[h]
\psfrag{a}[c]{$ts$}\psfrag{b}[c]{$t$}\psfrag{c}[c]{$t$}\psfrag{d}[c]{$ts$}
\includegraphics[width=200pt]{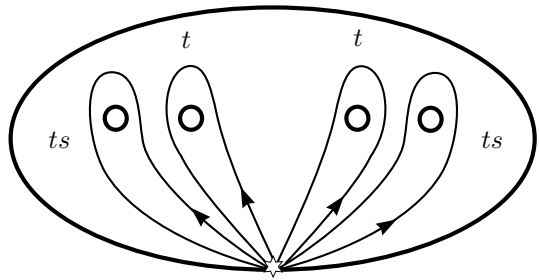}
\caption{\label{F:FIGC}Generators for
$\pi_1\left(\Sigma\times[0,1]\right)$ and their images in $\Dn$.}
\end{figure}

\subsubsection{Constructing $\covcyl$}

We now have language and notation which is sufficiently explicit to
describe the construction of $\covcyl$, the dihedral cover of
$\Sigma\times [0,1]$ with respect to $\rho$, and its embedding in
the branched dihedral covering over $(\K_k,\rho_k)$, which is $S^3$
because $\K_k$ is a $2$--bridge knot (see \textit{e.g.}
\cite{Bir76}).\par


We build $\covcyl$ embedded in $\mathbb{R}^2\times
[0,1]\subset\mathbb{R}^3$ by slotting together copies of
$\Sigma\times[0,1]$ cut open along planes, as shown in Figure
\ref{F:FIGD}. These are our ``lego blocks'', which we can bend,
stretch, and shrink. To construct $\covcyl$ (embedded in
$\mathds{R}^3$) we take $n$ blocks denoted $X_1,\ldots,X_n$ (copies
of the cut-open $\Sigma\times[0,1]$ of Figure \ref{F:FIGD}) and slot
them together in the usual way: always matching an $A$ to an $A'$,
and so on, and using the representation to decide which copy of the
$3$--cell one passes to when crossing a cut.\par

\begin{figure}[h]
\centering \psfrag{T}[c]{$T$}
\begin{minipage}[c]{0.49\textwidth}
\centering
\includegraphics[width=140pt]{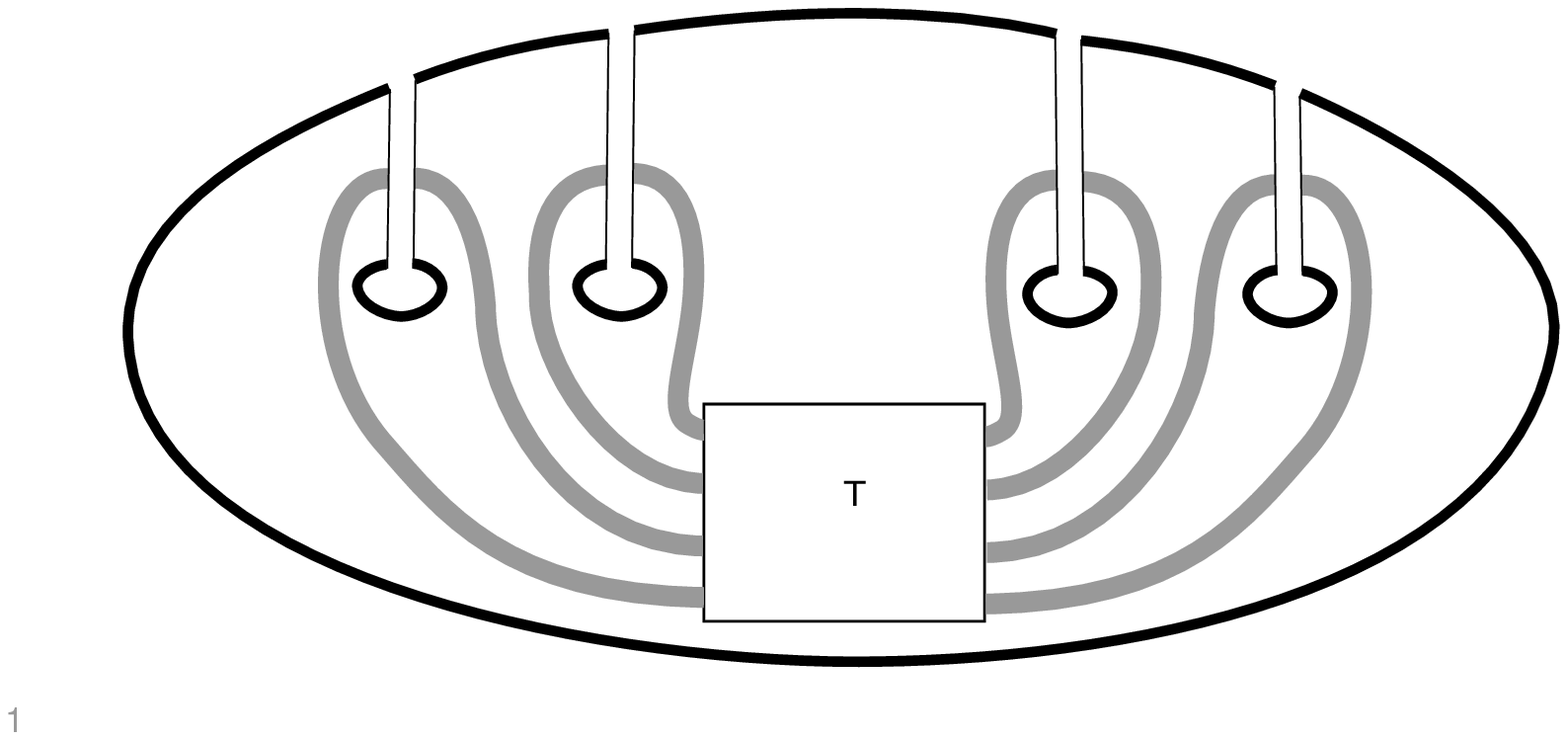}
\end{minipage}
\begin{minipage}[c]{0.49\textwidth}
\centering%
\psfrag{1}[c]{\Tiny$C\phantom{^\prime}$}
\psfrag{2}[c]{\Tiny$D\phantom{^\prime}$}
\psfrag{3}[c]{\Tiny$D^\prime$} \psfrag{4}[c]{\Tiny$C^\prime$}
\psfrag{5}[c]{\Tiny$E\phantom{^\prime}$}
\psfrag{6}[c]{\Tiny$F\phantom{^\prime}$}
\psfrag{7}[c]{\Tiny$F^\prime$} \psfrag{8}[c]{\Tiny$E^\prime$}
\psfrag{9}[c]{\Tiny$A^\prime$}\psfrag{0}[c]{\Tiny$B^\prime$}\psfrag{q}[c]{\Tiny$B\phantom{^\prime}$}
\psfrag{w}[c]{\Tiny$A\phantom{^\prime}$}\psfrag{e}[c]{\Tiny$G^\prime$}\psfrag{r}[c]{\Tiny$H^\prime$}
\psfrag{t}[c]{\Tiny$H\phantom{^\prime}$}\psfrag{y}[c]{\Tiny$G\phantom{^\prime}$}
\includegraphics[width=155pt]{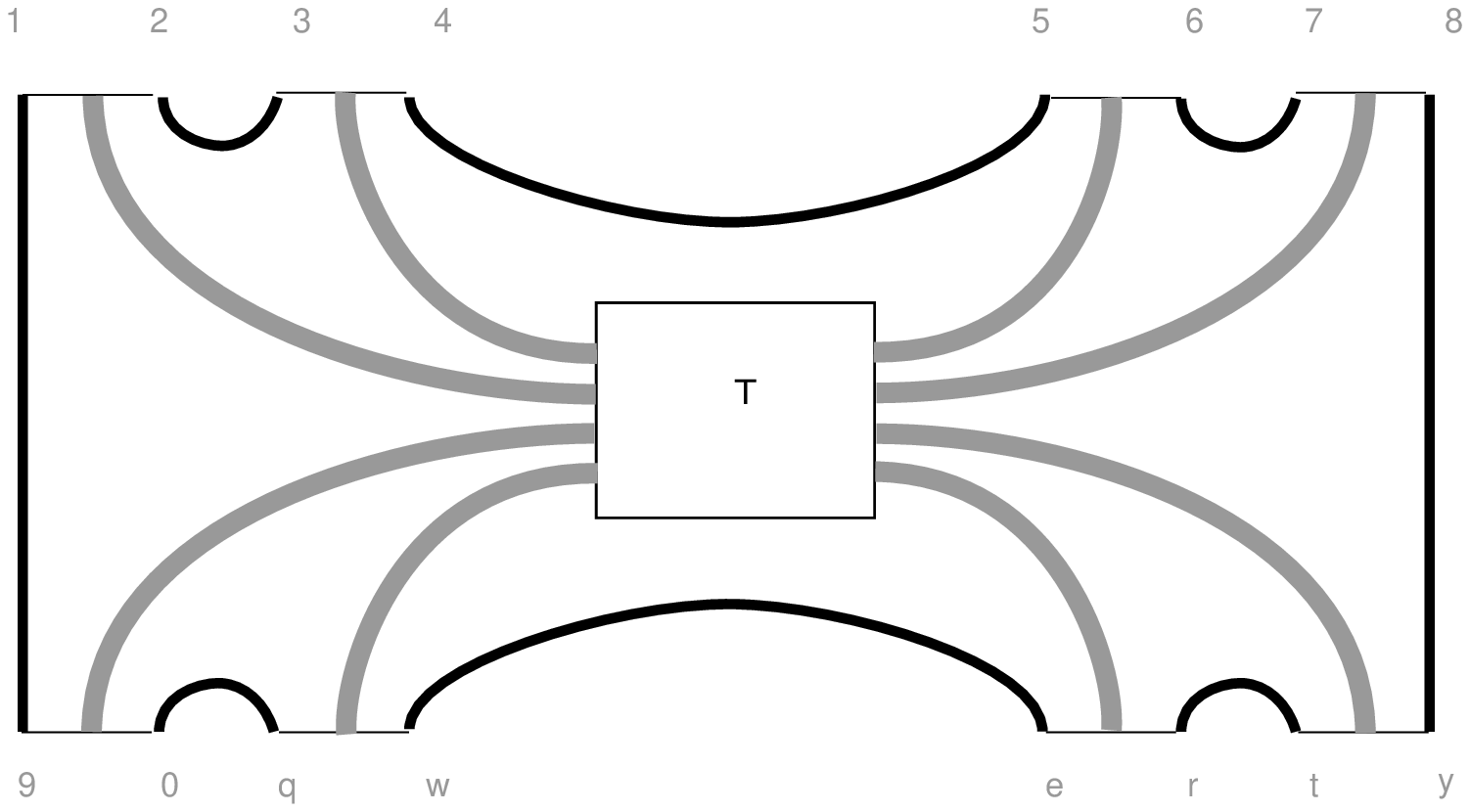}
\end{minipage}
\caption{\label{F:FIGD}A block $X_i$ obtained by cutting open
$\Sigma\times[0,1]$ along planes, and $X_i$ after being `opened out'
by isomorphism.}
\end{figure}

To see that combinatorially the blocks end up lined up in a line,
consider the graph with $n$ vertices labeled $1,\ldots,n$ and an arc
connecting vertices $i$ and $j$ if and only if $X_i$ and $X_j$ are
incident in $\covcyl$, \textit{i.e.} if and only if we slot $X_i$
and $X_j$ together, which is if and only if $t(i)=j$ or $ts(i)=j$
where $\Dn$ is acting on $1,\ldots,n$ by symmetries of the regular
$n$-gon (remember that when crossing a cut labeled $t$ we are going
to be crossing from $X_i$ to $X_{t(i)}$, and similarly for $ts$).
When $n=7$ the graph is given in Figure \ref{F:FIGH}. Because
$t(1)=1$ and $ts(\frac{n+3}{2})=\frac{n+3}{2}$, the graph will
consist of two loops and a path from $1$ to $\frac{n+3}{2}$.\par

\begin{figure}[h]
\centering
\psfrag{1}[c]{$1$}\psfrag{2}[c]{$2$}\psfrag{3}[c]{$3$}\psfrag{4}[c]{$4$}\psfrag{5}[c]{$5$}\psfrag{6}[c]{$6$}\psfrag{7}[c]{$7$}
\psfrag{a}{}\psfrag{b}{}\psfrag{c}{}\psfrag{d}{}\psfrag{e}{}\psfrag{f}{}\psfrag{g}{}\psfrag{h}{}
\psfrag{t}[c]{$t$}\psfrag{s}[l]{$ts$}
\begin{minipage}[c]{0.4\textwidth}
\centering
\includegraphics[width=140pt]{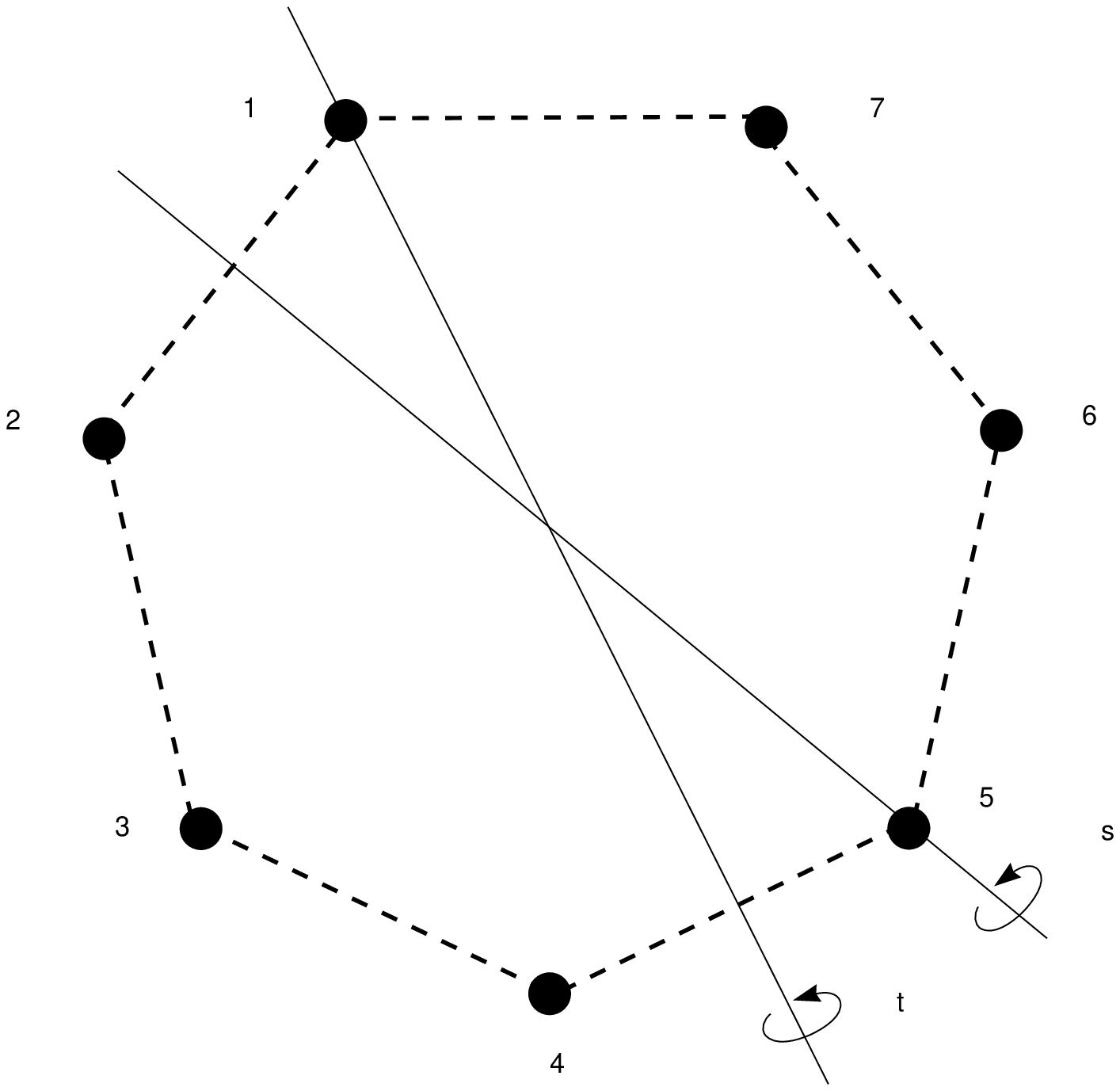}
\end{minipage}
\begin{minipage}[c]{0.13\textwidth}
\quad\includegraphics[width=30pt]{fluffyarrow}
\end{minipage}
\begin{minipage}[c]{0.4\textwidth}
\centering
\includegraphics[width=80pt]{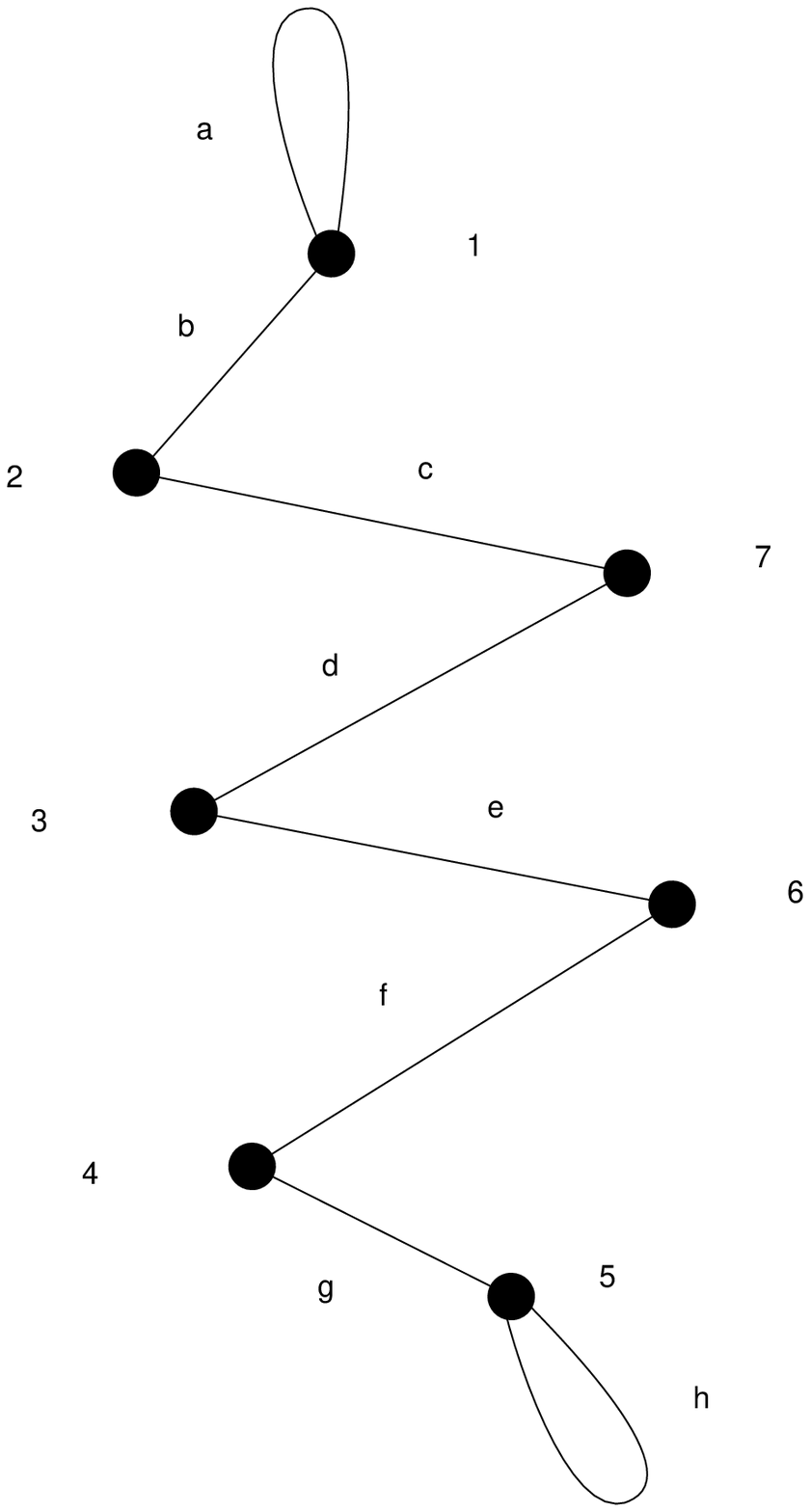}
\end{minipage}
\caption{\label{F:FIGH}The graph showing in which order the blocks
$X_i$ slot together to build $\covcyl$.}
\end{figure}

Now that we know what $\covcyl$  looks like combinatorially, we
describe its embedding in $\mathds{R}^3$ which we will use in the
presentation of the final result. For this purpose it is useful to
notice that the construction of $\covcyl$ defines a permutation
$\tau$ of $1,\ldots,n$, taking $i$ (representing $X_i$) to the
position of $X_i$ on the path from $1$ to $\frac{n+3}{2}$ (one plus
its distance on the graph from the vertex labeled $1$). Thus for
Figure \ref{F:FIGH}:
$$
\tau= \left(\begin{matrix}1 & 2 & 3 & 4 & 5 & 6 & 7\\
1 & 2 & 4 & 6 & 7 & 5 & 3
\end{matrix}\right)
$$

Now, for each $i=1,\ldots,n$, if $\tau(i)$ is even, bend the arms of
the dumbbell (Figure \ref{F:FIGD}) down and place the block $X_i$ in
$\mathds{R}^3$ the position shown in Figure \ref{F:FIGE}. If
$\tau(i)$ is even, bend the arms up and place the result in the
position shown in Figure \ref{F:FIGF}. The reader can observe that
the resulting identifications are exactly those determined by the
representation. Finish the construction by gluing up the four
remaining pairs of cuts--- the cuts next to each other around the
points $(-n-1,0,0)$, $(-1,0,0)$, $(1,0,0)$ and $(n+1,0,0)$.

\begin{figure}[h]
\centering \psfrag{T}[c]{$T$}
\psfrag{1}[c]{\Tiny$C\phantom{^\prime}$}
\psfrag{2}[c]{\Tiny$D\phantom{^\prime}$}
\psfrag{3}[c]{\Tiny$D^\prime$} \psfrag{4}[c]{\Tiny$C^\prime$}
\psfrag{5}[c]{\Tiny$E\phantom{^\prime}$}
\psfrag{6}[c]{\Tiny$F\phantom{^\prime}$}
\psfrag{7}[c]{\Tiny$F^\prime$} \psfrag{8}[c]{\Tiny$E^\prime$}
\psfrag{9}[c]{\Tiny$A^\prime$}\psfrag{0}[c]{\Tiny$B^\prime$}\psfrag{q}[c]{\Tiny$B\phantom{^\prime}$}
\psfrag{w}[c]{\Tiny$A\phantom{^\prime}$}\psfrag{e}[c]{\Tiny$G^\prime$}\psfrag{r}[c]{\Tiny$H^\prime$}
\psfrag{t}[c]{\Tiny$H\phantom{^\prime}$}\psfrag{y}[c]{\Tiny$G\phantom{^\prime}$}
\centering\psfrag{b}[c]{\Tiny$x_1=-\tau(i)-1$}\psfrag{n}[c]{\Tiny$x_1=-\tau(i)$}\psfrag{m}[c]{\Tiny$x_1=\tau(i)$}\psfrag{o}[c]{\Tiny$x_1=\tau(i)+1$}
\psfrag{z}[r]{\Tiny$x_2=-\frac{1}{2}$}\psfrag{x}[r]{\Tiny$x_2=\frac{1}{2}$}\psfrag{c}[r]{\Tiny$x_2=\tau(i)$}\psfrag{v}[r]{\Tiny$x_2=\tau(i)+1$}
\includegraphics[width=3.3in]{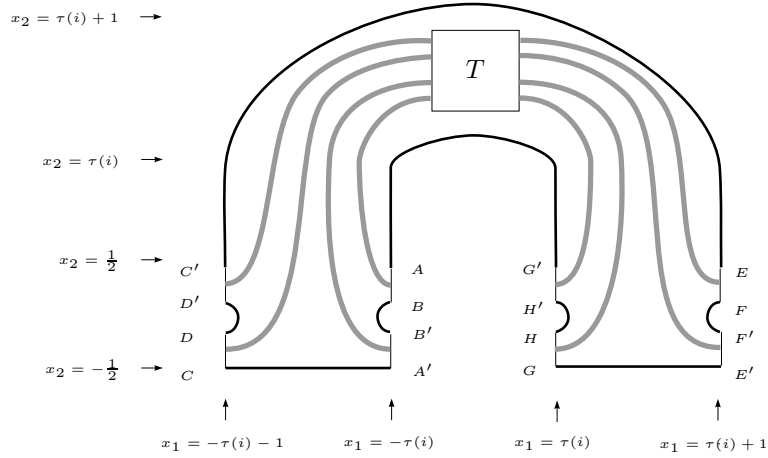}
\caption{$X_i$ with arms bent down}\label{F:FIGE}
\end{figure}

\begin{figure}[h]
\centering \psfrag{T}[c]{$T$}
\psfrag{1}[c]{\Tiny$C\phantom{^\prime}$}
\psfrag{2}[c]{\Tiny$D\phantom{^\prime}$}
\psfrag{3}[c]{\Tiny$D^\prime$} \psfrag{4}[c]{\Tiny$C^\prime$}
\psfrag{5}[c]{\Tiny$E\phantom{^\prime}$}
\psfrag{6}[c]{\Tiny$F\phantom{^\prime}$}
\psfrag{7}[c]{\Tiny$F^\prime$} \psfrag{8}[c]{\Tiny$E^\prime$}
\psfrag{9}[c]{\Tiny$A^\prime$}\psfrag{0}[c]{\Tiny$B^\prime$}\psfrag{q}[c]{\Tiny$B\phantom{^\prime}$}
\psfrag{w}[c]{\Tiny$A\phantom{^\prime}$}\psfrag{e}[c]{\Tiny$G^\prime$}\psfrag{r}[c]{\Tiny$H^\prime$}
\psfrag{t}[c]{\Tiny$H\phantom{^\prime}$}\psfrag{y}[c]{\Tiny$G\phantom{^\prime}$}
\centering\psfrag{b}[c]{\Tiny$x_1=-\tau(i)-1$}\psfrag{n}[c]{\Tiny$x_1=-\tau(i)$}\psfrag{m}[c]{\Tiny$x_1=\tau(i)$}\psfrag{o}[c]{\Tiny$x_1=\tau(i)+1$}
\psfrag{z}[r]{\Tiny$x_2=\frac{1}{2}$}\psfrag{x}[r]{\Tiny$x_2=-\frac{1}{2}$}\psfrag{c}[r]{\Tiny$x_2=-\tau(i)$}\psfrag{v}[r]{\Tiny$x_2=-\tau(i)-1$}
\includegraphics[width=3.3in]{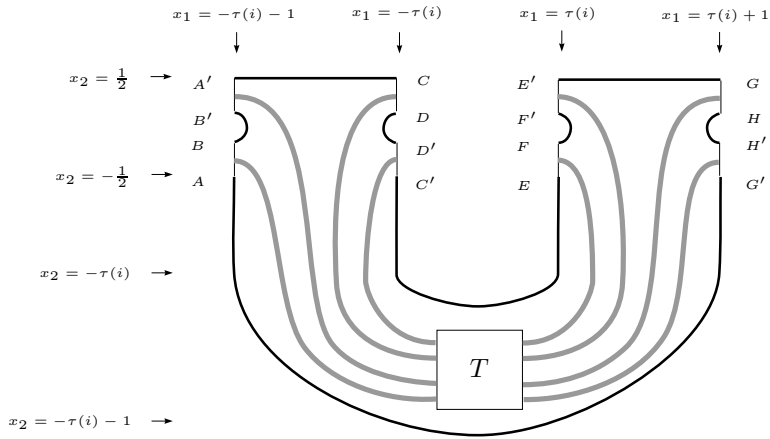}
\caption{$X_i$ with arms bent up}\label{F:FIGF}
\end{figure}

For example, the result for $n=3$ is displayed in Figure
\ref{F:FIGG}.

\begin{figure}[h]
\psfrag{T}[c]{$T$}
\includegraphics[width=300pt]{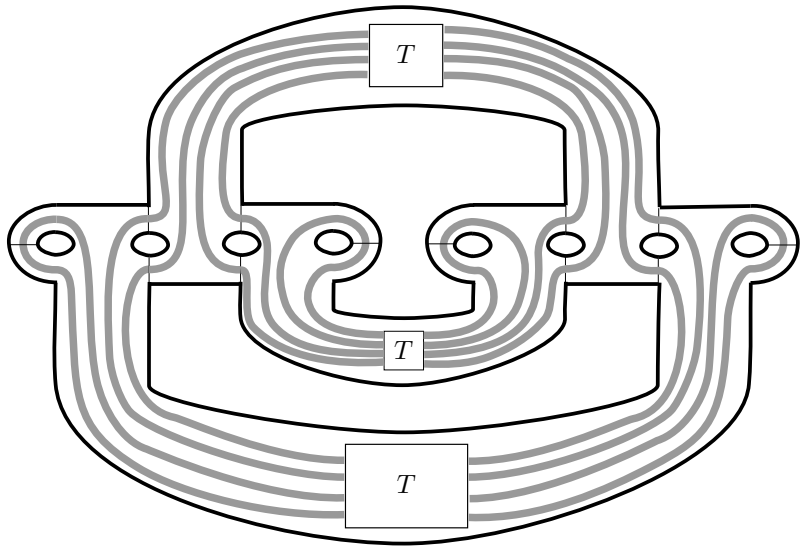}
\caption{\label{F:FIGG}$\widetilde{\Sigma\times[0,1]}$ embedded in
$\mathds{R}^3$ for $n=3$.}
\end{figure}

\subsection{The branching set}
We began this section with a framed link in $\Sigma\times[0,1]$ in
the complement of a coloured knot $\K_k$. Since $\K_k$ happens to be
a $2$--bridge knot, its dihedral covering space is $S^3$ with an
$\frac{n+1}{2}$--component covering link embedded in it (see
\textit{e.g.} \cite{Bir76}). It remains for us to describe this
link, and show how $\covcyl$ embeds into its complement.\par

To present the result we need to introduce some additional notation.
The result will use certain braids on $2(n+1)$ strands. The strands
of the braids will be indexed by the set
\[
I_n=\{-n-1\leq i\leq n+1\,,\, i\in \mathbb{Z}/\{0\}\}. \] The
coordinate $x_3$ will be the vertical coordinate of the braid, and
the projections of the endpoints of the strands to the
$(x_1,x_2)$--plane will be the points $\{(x,0),x\in I_n\}$.
 (Note that these are precisely the
coordinates of the `holes' in the construction we just gave of
$\covcyl$.)

Let $i<j$ be indices from $I_n$. Let $\Jh[i,j]$ denote the braid you
get by putting a clockwise half-twist into the group of strands
starting with the strand at position $i$, up to the strand at
position $j$. For example, if $n=4$, then $\Jh[-2,2]$ denotes the
braid shown in Figure \ref{F:FIGJ}.

\begin{figure}[h]
\psfrag{T}[c]{$T$}
\includegraphics[width=200pt]{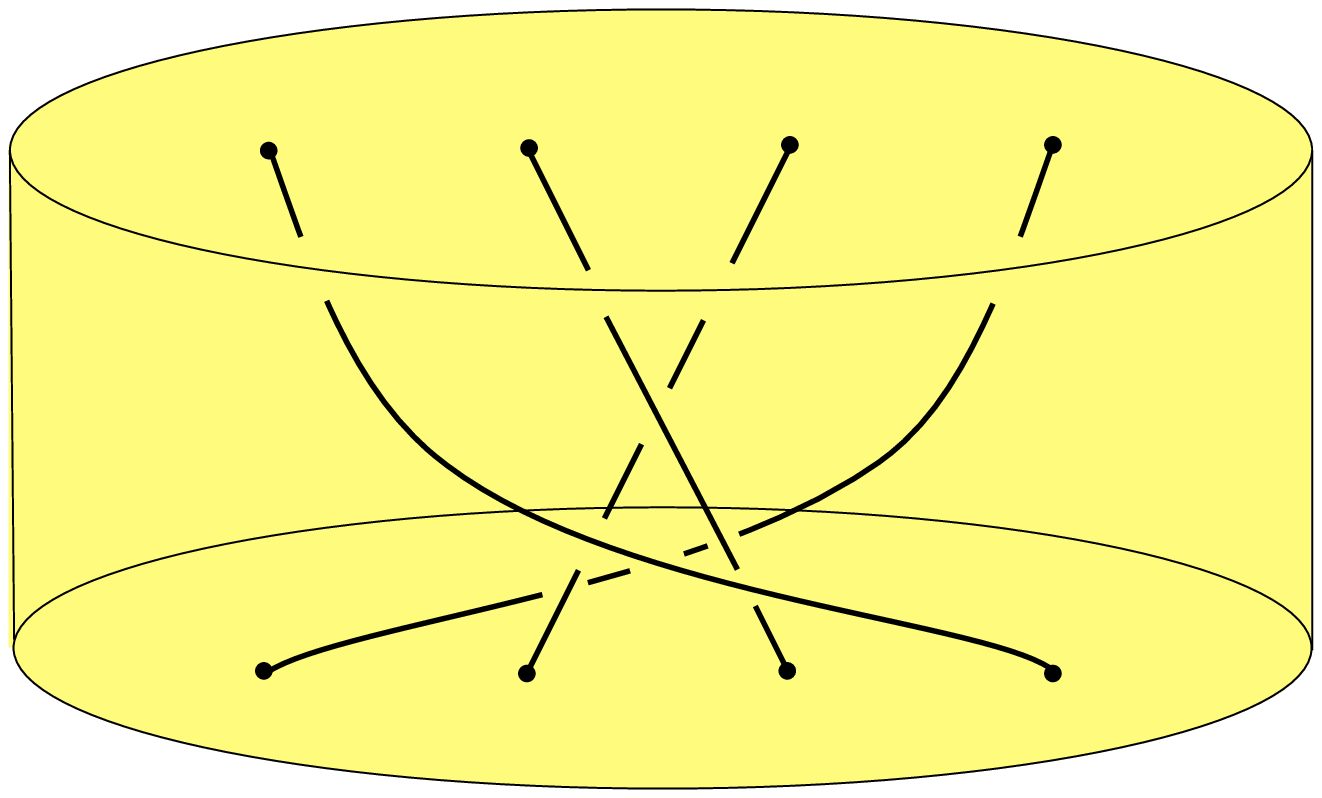}
\caption{\label{F:FIGJ} The half-twist $\Jh[-2,2]$.}
\end{figure}

We can now state the result.
\begin{thm}
Take the construction given earlier of $\covcyl$ as a subset of
$\mathds{R}^3$. The branching set over $\K_k$ lies in its complement
as shown in Figure \ref{F:FIGK}, where $B$ denotes the braid:
\[
\Jh[-n,n]\cdot\Jh[-n+1,n-1]\thinspace \cdots\thinspace
\Jh[-2,2]\cdot \Jh[-1,1].
\]
\end{thm}

\begin{figure}[h]
\psfrag{G}[c]{\scalebox{0.75}{$H[-p-1,-1]$}}\psfrag{H}[c]{\scalebox{0.75}{$H[-p-1,-1]^k$}}\psfrag{S}[c]{$\widetilde{\Sigma\times[0,1]}$}\psfrag{B}[c]{$B$}
\psfrag{l}[c]{$\cdots$}\psfrag{m}[c]{$\cdots$}\psfrag{n}[c]{$\cdots$}\psfrag{o}[c]{$\cdots$}\psfrag{p}[c]{$\cdots$}\psfrag{q}[c]{$\vdots$}\psfrag{r}[c]{$\vdots$}
\includegraphics[width=170pt]{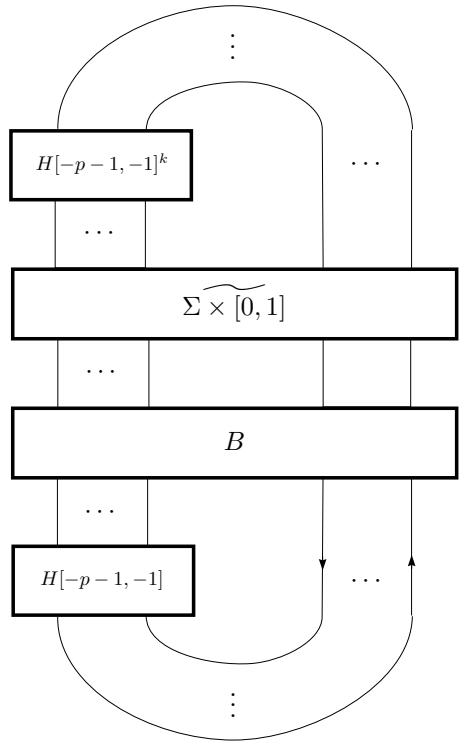}
\caption{\label{F:FIGK} The lifted picture.}
\end{figure}

\section{Odds and Ends}\label{S:oddsandends}

In this section we consider several corollaries to the constructions
given in the previous sections. In Section
\ref{SS:AdditionalBaseKnots} we list some different choices of
complete sets of base knots which we might end up with via the band
projection approach. In Section \ref{SS:Sakuma} we show how one of
these choices leads to a proof that closed $3$-manifolds with
$\Dn$-symmetry have surgery presentations with $\Dn$-symmetry.

\subsection{Different choices for a complete set of base-knots}\label{SS:AdditionalBaseKnots}

Our choice of $(\K_k,\rho_k)$ as a complete set of base-knots was
made because we have an explicit algorithm to reduce any
$\Dn$-coloured knot to one of them by surgery, and because in
addition we know how to explicitly find their branched dihedral
covering spaces, covering links, and the lifts of the surgery
presentations. This set was found by trial and error. Other complete
sets of base-knots are possible of course, and some of these have
advantages over $(\K_k,\rho_k)$.\par

Our starting point is a genus one knot with unknotted bands and with
the surface data given by Equation \ref{E:finalform}, repeated here
for the reader's convenience.

$$
(\M,\vec{v})=
\genusoneknot{kn}{\frac{n-1}{2}}{\frac{n+1}{2}}{\frac{1-n}{2}}{s}{1}
$$

\subsubsection{Linking number zero with the distinguished
component}\label{S:lknumzero} In this section we prove that we may
choose a \sds presentation such that the curves in $\ker\rho$ all
have linking number zero with the distinguished surgery component.
First perform the band slide we did in order to obtain
$(\K_k,\rho_k)$:

$$
(\M,\vec{v})=
\genusoneknot{kn}{\frac{n-1}{2}}{\frac{n+1}{2}}{\frac{1-n}{2}}{s}{1}\mapsto
\genusoneknot{kn+\frac{n+1}{2}}{0}{1}{\frac{1-n}{2}}{s}{s^{-1}}
$$

Perform $\frac{n+1}{2}$ additional surgeries between the bands:

$$
(\M,\vec{v})\mapsto
\genusoneknot{(k+1)n+1}{\frac{n+1}{2}}{\frac{n+3}{2}}{1}{s}{s^{-1}}
$$

Slide $B_2$ over $B_1$ repeatedly $\frac{n+1}{2}$ times:

\begin{figure}
\begin{minipage}{250pt}
\psfrag{L}[c]{\Large$T$}\psfrag{l}[c]{\Huge$\cdots$}\psfrag{k}[c]{\parbox{0.7in}{$k^\prime
n+m^\prime$\\half--twists}}
\psfrag{m}[c]{\rotatebox{270}{$\left\{\rule{0pt}{0.8in}\right.$}}\psfrag{s}[c]{$s$}
\psfrag{a}[c]{$ts$}\psfrag{b}[c]{$t$}
\includegraphics[width=250pt]{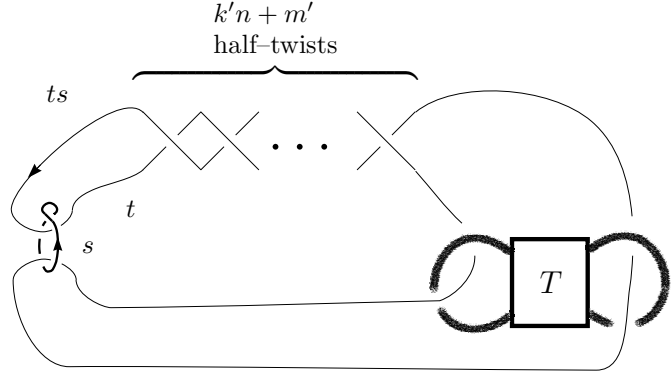}
\end{minipage}\rule{0pt}{1.1in}
\caption{\label{F:twistknot-1}Untying the twist knot.}
\end{figure}

$$
(\M,\vec{v})\mapsto
\genusoneknot{(k+1)n+m+1}{0}{1}{1}{s}{s^{-\frac{n+3}{2}}}
$$
\noindent where $k^\prime=0,\ldots,n-1$ and $m=
\frac{n+1}{2}-2\sum_{i=1}^{\frac{n+1}{2}}i$. If $\frac{n+1}{2}$ is
even, then $m=-\frac{(n+1)^2}{2}$, while if $\frac{n+1}{2}$ is odd
then $m=1-\frac{n^2+1}{2}$. This is the twist knot with $(k+1)n+m+1$
twists. Untie this knot by a single surgery as shown in Figure
\ref{F:twistknot-1}, where we redefine $k^\prime\ass k+1$ and
$m^\prime\ass m+1$. Put this into a \sds presentation by untying the
distinguished surgery component by surgery in $\ker\rho$. We obtain
a \sds presentation where all surgery components in $\ker\rho$ have
linking number zero not only with the knot, but also with the
distinguished surgery component.

\subsubsection{Torus knot presentation}\label{SSS:torusknotpresentation}

By constructing complete sets of base knots with cardinality $n$ in
previous sections, we proved Corollary \ref{C:RhoCorollary} which
states that two knots are $\rho$--equivalent if and only if they
have the same coloured untying invariant. As calculated in
\cite{Mos06b} (see also \cite{LiWal08}), the left-hand
$((2k+1)n,2)$--torus knots of Figure \ref{F:torusknot-2} are
examples of $\Dn$-coloured knots with coloured untying invariant
$k=1,\ldots,n$. Thus we have:

\begin{cor}\label{C:torusknotrepresentative}
The knots depicted in Figure \ref{F:torusknot-2} (the
$((2k+1)n,2)$--torus knots with the given colouring for
$k=1,\ldots,n$) comprise a complete set of base-knots for $\Dn$.
\end{cor}
\begin{figure}
\begin{minipage}{210pt}
\psfrag{l}[c]{\Huge$\cdots$}\psfrag{k}[c]{\parbox{0.7in}{\quad\ \ $(2k+1)n$\\[0.05cm]
half--twists\\}}
\psfrag{m}[c]{\rotatebox{90}{$\left\{\rule{0pt}{0.9in}\right.$}}\psfrag{s}{$ts$}\psfrag{t}{$t$}
\includegraphics[width=210pt]{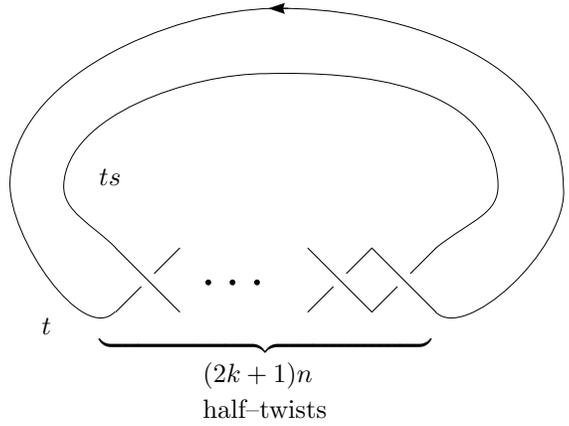}
\end{minipage}\rule{0pt}{1in}
\caption{\label{F:torusknot-2}Torus knots as base knots.}
\end{figure}

The surgery presentation of the branched dihedral covering and of
the covering link which this picture gives is:

$$
\begin{minipage}{150pt}
\psfrag{h}[c]{$B^k$}\psfrag{a}[c]{}
\includegraphics[width=150pt]{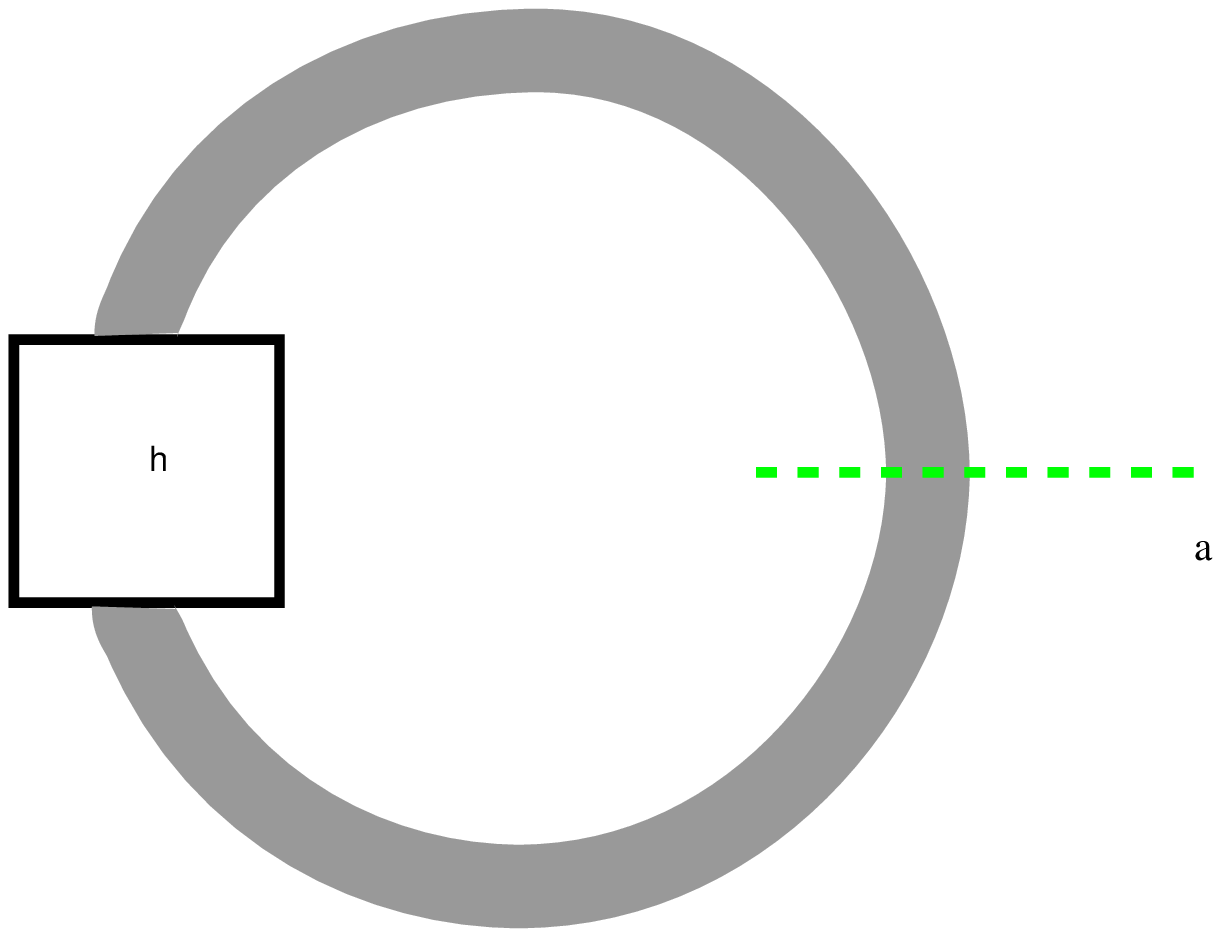}
\end{minipage}
$$
\noindent with the thick line denoting $n+1$ parallel strands and
with

$$
\begin{minipage}{200pt}
\psfrag{1}[c]{L}\psfrag{2}[c]{L}\psfrag{3}[c]{L}\psfrag{4}[c]{L}\psfrag{d}[c]{\Huge$\cdots$}
\includegraphics[width=200pt]{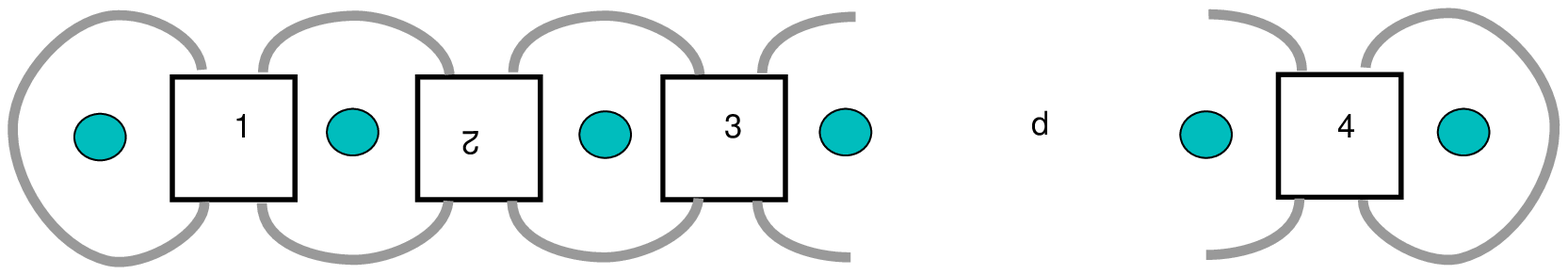}
\end{minipage}
$$
 \noindent being the lift of the covering link, slotted into the lift of the torus knot at the dotted line, where the strands of
 the covering link of the torus knot thread up out of the page
 through the holes indicated.

\subsubsection{One knot, different
colourings}\label{SSS:OneKnotManyColour} We can choose a complete
set of base-knots as a fixed knot $K$ whose colouring varies. Use
\ref{E:RR} to kill $a_{22}$ and then for $a_{11}=kn$ slide $B_2$
over $B_1$ counterclockwise repeatedly $k$ times. We obtain:
\begin{equation}
(\M,\vec{v})\mapsto
\genusoneknot{0}{\frac{n-1}{2}}{\frac{n+1}{2}}{0}{s^{\phantom{k}}}{s^k}
\end{equation}
\noindent Since the coloured Seifert matrix uniquely characterizes a
$\Dn$-coloured knot modulo $\rho$--equivalence, this gives a minimal
complete set of base-knots, as in Figure \ref{F:1knot8col}.

\subsection{Visualizing dihedral actions on
manifolds}\label{SS:Sakuma} The following section deals with an
observation due to Makoto Sakuma, that Corollary
\ref{C:torusknotrepresentative} implies a visualization theorem for
$\Dn$ actions on manifolds. We summarize his argument, essentially
contained in \cite{Sak01}.\par

Let $\Dn$ act on a closed oriented connected $3$--manifold $M$ via
orientation preserving diffeomorphisms $\mathbf{f}:=(f_{t},f_{s})$
where $f_{t}^{2}=f_{s}^{n}=1$, and $f_{t}f_{s}f_{t}=f_{s}^{n-1}$.
Actually the assumption that $f(t)$ and $f(s)$ are smooth may be
replaced by the weaker assumption that they be locally linear
\cite[Remark 2.3]{Sak01}. Viewing the $3$--sphere as a one point
compactification of $\mathds{R}^3$, the claim is then that $M$ has a
surgery presentation $L\subset S^3$ such that $L$ is invariant under
$\frac{2\pi}{n}$ rotation around the $Z$--axis and under $\pi$
rotation around the $X$--axis as a framed link.\par

The proof is by taking the quotient smooth orbifold $\Or\ass M/\Dn$
(see \textit{e.g.} \cite[Section 2.1]{CHK00}), with singular set
$\Sigma$. So $\mathrm{pr}\co M\twoheadrightarrow\Or$ is a $2n$--fold
regular dihedral covering space (see \textit{e.g.} \cite{Rol90})
with monodromy given by a representation $\psi\co
\pi_{1}(\Or-\Sigma)\twoheadrightarrow \Dn$ induced by the action of
$\mathbf{f}$. The idea is to construct a surgery link $\mathcal{L}$
to make the following diagram commute:

\qquad\qquad\qquad\qquad\qquad\qquad
\begin{equation}
\begin{CD}%
\ @. M  @<\mathrm{surg}(\tilde{\mathcal{L}})<< S^{3} @. \hspace{-13pt}\supset\tilde{\mathcal{L}}\\
@. @V\mathrm{pr}_{\psi}VV     @VV\mathrm{pr}_{\rho_\mathfrak{t}}V \\
\Sigma\subset\ @. \Or @<<\mathrm{surg}(\mathcal{L})< S^3 @. \ \ \
\supset \mathcal{L}\cup\mathfrak{t}((2k+1)n,2)
\end{CD}
\end{equation}

\noindent where $\mathrm{surg}(-)$ performs surgery by its argument
(note that this is not a map), and $\Sigma$ and
$\mathfrak{t}((2k+1)n,2)$ are the covering loci. The lifted link
$\tilde{\mathcal{L}}$ will then have the required dihedral symmetry
by construction, inherited from the dihedral symmetry of
$\mathfrak{t}((2k+1)n,2)$ lying symmetrically along a torus.\par

The link $\mathcal{L}$ is constructed as the combination of two
framed links $\mathcal{L}_1\cup\mathcal{L}_2$ such that
\begin{enumerate}
\item The sublink $\mathcal{L}_1$ is in $\ker\rho$, its components are $\pm1$--framed and are unknotted, and
$\mathrm{surg}(\mathcal{L}_1)\co S^3\Too S^3$ takes
$(\mathfrak{t}((2k+1)n,2),\rho_\mathfrak{t})$ to some $\Dn$-coloured
knot
$(K^\prime,\rho^\prime)$.%
\item For the sublink $\mathcal{L}_2$, the procedure $\mathrm{surg}(\mathcal{L}_2)\co
S^3\Too \Or$ takes $(K^\prime,\rho^\prime)$ to $(\Sigma,\psi)$.
\end{enumerate}

The sublink $\mathcal{L}_1$ is given to us by Corollary
\ref{C:torusknotrepresentative}, while $\mathcal{L}_2$ may be
constructed in complete analogy with \cite[Pages 383--384 and
Section 4]{Sak01}.

\bibliographystyle{amsplain}

\end{document}